\documentclass[10pt]{article}
\usepackage{amsfonts}
\usepackage{amsfonts,latexsym}
\usepackage{amsmath}
\usepackage{amsthm}
\usepackage{amssymb}
\usepackage{mathrsfs}
\usepackage{color}
\usepackage[]{cite}
\usepackage{dsfont}
\usepackage{graphicx}
\usepackage[toc]{appendix}
\usepackage{mathtools}
\allowdisplaybreaks[4]
\usepackage{indentfirst}
\usepackage{geometry}
\newtheorem{theorem}{Theorem}[section]

\newtheorem{definition}[theorem]{Definition}
\newtheorem{Assumption}[theorem]{Hypothesis}
\newtheorem{lemma}[theorem]{Lemma}
\newtheorem{remark}[theorem]{Remark}
\newtheorem{corollary}[theorem]{Corollary}

\begin{document}
\title{\bf Long-time dynamics of stochastic 2D hydrodynamic-type evolution equations driven by multiplicative L\'{e}vy noise \footnote{The research is supported by National Natural Science Foundation of China (Grant No. 12371198)}}

%\author{{Jiangwei Zhang$^\text{a}$,  \,\,  Daiwen Huang$^\text{a}\footnote{Corresponding author.}$
\author{{Jiangwei Zhang$^\text{a}$
	}\\
	{ \small\textsl{$^\text{a}$ Institute of Applied Physics and Computational Mathematics, }}\\
	{ \small \textsl{Beijing, 100088,  P.R. China}}
}
\footnotetext{
	\emph{E-mail addresses}: zjwmath@163.com (J. Zhang). %huang\_daiwen@iapcm.ac.cn (D. Huang).
}
\date{}

%%%%%%%%%%%%%%%%%%%%%%%%

\renewcommand{\theequation}{\arabic{section}.\arabic{equation}}
\numberwithin{equation}{section}

\maketitle

%%%%%%%%%%%%%%%%%%%%%
\begin{abstract} This paper investigates the long-time dynamics of solutions for an abstract nonlinear stochastic hydrodynamic-type equation driven by multiplicative L\'{e}vy noise. The framework encompasses several key hydrodynamical models, including the stochastic 2D Navier-Stokes equations, magnetohydrodynamic equations, the magnetic B\'{e}rnard problem, as well as various stochastic shell models of turbulence.
Under the assumption that the nonlinear noise coefficients satisfy local Lipschitz and linear growth conditions, we first establish global well-posedness using a truncation technique.
Then, by introducing a mean random dynamical system, we prove the existence and uniqueness of weak pullback mean random attractors for the system. 
Furthermore, when the external force is time-independent, we study the existence of invariant measures for the corresponding autonomous system, as well as the double limiting behavior of invariant measures with respect to the intensities of Gaussian and L\'{e}vy noise. 
Finally, under additional assumptions on the bilinear nonlinear term (e.g., as in the Navier-Stokes equations), we examine the existence and uniqueness of pullback measure attractors, along with the asymptotically autonomous stability of such attractors as the time parameter tends to negative infinity. It is worth noting that the results of this paper are new even
for the single stochastic 2D Navier-Stokes equations.

	\medskip
	
\end{abstract}

~ \textbf{AMS subject classifications}:  {Primary 37L55; Secondary 35B40, 35B41, 60H15.}

	\tableofcontents

\section{Introduction}
In this paper, we are interested in the long-term dynamics of the following abstract non-autonomous stochastic evolution driven by multiplicative L\'evy noise:
\begin{align}\label{2DHD-1.1}
	\left\{
\begin{aligned}
		& d u(t)+\mu \mathcal{A} u(t)d t+B(u(t),u(t))d t
		 = f(t)dt+ \varepsilon_1 h(t,u(t))d W(t) +\varepsilon_2\int_{\mathcal{Z}}G(u(t-),z)\widetilde{N}(dt,dz),
		\\
		& u(\tau)=u_{\tau}.
	\end{aligned}
\right.
\end{align}
Above, $\mu>0$ is a constant, the initial value $u_{\tau}$ belongs to the separable Hilbert space $H$, $\mathcal{A}$ is a positive self-adjoint linear operator with dense domain in $H$,  $B$ is a bilinear map defined on dense subset of $H$, $f(t)$ is a time-dependent forcing term, the small parameters $\varepsilon_1, \varepsilon_2\in (0,1]$ denote the intensities of noises. $W$ is a two-sided  cylindrical Wiener process on some Hilbert space $U$,   
and $\widetilde{N}$ is a compensated Poisson random measure, see Section \ref{2DHD-Sec2.1} below for more precise descriptions. Assume that $\widetilde{N}$ and $W$ are independent. For the linear operator $A$, the bilinear map $B$, and the measurable coefficients $h$ and $G$ governing the stochastic perturbations, specific conditions will be stipulated later. 

%In contrast, the dynamical behavior under long-time uniform estimates of solutions has remained largely unexplored.
%In contrast to SPDEs driven by Gaussian noise, it is readily observed that establishing the well-posedness and limiting dynamical behavior of SPDEs with L\'{e}vy noise typically requires the following two assumptions in the literature:
%To establish the well-posedness of SPDEs with L\'{e}vy noise, standard approaches in the literature typically require the following two assumptions:
As an abstract evolution equation incorporating a wide variety of canonical fluid models, equation \eqref{2DHD-1.1} is of great theoretical significance for the study of its well-posedness and limiting dynamical behavior.
 In recent years, extensive progress has been made on stochastic fluid equations driven by Gaussian noise and L\'{e}vy noise (see, e.g., \cite{Li-Ro-2010-JFA,Flandoli-1995-PTRF,Brze2017,Barbu-AMO-2007,Bre-NA-2013,Chueshov-PRSLS-2005,Constantin-PD-2006,Ferrario-NoDEA-1997,Katz-TAMS-2005,Chueshov-AMO-2010,Bessaih-NoDEA-2014} and the references therein).
In particular, for stochastic partial differential equations (SPDEs) driven by L\'{e}vy noise, existing research has largely focused on well-posedness, invariant measures, ergodicity, mixing, and large deviation principles. 
Most of these studies rely on two fundamental assumptions, which are typically required in the literature to establish the well-posedness of SPDEs with L\'{e}vy noise. In contrast to the Gaussian noise case, these assumptions are as follows:
\begin{itemize}
	\item[$\bullet$] For any $k>0$,	$\sup_{\|u\|_{H}\leq k} \int_{\|z\|_{\mathcal{Z}}\leq \varepsilon} \|G(t,u,z)\|_{H}^2\nu(dz)\rightarrow 0, \text{ as } \varepsilon \rightarrow 0$, see \cite{Dong-SCM-2009}; or
	\item[$\bullet$] For some $p>2$, $\int_{\mathcal{Z}}\|G(t,u,z)\|_{H}^p\nu(dz)\leq C(1+\|u\|_{H}^p)$, see \cite{Bre-NA-2013,BreLZ-NA-2014}.
\end{itemize}

A natural question is
 whether the two assumptions mentioned above can be removed in order to prove global well-posedness of SPDEs with L\'{e}vy noise. In an effort to resolve this, Brze\'{z}niak et al. \cite{Bre-JEMS-2023}, by introducing a novel cutting-off argument, proved the existence and uniqueness of
global strong solutions, both probabilistically and in the sense of PDEs, for the 2D Navier-Stokes equations with L\'{e}vy noise, and subsequently established a Freidlin-Wentzell type large deviation principle. In addition, building on the innovative approach developed in \cite{Bre-JEMS-2023}, Peng et al. \cite{Peng-EJP-2022} proved the global well-posedness of problem \eqref{2DHD-1.1} under a natural Lipschitz condition on balls and linear growth assumptions on the jump coefficient (i.e., Hypothesis \ref{2DHD-Ass2.3}). Their proof combines a similar cutting-off argument with new a priori estimates, a slightly modified localization argument, and the fixed point theorem.
Extending the analysis of \cite{Peng-EJP-2022}, Shang et al. established a moderate deviation principle for problem \eqref{2DHD-1.1} with $\varepsilon_{1}=0$ via a weak convergence method, removing the compact embedding assumption on the associated Gelfand triple. Apart from these works, to the best of our knowledge, no further results are available on other dynamical behaviors (such as attractors and invariant measures) of solutions to problem \eqref{2DHD-1.1}.
% (such as weak mean attractors, invariant measures, and pullback measure attractors, among others), which motivates the research presented in this paper

It is known that the theory of pathwise random attractors \cite{Caraballo-2009} is inadequate for SPDEs driven by nonlinear noise. To overcome this limitation, Wang et al. put forward two effective approaches: $(i)$ the construction of weak mean random attractors \cite{WBX2019} for mean random dynamical systems in spaces of Bochner integrable functions, which builds upon a generalization of the results \cite{CKloeden2012}; and $(ii)$ the construction of pullback measure attractors for non-autonomous dynamical systems on the space of measures \cite{LDS-JDE-2024}, by means of an extension of autonomous measure attractors \cite{schmalfuss1991,schmalfuss1999} to the non-autonomous setting. 
In recent years, numerous scholars have investigated the weak mean random attractors and measure attractors of stochastic evolution equations driven by nonlinear noise based on the aforementioned methods, see e.g., \cite{WBX2019a,WBX2019b,ZJW-SD-2023,CYZ-SAA-2022,ZCH-JMAA-2026,LZH-AML,BCS-JDE-2025,LYR-QTDS-2025} and the references therein. However, to date, there have been no relevant studies on the associated dynamical properties of the abstract nonlinear stochastic hydrodynamic-type equation \eqref{2DHD-1.1} driven by nonlinear multiplicative L\'evy noise.

In view of the above, despite the progress in well-posedness and large deviations, and the development of new dynamical theories for nonlinear noise, there have been no relevant studies on the long-term dynamical properties of the abstract nonlinear stochastic hydrodynamic-type equation \eqref{2DHD-1.1} driven by nonlinear multiplicative L\'{e}vy noise. Such properties include weak pullback mean random attractors, invariant measures, and pullback measure attractors. This paper aims to fill this gap by focusing on the limiting dynamical behavior of solutions to the stochastic abstract evolution equation \eqref{2DHD-1.1} driven by multiplicative L\'{e}vy noise.

Nevertheless, the present analysis faces several nontrivial challenges: $(i)$ extending the global Lipschitz conditions to local Lipschitz ones for the well-posedness of \eqref{2DHD-1.1}; $(ii)$ verifying the Feller property of the associated transition semigroup; and $(iii)$ deriving uniform pullback high-regularity estimates for solutions to establish pullback measure attractors.
The aforementioned difficulties are resolved as follows. First, the global well-posedness of \eqref{2DHD-1.1} is proved using a truncation argument that approximates locally Lipschitz functions by globally Lipschitz ones. Second, by virtue of an appropriate stopping time technique and the pathwise Gronwall lemma, the convergence of the solution sequence with respect to the initial data is obtained. The Feller property of the transition semigroup is then established through the convergence in probability of solutions. Finally, by imposing additional conditions on the bilinear operator $B$ and the nonlinear noise coefficients $h$ and $G$, we establish uniform pullback high-regularity estimates in the weighted sense for solutions of \eqref{2DHD-1.1} via a non-standard technique involving the introduction of a suitable weight function.
The principal contributions of the present work can be summarized as follows:

$(1)$ The well-posedness of problem \eqref{2DHD-1.1} is proved under local Lipschitz conditions. This allows us to define a mean dynamical system, for which the existence and uniqueness of a weak $\mathfrak{D}$-pullback mean random attractor is established.

$(2)$ We prove the existence of invariant measures for problem \eqref{2DHD-1.1} in the autonomous case and investigate the limiting behavior of invariant measures with respect to the noise intensity.

$(3)$ Under additional conditions on $B$, $h$ and $G$, the existence and uniqueness of pullback measure attractors for problem \eqref{2DHD-1.1} is established, and the asymptotic autonomy robustness of pullback measure attractors with respect to the time perturbation parameter is further investigated.

The rest of paper is arranged as follows. In the next section, we introduce some necessary notations and the basic assumptions that will be used throughout this paper.
In Section \ref{2DHD-EUS}, we establish the global existence and uniqueness of solutions under local Lipschitz conditions.
In Section \ref{WEEdfikjik}, we prove
the existence and uniqueness of weak pullback mean random attractors.  In Section \ref{2DHD-inv5}, we first prove the existence of invariant measures in the autonomous case, and then investigate the limiting behavior of invariant measures with respect to noise intensity. Finally, Section \ref{2DHD-PMAs6} demonstrates the existence and uniqueness of pullback measure attractors and establishes their asymptotic autonomy robustness.

\section{Preliminaries}
To begin, we establish the analytical framework for this study. Let $H$ be a separable Hilbert space with the inner product $\langle \cdot, \cdot\rangle $ and norm $\|\cdot\|_{H}$, and let $H'$ denote its dual space. Set $V=\rm{dom} (\mathcal{A} ^{1/2})$ equipped with norm  $\|y\|_{V}:=|\mathcal{A} ^{\frac{1}{2}}y|$, $y\in V$, such that $V\subset H$ continuously and densely. Denote  by $V'$ the dual space of $V$.
 Identifying $H$ and $H'$ via the Riesz isomorphism  we have that $V\subset H \cong H' \subset V'$ continuously and densely. In addition, the Poincar\'{e} inequality holds: $\lambda_1\|\cdot\|_{H}^2\leq \|\cdot\|_{V}^2$.
 With a slight abuse of notation, the duality pairing between $V$ and $V'$ is also written as $\langle \cdot, \cdot\rangle $. For simplicity, we set $\mathbb{R}^+=[0,\infty)$ and $\mathbb{R}^{\tau}=[\tau,\infty)$.

\subsection{Stochastic setting and notations}\label{2DHD-Sec2.1}
Let $ (\Omega,\mathscr{F} ,\{\mathscr{F}_t \}_{t\in \mathbb{R}},\mathbb{P}) $  be a filtered probability space, where $\{\mathscr{F}_t \}_{t\in \mathbb{R}}$ is a filtration satisfying the usual conditions (i.e., it is right continuous family of sub-$\sigma$-algebras of $\mathscr{F}$ that contains all $\mathbb{P}$-null sets). Let $\{W(t), t\geq \tau\}$ be a $H$-valued cylindrical Wiener process with the form $W(t,\omega)=\sum_{k\geq1}r_kw_k(t,\omega)$, where $\{r_k\}_{k\geq 1}$ is a complete orthonormal basis of a separable Hilbert space $U$,
$\{w_k\}_{k\geq1}$ is a sequence of independent standard Brownian motions on $(\Omega,\mathscr{F},\{\mathscr{F}_t\}_{t\in \mathbb{R}},\mathbb{P})$.
Consider the collection of all strongly measurable, square-integrable $H$-valued random variables, we denote by $L^2(\Omega; H)$ a Banach space endowed with the norm $\|u(\cdot)\|_{L^2(\Omega;H)}=\sqrt{\mathbb{E}\left[|u(\cdot,\omega)|^2\right]}$, where $\mathbb{E}$ denote the mathematical expectation. Let $\mathcal{L}_2:=\mathcal{L}_2(U;H)$ denote the space of all Hilbert-Schmidt operators from $U$ into $H$. The space $\mathcal{L}_2$ is a separable Hilbert space, equipped with the norm $\|\cdot\|_{\mathcal{L}_2(U;H)}$.

%For simplicity, we denote $\mathcal{B}(\mathcal{Z})$ by $\mathscr{Z}$ throughout the paper.
Let $(\mathcal{Z},\mathscr{Z}=\mathcal{B}(\mathcal{Z}))$ be a measurable space and $\nu$ be a $\sigma$-finite positive measure on it, $\nu(\{0\})=0$,
here $\mathscr{Z}$ denotes the Borel $\sigma$-field of $\mathcal{Z}$.  Let $N: \Omega \times\mathcal{B}(\mathbb{R}^+)\times \mathscr{Z} \rightarrow \mathbb{N}\cup \{\infty\}$ be a time-homogeneous Poisson random measure with intensity measure $\nu$ on $(\mathcal{Z},\mathscr{Z})$ over the filtered probability space $(\Omega,\mathscr{F},\{\mathscr{F}_t\}_{t\in \mathbb{R}},\mathbb{P})$, here $\mathbb{N}:=\{1,2,\cdots\}$. We denote by
$\widetilde{N}(dt,dz)={N}(dt,dz)-\nu(dz)dt$ 
the compensated Poisson random measures associated with $N$.

We recall the following Burkholder-Davis-Gundy (BDG) inequality in the sense of jump noise, see \cite{Applebaum-2009} for more details.
\begin{lemma}\label{BDSL}
	For any $p \geq 1$ and $\tau\in \mathbb{R}$, there exists a positive constant $C_p$ such that for any real-valued square-integrable c\`{a}dl\`{a}g martingale $M$ with $M(0)=0$, the following inequality holds:
	$$
	C_{p}^{-1} \mathbb{E}[M]_{T}^{\frac{p}{2}} \leq \mathbb{E}\left[\sup_{\tau \leq t \leq \tau+T}|M(t)|^{p}\right] \leq C_{p} \mathbb{E}[M]_{T}^{\frac{p}{2}},\quad \forall \, T > 0,
	$$
	where $[M]_{T}$ is called the quadratic variation of $M$.
\end{lemma}

Let $I$ be a bounded subset of $\mathbb{R}^{\tau}$, suppose that $\mathcal{D}(I;H)$ is the space of all c\`{a}dl\`{a}g functions $y: I\rightarrow H$, and it is endowed
with the Skorokhod topology. Then $\mathcal{D}(I;H)$ is metrizable and separable via a complete metric, see \cite[Chapter 2]{Mtivier-1988} for more details.   
Throughout the paper, we denote by $C_{c_1,c_2,\ldots}$ a generic positive constant that may change from line to line but depends only on $c_1,c_2,\cdots$.

\subsection{Hypotheses}
To address the objectives of this paper, this subsection presents the hypotheses regarding the operators $A$, $B$, and the coefficients of the nonlinear stochastic perturbations $h$ and $G$. Let $\mathcal{A}$ be a (possibly unbounded) linear operator with domain $D(\mathcal{A})$, where $D(\mathcal{A})$ is endowed with the graph norm and takes values in $H$. The operator $A$ is subject to the following conditions.

\begin{Assumption}\label{2DHD-Ass2.0}
 Assume that $\mathcal{A}$ is a  positive self-adjoint operator and its domain
 $D(\mathcal{A})$ is densely and compactly embedded in $H$.
\end{Assumption}

Let $B:V\times V\rightarrow V'$ be a bilinear map satisfying the following conditions.
\begin{Assumption}\label{2DHD-Ass2.1}
		Assume that $B:V\times V\rightarrow V'$   is a continuous bilinear mapping such that
	\begin{description}
		\item[(B.1)]{\rm (Skewsymmetricity of $B$)}
		\begin{eqnarray}\label{2DHD-eq2.1}
			\langle B(u,v),w\rangle=-\langle B(u,w),v\rangle, \text{ for all } u,v,w\in V.
		\end{eqnarray}		
		\item[(B.2)] There exist a reflexive and separable Banach space $(Q, \|\cdot\|_Q)$ and a positive constant $\textbf{c}_0$ such that
		\begin{eqnarray}\label{2DHD-eq2.2}
		V\subset Q \subset H
			\quad \text{~~and~~} \quad \|u\|_Q^2\leq \textbf{c}_0\|u\|_{H} \|u\|_{V}, \text{  for all } u \in V.
		\end{eqnarray}
		\item[(B.3)] There exists a constant $C>0$ such that
		\begin{eqnarray}\label{2DHD-eq2.3}
			|\langle B(u,v), w\rangle|\leq C\|u\|_Q \|v\|_{V} \|w\|_Q,
			\text{  for all } u,v,w\in V.
		\end{eqnarray}
	\end{description}
\end{Assumption}

\begin{remark}\label{Remweee2.3}
The hypotheses $\mathbf{(B.1)}$  implies that $\langle B(u,v),v\rangle=0$ for all  $u,v\in V$. Moreover, by $\mathbf{(B.2)}$ and $\mathbf{(B.3)}$ we can deduce
$$
|\langle B(u,v), w\rangle|\leq C_{\textbf{c}_0}\|u\|_{H}^{1/2} \|u\|_{V}^{1/2} \|v\|_{V} \|w\|_{H}^{1/2} \|w\|_{V}^{1/2}, \text{  for all } u,v,w\in V.
$$ 
\end{remark}

\begin{Assumption}\label{2DHD-Ass2.3}
	Assume that $h: \mathbb{R} \times H\rightarrow \mathcal{L}_2({U}; H)$ and $G: \mathbb{R} \times H \times \mathcal{Z}\rightarrow H$ are Borel measurable maps
	and satisfy the global Lipschitz and linear growth conditions. More precisely,
	\begin{description}

			\item[(C.1)]{\rm(Global Lipschitz)} There exists  a positive constant
			 $ L_1$ such that for all $t\in \mathbb{R}$ and $u, v\in V$,
			$$
			\|h(t,u)-h(t,v)\|_{\mathcal{L}_2(U;H)}^2
			+
			\int_\mathcal{Z}\|G(u,z)-G(v,z)\|_{H}^2\nu(dz)\leq L_1\|u-v\|_{H}^2.
			$$
		
		\item[(C.2)]{\rm(Growth)} There exists a positive constant $L_{g}>0$ such that for all $t\in \mathbb{R}$ and $u\in H$,		
		$$
		\|h(t,u)\|^2_{\mathcal{L}_2(U;H)}
		+
		\int_{\mathcal{Z}}\|G(u,z)\|_{H}^2\nu(dz)\leq L_{g}\left(1+\|u\|_{H}^2\right).
		$$
\end{description}
\end{Assumption}

\begin{Assumption}\label{2DHD-Ass2.4}
Assume that $h: \mathbb{R} \times H\rightarrow \mathcal{L}_2({U}; H)$ and 	$G:\mathbb{R}\times H\times \mathcal{Z}\rightarrow H$ are measurable mappings such that
	\begin{description}
				\item[(C.3)]{\rm(Local Lipschitz)} For every $r>0$, there exists  a positive constant
		$L_{r}$ such that, for all $t\in \mathbb{R}$, $u,v\in H$ with $\|u\|_{H}\vee \|v\|_{H}\leq r$,
		$$
		\|h(t,u)-h(t,v)\|_{\mathcal{L}_2(U;H)}^2
		+
		\int_\mathcal{Z}\|G(u,z)-G(v,z)\|_{H}^2\nu(dz)\leq L_{r}\|u-v\|_{H}^2,
		$$
			and $h$ and $G$ satisfy the hypothesis $\mathbf{(C.2)}$, that is the linear growth condition.  
	\end{description}
\end{Assumption}

\begin{remark}
	It should be noted that the linear growth condition $\mathbf{(C.2)}$ can be derived from the global Lipschitz condition $\mathbf{(C.1)}$, provided that for given $t\in \mathbb{R}$, 
	$$
	\|h(t,0)\|^2_{\mathcal{L}_2(U;H)}+\int_\mathcal{Z}\|G(0,z)\|_{H}^2\nu(dz)<\infty.
	$$
	In this case, it suffices to take $L_g=2\max\left\{L_1,\|h(t,0)\|^2_{\mathcal{L}_2(U;H)},\int_\mathcal{Z}\|G(0,z)\|_{H}^2\nu(dz)\right\}$.
\end{remark}

\begin{remark}\label{rem*LtimeB2.7}
	The abstract stochastic evolution equation \eqref{2DHD-1.1}, which satisfies the hypotheses imposed on the operators $\mathcal{A}$ and $B$ (Hypothesis \ref{2DHD-Ass2.1}), encompasses a broad class of stochastic fluid dynamics models, including the stochastic 2D Navier-Stokes equations, 2D stochastic magnetohydrodynamic (MHD) equations, 2D stochastic magnetic B\'{e}rnard problem, stochastic Leray-$\alpha$ model for Navier-Stokes equations, and several stochastic shell models of turbulence. For a detailed verification, the reader is referred to \cite[Section $2$]{Fernando-CMP-2016}, while specific details of the models can be found in  \cite{Bre-NA-2013,Constantin-1988,Chueshov-AMO-2010,Chueshov-PRSLS-2005,Barbu-AMO-2007,Constantin-PD-2006,Ferrario-NoDEA-1997,Katz-TAMS-2005,Chueshov-AMO-2010}.
\end{remark}

\section{Existence and uniqueness of solutions}\label{2DHD-EUS}

We will prove the existence and uniqueness of solutions to \eqref{2DHD-1.1} in the sense of the following definition of a solution.
\begin{definition}\label{2DHD-Def3.1}
	An $H$-valued c\`{a}dl\`{a}g $\mathscr{F}_t$-adapted process $\{u(t),t\in [\tau,\tau+T]\}$ is called a solution of \eqref{2DHD-1.1} if for all $\mathscr{F}_\tau$-measurable $u_{\tau}\in H$ and any $T>0$,
	
	\begin{description}
		\item{$(i)$}
	  $u\in \mathcal{D}([\tau,\tau+T],H)\cap L^2([\tau,\tau+T],V)$, $\mathbb{P}$-a.s.,

		\item{$(ii)$} the following  equality holds for every $t\in [\tau,\tau+T]$, as an element of $V'$, $\mathbb{P}$-a.s., 
		\begin{eqnarray*}
			u(t) &=&u_{\tau} -\mu\int_{\tau}^t \mathcal{A}u(s) ds - \int_{\tau}^t B(u(s),u(s)) ds+\int_{\tau}^t f(s) ds
			\\  &&+\varepsilon_1\int_{\tau}^t h(s,u(s))d W(s)+\varepsilon_2\int_{\tau}^t\int_{\mathcal{Z}}  G(u(s-),z) \widetilde{N}(ds,dz). 
		\end{eqnarray*}
		\end{description}
		An alternative version of Condition $(ii)$ stipulates that for every $t\in [\tau,\tau+T]$, $\mathbb{P}$-a.s.,
		\begin{align*}
			\langle u(t),\eta\rangle =&\langle u_{\tau},\eta\rangle -\mu\int_{\tau}^t\langle\mathcal{A}u(s), \eta\rangle ds - \int_{\tau}^t\langle B(u(s),u(s)),\eta\rangle ds+\int_{\tau}^t \langle f(s),\eta\rangle ds
			\\  
			&+\varepsilon_1\int_{\tau}^t\langle h(s,u(s))d W(s),
			\eta \rangle+\varepsilon_2\int_{\tau}^t\int_{\mathcal{Z}} \langle G(u(s-),z), \eta \rangle\widetilde{N}(ds,dz),  
		\end{align*}
		for any $\eta\in V$.
\end{definition}

We state the global existence of solutions under global Lipschitz conditions on the nonlinear coefficients of the noises, a result established in \cite{Peng-EJP-2022}.
\begin{theorem}\label{2DHD-the3.2}
	Let {Hypotheses} $\ref{2DHD-Ass2.0}$, $\ref{2DHD-Ass2.1}$ and $\ref{2DHD-Ass2.3}$ hold. 
	Then, for any $\mathscr{F}_{\tau}$-measurable $H$-valued initial data $u_{\tau}$ satisfying $\mathbb{ E} [\|u_{\tau}\|_{H}^2]<\infty$ and $f\in L^2([\tau,\tau+T],V')$, there exists a $\varepsilon_0>0$ such that for any $\varepsilon_1, \varepsilon_2\in (0,\varepsilon_0]$,
	the problem \eqref{2DHD-1.1} has a unique solution  $\{u(t),t\in [\tau,\tau+T]\}$ in the sense of Definition \ref{2DHD-Def3.1}. Moreover, the solution $u$ fulfills the following estimate, for any $T>0$,
	\begin{eqnarray*}
		\sup_{t\in [\tau, \tau+T]}\mathbb{ E}[\|u(t)\|_{H}^2]+\mathbb{ E}\left[\int_{\tau}^{\tau+T} \|u(t)\|_{V}^2d t\right]< \infty.
	\end{eqnarray*}
\end{theorem}

We now present the global existence theorem for solutions under the relaxed {Hypotheses} \ref{2DHD-Ass2.3}, in the sense of Definition \ref{2DHD-Def3.1}. In other words, we formulate the main result within this relaxed framework.
\begin{theorem}\label{2DHD-the3.3}
	Let {Hypotheses} $\ref{2DHD-Ass2.0}$, $\ref{2DHD-Ass2.1}$ and $\ref{2DHD-Ass2.4}$ hold.
	Then, for any $\mathscr{F}_{\tau}$-measurable $H$-valued initial data $u_{\tau}$ satisfying $\mathbb{ E} [\|u_{\tau}\|_{H}^2]<\infty$ and $f\in L^2([\tau,\tau+T],V')$, there exists a $\varepsilon_0>0$ such that for any $\varepsilon_1, \varepsilon_2\in (0,\varepsilon_0]\subseteq (0,1]$,
	there is a unique global  solution  $\{u(t),t\in [\tau,\tau+T]\}$ to problem \eqref{2DHD-1.1}. Moreover, the solution $u$ fulfills the following estimate, for any $T>0$,
	\begin{eqnarray}\label{2DHD-eq3.0}
		\mathbb{ E}\left[\sup_{t\in [\tau, \tau+T]}\|u(t)\|_{H}^2\right]+\mathbb{ E}\left[\int_{\tau}^{\tau+T} \|u(t)\|_{V}^2d t\right]\leq C_T\left(1+\mathbb{ E} [\|u_{\tau}\|_{H}^2]+\int_{\tau}^{\tau+T}\|f(s)\|^2_{V'}ds\right).
	\end{eqnarray}
\end{theorem}
\begin{proof}
The proof relies on a standard truncation argument (see, e.g., \cite[Theorem 3.4]{Mao-2008}), which necessitates approximating the locally Lipschitz continuous nonlinear terms $h$ and $G$ by globally Lipschitz continuous ones. For every $k \in \mathbb{N}$, $y\in H$ and $z\in \mathcal{Z}$, we define the mapping $h_k$ and $G_k$ by
\begin{align*}
	h_k(y)=h\left(\frac{\|y\|_{H}\wedge k}{\|y\|_{H}}y\right) \text{~~and~~}
	G_k(y,z)=G\left(\frac{\|y\|_{H}\wedge k}{\|y\|_{H}}y,z\right),
\end{align*}
where we set $\frac{\|y\|_{H}\wedge k}{\|y\|_{H}}y=0$	when $y=0$. Then, by Hypotheses \ref{2DHD-Ass2.4}, for any $y\in H$, it follows that $h_k(y)$ and $G_k(y,z)$ for $z \in \mathcal{Z}$ satisfy both a global Lipschitz and linear growth conditions (i.e., Hypotheses \ref{2DHD-Ass2.3}). More precisely, for every $k\in \mathbb{N}$, $T>0$, there exists a positive number $\mathcal{Q}(k)$ depending only on $k$ such that for all $x, y\in H$ and $t\in [\tau,\tau+T]$, it holds
	\begin{align}\label{2DHD-eq3.1}
		\|h_k(t,x)-h_k(t,y)\|_{\mathcal{L}_2(U;H)}^2
		+
		\int_\mathcal{Z}\|G_k(x,z)-G_k(y,z)\|_{H}^2\nu(dz)\leq \mathcal{Q}(k) \|x-y\|_{H}^2.
	\end{align}
Therefore, by Theorem \ref{2DHD-the3.3}, for any $\tau\in \mathbb{R}$ and $T>0$, there exists a unique solution $u^k(t)$ in
	$\mathcal{D}([\tau,\tau+T],H)\cap L^2([\tau,\tau+T],V)$. Furthermore,  this solution satisfies the following  equality, for every $t\in [\tau,\tau+T]$,  $\mathbb{P}$-a.s.,
	\begin{align*}
		u^k(t) &=u_{\tau} -\mu\int_{\tau}^t \mathcal{A}u^k(s) ds - \int_{\tau}^t B(u^k(s),u^k(s)) ds+\int_{\tau}^t f(s) ds
		\\  &+\varepsilon_1\int_{\tau}^t h(s,u^k(s))d W(s)+\varepsilon_2\int_{\tau}^t\int_{\mathcal{Z}}  G(u^k(s-),z) \widetilde{N}(ds,dz), ~~\text{in~}  V'.
	\end{align*}
	
Define a stopping time
$$
\zeta_k:=\inf \left\{t\geq \tau: \|u^k(t)\|_{H}>k\right\},
$$
where, by convention, $\zeta_{k}=+\infty$ if $\left\{t\geq \tau: \|u^k(t)\|_{H}>k\right\}=\emptyset$. Obviously, for any $t\in [\tau,\zeta_k)$, $\|u^k(t)\|_{H}\leq k$. Then, by the definitions of $h_k$ and $G_k$, we find that for any $k\in \mathbb{N}$, $z\in \mathcal{Z}$ and $t\in [\tau,\zeta_k)$, it holds
\begin{align}\label{2DHD-eq3.2}
	h^{k+1}(t,u^k(t))=h^{k}(t,u^k(t)) \text{~~and~~} G^{k+1}(u^k(t),z)=G^{k}(u^k(t),z).
\end{align}
In fact, we can verify that for any $k\in \mathbb{N}$ and $t\in [\tau,\zeta_k)$, $\varphi^{k+1}(t\wedge\zeta_{k})=\varphi^{k}(t\wedge\zeta_{k})$ and
$$ 
\zeta_{k+1}\geq\zeta_{k},\,\,a.s.,
$$
which implies that for all $t\in [\tau,\zeta_k)$, we have $\varphi^{k+1}(t)=\varphi^{k}(t)$, and the sequence $\zeta_{k}$ is increasing in $k$. 
Consequently,  we can define $\zeta$ as
$\zeta=\lim\limits_{k\rightarrow\infty}\zeta_{k}=\sup\limits_{k\in \mathbb{N}}\zeta_{k}$. This allows us to define $u(t)$ on $t\in [\tau,\zeta)$ as follows:
$$
u(t):=u^k(t), ~~\forall t\in [\tau,\zeta_{k}).
$$
Using \eqref{2DHD-eq3.1} and \eqref{2DHD-eq3.2}, we know that $u(t)$ is unique local solution of problem \eqref{2DHD-1.1} for $t\in [\tau,\zeta)$.
It remains to prove that
\begin{align}\label{2DHD-eq3.3}
	\mathbb{P}(\zeta=\infty)=1.
\end{align}

%Applying It\^{o}'s formula with jump (see Appendix \ref{It-Ap}) to $\|u(t\wedge \zeta_{k})\|_H^2$, for all $t\geq \tau$, we have that $\mathbb{P}$-a.s., 
Applying It\^{o}'s formula with jump to $\|u(t\wedge \zeta_{k})\|_H^2$, for all $t\geq \tau$, we have that $\mathbb{P}$-a.s., 
\begin{align}\label{2DHD-eq3.4}
	\begin{split}
		&~~\|u(t\wedge \zeta_{k})\|_H^2+2\mu \int_{\tau}^{t\wedge \zeta_{k}}\|u(s)\|_{V}^2ds
		= \|u_{\tau}\|_H^2+ 2\int_{\tau}^{t\wedge \zeta_{k}}\langle f(s), u(s)\rangle ds\\
		&+2\varepsilon_1 \int_{\tau}^{t\wedge \zeta_{k}}\langle h(s,u(s)),u(s)\rangle dW(s)
		+\varepsilon_1^2 \int_{\tau}^{t\wedge \zeta_{k}} \|h(s,u(s))\|^2_{\mathcal{L}_2(U;H)} ds\\
		&+2\varepsilon_2\int_{\tau}^{t\wedge \zeta_{k}} \int_{\mathcal{Z}}\langle G(u(s-),z),u(s-)\rangle\widetilde{N}(ds,dz)
		+\varepsilon_2^2\int_{\tau}^{t\wedge \zeta_{k}} \int_{\mathcal{Z}} \|G(u(s-),z)\|^2_{H}N(ds,dz).
	\end{split}
\end{align}
By H\"{o}lder's inequality and Young's inequality we have
\begin{align}\label{2DHD-eq3.5}
	\begin{split}
		2\mathbb{ E}\left[\int_{\tau}^{t\wedge \zeta_{k}}\langle f(s), u(s)\rangle ds\right]
		&\leq \frac{\mu}{2}\mathbb{E}\left[\int_{\tau}^{t\wedge \zeta_{k}} \|u(s)\|_{V}^2ds\right]+\frac{4}{\mu}\int_{\tau}^{t\wedge \zeta_{k}}\|f(s)\|_{V'}^2ds\\
		&\leq \frac{\mu}{2}\mathbb{E}\left[\int_{\tau}^{t\wedge \zeta_{k}} \|u(s)\|_{V}^2ds\right]+\frac{4}{\mu}\int_{\tau}^{t}\|f(s)\|_{V'}^2ds.
	\end{split}
\end{align}
According to the linear growth condition in Hypotheses \ref{2DHD-Ass2.4}, it can be shown that
\begin{align}\label{2DHD-eq3.6}
	\begin{split}
		&~~\varepsilon_1^2 \mathbb{E}\left[\int_{\tau}^{t\wedge \zeta_{k}} \|h(s,u(s))\|^2_{\mathcal{L}_2(U;H)} ds\right] +\varepsilon_2^2\mathbb{E}\left[\int_{\tau}^{t\wedge \zeta_{k}} \int_{\mathcal{Z}} \|G(u(s-),z)\|^2_{H}N(ds,dz)\right]\\
		&\leq \max\{\varepsilon_1^2,\varepsilon_2^2\} \mathbb{E}\left[\int_{\tau}^{t\wedge \zeta_{k}} \left(\|h(s,u(s))\|^2_{\mathcal{L}_2(U;H)}+\int_{\mathcal{Z}} \|G(u(s-),z)\|^2_{H}\nu(dz)\right) ds\right] \\
		&\leq C_{L_g}\mathbb{E}\left[\int_{\tau}^{t\wedge \zeta_{k}} (1+\|u(s)\|_H^2)ds\right]=C_{L_g}\mathbb{E}\left[\int_{\tau}^{t} (1+\|u(s\wedge \zeta_{k})\|_H^2)ds\right].
	\end{split}
\end{align}
Since the process $\int_{\tau}^{t\wedge \zeta_{k}} \int_{\mathcal{Z}}\langle G(u(s-),z),u(s-)\rangle\widetilde{N}(ds,dz)$ is a martingale for $t\geq \tau$, then it yields
$$
2\varepsilon_2\mathbb{ E}\left[\int_{\tau}^{t\wedge \zeta_{k}} \int_{\mathcal{Z}}\langle G(u(s-),z),u(s-)\rangle\widetilde{N}(ds,dz)\right]=0,
$$
which together with \eqref{2DHD-eq3.4}-\eqref{2DHD-eq3.6} can infer
\begin{align*}
	&~~\mathbb{E}\left[\|u(t\wedge \zeta_{k})\|_H^2\right]+\mu \mathbb{E}\left[\int_{\tau}^{t\wedge \zeta_{k}}\|u(s)\|_{V}^2ds\right]\\
	&\leq  \mathbb{E}\left[\|u_{\tau}\|_H^2\right]+\frac{4}{\mu}\int_{\tau}^{t}\|f(s)\|_{V'}^2ds+C_{L_g}\mathbb{E}\left[\int_{\tau}^{t} (1+\|u(s\wedge \zeta_{k})\|_H^2)ds\right].
\end{align*}
By using Gronwall's inequality, we obtain
 \begin{align}\label{2DHD-eq3.8}
 	\mathbb{E}\left[\|u(t\wedge \zeta_{k})\|_H^2\right]\leq  C_{L_g,\mu}\left(\mathbb{E}\left[\|u_{\tau}\|_H^2\right]+\int_{\tau}^{t}\|f(s)\|_{V'}^2ds+t\right)e^{C_{L_g,\mu}t}.
 \end{align}

It follows that 
\begin{align*}
	\mathbb{P}\left(\zeta_{k}\leq t\right)\leq \frac{\mathbb{E}\left[\|u(t\wedge\zeta_{k})\|_{H}^2\mathds{I}_{\zeta_{k}\leq t}\right]}{k^2}\leq \frac{C_{L_g,\mu}\left(\mathbb{E}\left[\|u_{\tau}\|_H^2\right]+\int_{\tau}^{t}\|f(s)\|_{V'}^2ds+t\right)e^{C_{L_g,\mu}t}}{k^2}.
\end{align*}
Letting $k\rightarrow \infty$, we have
$$
\mathbb{P}\left(\zeta \leq t\right)=0,
$$
which along with the arbitrariness of $t\geq \tau$ yields \eqref{2DHD-eq3.3}. 

Moreover, the proof of estimate \eqref{2DHD-eq3.0} follows a similar argument to that of \eqref{2DHD-eq3.8}, except that we need to apply the BDG inequality (see Lemma \ref{BDSL}) to the Poisson random measures. Since the argument is standard, we omit the details for brevity.
This completes the proof.
\end{proof}

\section{Existence of weak pullback mean random attractors}\label{WEEdfikjik}

This section focuses on the existence and uniqueness of weak pullback mean random attractors for the non-autonomous stochastic evolution system (\ref{2DHD-1.1}) driven by multiplicative L\'{e}vy noise in $L^{2}(\Omega, \mathscr{F};H)$ over $ (\Omega,\mathscr{F},\{\mathscr{F}_t\}_{t\in\mathbb{R}},\mathbb{P}) $.  
To this end, we suppose that there exists $\kappa\geq 1$ such that
\begin{align}\label{LSWs4.1}
	\frac{\mu\lambda_1}{2}-\kappa>0.
\end{align}
We now define a mean random dynamical system for problem (\ref{2DHD-1.1}), which will serve as the foundation for studying the existence and uniqueness of weak pullback mean random attractors. According to Theorem \ref{2DHD-the3.3}, we know that for each $ \tau \in \mathbb{R} $ and any initial value $u_{\tau}\in L^2(\Omega, \mathscr{F}_\tau;H)$, problem (\ref{2DHD-1.1}) has a unique $H$-valued $\mathscr{F}_t$-adapted c\`{a}dl\`{a}g solution $u(t,\tau,u_\tau)$ with initial value $u_{\tau}$. Moreover, Theorem \ref{2DHD-the3.3} shows that $u(\cdot,\tau,u_\tau) \in L^2(\Omega,\mathscr{F};\mathcal{D}(\mathbb{R}^{\tau},H))$. By the Lebesgue dominated convergence theorem and \eqref{2DHD-eq3.0} we can deduce that $u\in \mathcal{D}(\mathbb{R}^{\tau},L^2(\Omega,\mathscr{F};H))$. This allows us to define a cocycle generated by the problem under consideration.

Given $ \tau \in \mathbb{R} $ and $ t\in \mathbb{R}^{+} $, let $ \Phi(t,\tau):\, L^{2}(\Omega,\mathscr{F}_{\tau};H) \rightarrow  L^{2}(\Omega,\mathscr{F}_{t+\tau};H)$ be a mapping given by  
\begin{equation}\label{3.1}
	\Phi(t,\tau)(u_{\tau})=u(t+\tau,\tau,u_{\tau}), \, \,\, \, \forall u_{\tau} \in L^{2}(\Omega,\mathscr{F}_{\tau};H),
\end{equation}
where $ u(t+\tau,\tau,u_{\tau}) $ denotes the solution of system (\ref{2DHD-1.1}) with initial value $ u_{\tau} $.
Observe that $ \Phi(0,\tau) $ is the identity operator on $ L^{2}(\Omega,\mathscr{F}_{\tau};H)$. By the uniqueness of solution for system (\ref{2DHD-1.1}), we find that for any $ \tau \in \mathbb{R} $ and $s, t \in \mathbb{R}^{+}$,
%$\Phi(t+s,\tau,u_{\tau})=\Phi(t,s+\tau,\Phi(s,\tau,u_{\tau}))$, that is,
$\Phi(t+s,\tau)=\Phi(t,s+\tau)\,\circ\, \Phi(s,\tau)$.
Therefore, the cocycle $\Phi $ is a mean random dynamical system associated with system (\ref{2DHD-1.1}) on $ L^{2}(\Omega,\mathscr{F};H)$ over $ (\Omega,\mathscr{F},\{\mathscr{F}_t\}_{t\in\mathbb{R}},\mathbb{P}) $ . 

Let $ B=\{ B(\tau) \subseteq L^{2}(\Omega,\mathscr{F}_{\tau};H): \tau \in \mathbb{R}\} $ 
denote a family of nonempty bounded sets such that
\begin{equation}\label{4.3}
	\lim\limits_{\tau \to -\infty} e^{\kappa\tau}\|B(\tau)\|^2_{L^{2}(\Omega,\mathscr{F}_{\tau};{H})}=0,
\end{equation}
where $\kappa\in [1,\mu\lambda_1/2)$, and
$$ \|B(\tau)\|_{L^{2}(\Omega,\mathscr{F}_{\tau};{H})}=\sup_{u\in B(\tau)}\|u\|_{L^{2}(\Omega,\mathscr{F}_{\tau};{H})} .$$ 
We use $\mathfrak{D}$ to denote the collection of all families of nonempty
bounded sets satisfying (\ref{4.3}), it is given as follows:
\begin{equation}\label{SCLDD4.3}
	\mathfrak{D}=\{B= \{B(\tau)\subseteq L^{2}(\Omega,\mathscr{F}_{\tau};H):B(\tau) \ne \emptyset \text{ bounded},\tau \in \mathbb{R} \}: B  \text{ satisfies } (\ref{4.3}) \}.
\end{equation}
We shall consider the existence of weak $\mathfrak{D} $-pullback mean random attractors of $\Phi $ associated with (\ref{2DHD-1.1}). For this purpose, we further assume that the deterministic forcing term $f$ satisfies
\begin{align}\label{LSWs4.5}
	\int_{-\infty}^{\tau}e^{\kappa s} \|f(s)\|_{V'}^2 ds<\infty, \quad \forall \tau\in\mathbb{R}.
\end{align}
It should be noted that \eqref{LSWs4.5} does not require $f(s)$ to be bounded in $V'$ as $s\rightarrow \pm \infty$. As a result, it holds
$$
\lim_{\tau\rightarrow -\infty}\int_{-\infty}^\tau e^{\kappa s}\|f(s)\|^2_{V'}ds=0.
$$

We first derive the following uniform estimates on the solution of system (\ref{2DHD-1.1}) in $L^{2}(\Omega,\mathscr{F}_{\tau};{H})$. 
\begin{lemma}\label{2DHD-lem4.1}
	Let Hypotheses $\ref{2DHD-Ass2.0}$, $\ref{2DHD-Ass2.1}$ and $\ref{2DHD-Ass2.4}$ hold. Assume that \eqref{LSWs4.1} and \eqref{LSWs4.5} are satisfied. Then, there exists $
	\varepsilon_0=\min \left\{1,\sqrt{\frac{\mu \lambda_1}{2L_g}}\right\}$ such that for any $\varepsilon_1, \varepsilon_2 \in (0,\varepsilon_0]$,
	$ \tau \in \mathbb{R} $ and $ B=\{ B(t) \}_{t\in \mathbb{R}}\in \mathfrak{D} $, there exists $ T:=T(\tau,B)>0 $ such that for all $ t\ge T $, the solution $ u$ of system (\ref{2DHD-1.1}) satisfies
	\begin{equation}\label{3.001}
			 \mathbb{E}\left[\left\|u\left(\tau, \tau-t, u_{\tau-t}\right)\right\|^2_{H}\right]  \leq \rho_0+\rho_0 e^{-\kappa \tau}\int_{-\infty}^{\tau} e^{\kappa s}\|f(s)\|_{V'}^2ds,
	\end{equation}
	where 
	$u_{\tau-t} \in B(\tau-t)$, and $\rho_0>0$ is a constant independent of $\tau$ and $B$.
\end{lemma}
\begin{proof}
	Applying It\^{o}'s formula with jump to $e^{\kappa t}\|u\|_{H}^2$, by (\ref{2DHD-1.1}) we have
	\begin{align}\label{2DHD-eq4.7}
		\begin{split}
			&~~d\left(e^{\kappa t}\|u\|_{H}^2\right)+\left(2\mu e^{\kappa t}\|u\|_{V}^2-\kappa e^{\kappa t}\|u\|_{H}^2\right)dt\\
			&=2 e^{\kappa t}\langle f(t), u(t)\rangle dt + \varepsilon_1^2e^{\kappa t}\|h(t,u(t))\|^2_{\mathcal{L}_2(U;H)}dt
			+2 \varepsilon_1 e^{\kappa t} \langle u(t),h(t,u(t))\rangle dW(t)\\
			&+ 2 \varepsilon_2e^{\kappa t}\int_{\mathcal{Z}}\langle u(t-),G(u(t-),z)\rangle\widetilde{N}(dt,dz)
			+\varepsilon_2^2e^{\kappa t}\int_{\mathcal{Z}} \|G(u(t-),z)\|^2_{H}N(dt,dz).
		\end{split}
	\end{align}
	Taking the expectation of \eqref{2DHD-eq4.7} and integrating on $t$ over $[\tau,t]$, we obtain
	\begin{align}\label{2DHD-eq4.8}
%		\begin{split}
			&~~e^{\kappa t}\mathbb{E}\left[\|u(t,\tau,u_{\tau})\|_{H}^2\right]+\frac{\mu}{2} \int_{\tau}^{t} e^{\kappa s}\mathbb{E}\left[\|u(s,\tau,u_{\tau})\|_{V}^2\right]ds+\left(\mu\lambda_1  -\kappa \right)\int_{\tau}^{t}e^{\kappa s}\mathbb{E}\left[\|u(s,\tau,u_{\tau})\|_{H}^2\right]ds\notag \\
			&\leq e^{\kappa \tau}\mathbb{E}\left[\|u_{\tau}\|_{H}^2\right]+\frac{2}{\mu}\int_{\tau}^{t} e^{\kappa s}\|f(s)\|_{V'}^2ds + \varepsilon_1^2\mathbb{E}\left[\int_{\tau}^{t} e^{\kappa s}\|h(s,u(s,\tau,u_{\tau}))\|^2_{\mathcal{L}_2(U;H)}ds\right]\notag \\
			&
			+\varepsilon_2^2\mathbb{E}\left[\int_{\tau}^{t}e^{\kappa s}\int_{\mathcal{Z}}  \|G(u(s-,\tau,u_{\tau}),z)\|^2_{H}N(ds,dz)\right]\notag \\
			&\leq e^{\kappa \tau}\mathbb{E}\left[\|u_{\tau}\|_{H}^2\right]+\frac{2}{\mu}\int_{\tau}^{t} e^{\kappa s}\|f(s)\|_{V'}^2ds + \varepsilon_1^2\int_{\tau}^{t} e^{\kappa s}\mathbb{E}\left[\|h(s,u(s,\tau,u_{\tau}))\|^2_{\mathcal{L}_2(U;H)}\right]ds\notag \\
			&
			+\varepsilon_2^2\int_{\tau}^{t}e^{\kappa s}\mathbb{E}\left[\int_{\mathcal{Z}}  \|G(u(s-,\tau,u_{\tau}),z)\|^2_{H}\nu(dz)\right] ds,
%		\end{split}
	\end{align}
	where we used the following inequality
	 $$
	2 \int_{\tau}^{t}e^{\kappa s}\mathbb{E}\left[\langle f(s), u(s,\tau,u_{\tau})\rangle\right] ds \leq \frac{\mu}{2} \int_{\tau}^{t} e^{\kappa s}\mathbb{E}\left[\|u(s,\tau,u_{\tau})\|_{V}^2\right]ds+\frac{2}{\mu}\int_{\tau}^{t} e^{\kappa s}\|f(s)\|_{V'}^2ds.
	$$
	By the hypothesis $\mathbf{(C.2)}$, we have
	\begin{align}\label{2DHD-eq4.9}
		&~~\varepsilon_1^2\int_{\tau}^{t} e^{\kappa s}\mathbb{E}\left[\|h(s,u(s,\tau,u_{\tau}))\|^2_{\mathcal{L}_2(U;H)}\right]ds
		+\varepsilon_2^2\int_{\tau}^{t}e^{\kappa s}\mathbb{E}\left[\int_{\mathcal{Z}}  \|G(u(s-,\tau,u_{\tau}),z)\|^2_{H}\nu(dz)\right] ds\notag \\
		&\leq \max\{\varepsilon_1^2,\varepsilon_2^2\}\int_{\tau}^{t} e^{\kappa s}\mathbb{E}\left[\|h(s,u(s,\tau,u_{\tau}))\|^2_{\mathcal{L}_2(U;H)}+\int_{\mathcal{Z}}  \|G(u(s-,\tau,u_{\tau}),z)\|^2_{H}\nu(dz)\right]ds\notag \\
		%&\leq L_{g}\max\{\varepsilon_1^2,\varepsilon_2^2\}\int_{\tau}^{t} e^{\kappa s}(1+\mathbb{E}\left[\|u(s,\tau,u_{\tau})\|_{H}^2\right])ds\notag \\
		&\leq L_{g}\varepsilon_0^2\left(\frac{e^{\kappa t}} {\kappa}+\int_{\tau}^{t} e^{\kappa s}\mathbb{E}\left[\|u(s,\tau,u_{\tau})\|_{H}^2\right]ds\right)\notag \\
		&\leq \frac{\mu \lambda_1}{2\kappa}e^{\kappa t}+\frac{\mu \lambda_1}{2}\int_{\tau}^{t}e^{\kappa s}\mathbb{E}\left[\|u(s,\tau,u_{\tau})\|_{H}^2\right]ds.
	\end{align}
	Combining with \eqref{2DHD-eq4.8}-\eqref{2DHD-eq4.9} and \eqref{LSWs4.1}, we can infer that
	\begin{align}\label{2DHD-eq4.10}
				\begin{split}
			&~~\mathbb{E}\left[\|u(t,\tau,u_{\tau})\|_{H}^2\right]+\frac{\mu}{2} \int_{\tau}^{t} e^{\kappa (s-t)}\mathbb{E}\left[\|u(s,\tau,u_{\tau})\|_{V}^2\right]ds \\
			&\leq e^{-\kappa(t-\tau) }\mathbb{E}\left[\|u_{\tau}\|_{H}^2\right]+\frac{2}{\mu}\int_{\tau}^{t} e^{\kappa (s-t)}\|f(s)\|_{V'}^2ds + \frac{\mu \lambda_1}{2\kappa}.
					\end{split}
	\end{align}
	
	We replace $t$ and $\tau$ in \eqref{2DHD-eq4.10} by $\tau$ and $\tau-t$ (with $t\geq 0$). It follows that
	\begin{align}\label{2DHD-eq4.11}
		%		\begin{split}
			&~~\mathbb{E}\left[\|u(\tau,\tau-t,u_{\tau-t})\|_{H}^2\right]+\frac{\mu}{2} \int_{\tau-t}^{\tau} e^{\kappa (s-\tau)}\mathbb{E}\left[\|u(s,\tau-t,u_{\tau-t})\|_{V}^2\right]ds\notag \\
			&\leq e^{-\kappa t }\mathbb{E}\left[\|u_{\tau-t}\|_{H}^2\right]+\frac{2}{\mu}\int_{\tau-t}^{\tau} e^{\kappa (s-\tau)}\|f(s)\|_{V'}^2ds + \frac{\mu \lambda_1}{2\kappa}\notag \\
			&\leq e^{-\kappa t }\mathbb{E}\left[\|u_{\tau-t}\|_{H}^2\right]+\frac{2}{\mu}\int_{-\infty}^{\tau} e^{\kappa (s-\tau)}\|f(s)\|_{V'}^2ds + \frac{\mu \lambda_1}{2}.
			%		\end{split}
	\end{align}
	Thanks to $u_{\tau-t}\in B(\tau-t)$ and $B=\{B(t)\}_{t\in \mathbb{R}}\in \mathfrak{D}$, then we can deduce
	\begin{align*}
		\begin{split}
			\lim_{t\rightarrow +\infty}e^{ - \kappa t}\mathbb{E}\left[\|u_{\tau-t}\|_{H}^2\right]
			\leq e^{-\kappa \tau}\lim_{t\rightarrow +\infty}e^{ \kappa (\tau-t)}\|B(\tau-t)\|^2_{L^2(\Omega,\mathscr{F}_{\tau-t};H)}=0.
		\end{split}
	\end{align*}
	Hence, there exists $T:=T(\tau,B)>0$ such that
	$$
	e^{ - \kappa t}\mathbb{E}\left[\|u_{\tau-t}\|_{H}^2\right]\leq 1, \quad \forall t\geq T.
	$$
	We further obtain that for all $t\geq T$,
	\begin{align*}
		\begin{split}
			\mathbb{E}\left[\|u(\tau,\tau-t,u_{\tau-t})\|_{H}^2\right]\leq \frac{2+\mu \lambda_1}{2}+\frac{2}{\mu}\int_{-\infty}^{\tau} e^{\kappa (s-\tau)}\|f(s)\|_{V'}^2ds,
		\end{split}
	\end{align*}
	which along with \eqref{LSWs4.5} can derive that there exists a suitable $\rho_0>0$ such that the desired conclusion holds. This completes the proof.
\end{proof}

\begin{remark}
	In \eqref{LSWs4.1} and \eqref{4.3}, we impose the condition $\kappa\geq 1$, It should be emphasized that this requirement is not essential; it is introduced solely for the convenience of the scaling argument in \eqref{2DHD-eq4.11}. In fact, it suffices to consider any $\kappa$ that fulfills \eqref{LSWs4.1}, such as $\kappa=\mu\lambda_1 /4$.
\end{remark}

According to Lemma \ref{2DHD-lem4.1}, we can construct the existence of a weakly compact $\mathfrak{D} $-pullback absorbing set for system (\ref{2DHD-1.1}).
\begin{lemma}\label{2DHD-lem4.2}
	Under the assumptions of Lemma $\ref{2DHD-lem4.1}$. There exists $\varepsilon_0>0$ such that for any $\varepsilon_1, \varepsilon_2\in (0,\varepsilon_0]$, the mean random dynamical system $\Phi$ generated by system (\ref{2DHD-1.1}) has a weakly compact $\mathfrak{D}$-pullback bounded absorbing set $K=\{K(\tau):\tau \in \mathbb{R}\}\in \mathfrak{D}$, for any $\tau \in \mathbb{R}$, $K(\tau)$ is given by
	$$
	K(\tau)=\left\{u\in L^2(\Omega,\mathscr{F}_\tau;H):\mathbb{E}\left[\left\| u \right\|_{H}^2\right]\leq R(\tau)\right\},
	$$
	where $R(\tau):=\rho_0+\rho_0 e^{-\kappa \tau}\int_{-\infty}^{\tau} e^{\kappa s}\|f(s)\|_{V'}^2ds$, and $\rho_0$ is the same constant as in Lemma $\ref{2DHD-lem4.1}$.
\end{lemma}
\begin{proof}
	By the assumption \eqref{LSWs4.5}, it is easy to see that $\int_{-\infty}^{\tau} e^{\kappa s}\|f(s)\|_{V'}^2ds$ is well-defined.
	Since $K(\tau)$ is a bounded
	closed convex subset of the reflexive Banach space $L^2(\Omega,\mathscr{F}_\tau;H)$, we know that $K(\tau)$ is weakly compact in $L^2(\Omega,\mathscr{F}_\tau;H)$. Moreover, Lemma \ref{2DHD-lem4.1} implies that for any $\tau \in \mathbb{R}$ and $B=\{B(t)\}_{t \in \mathbb{R}}$,  there exists $T=T(\tau,B)>0$ such that
	$$
	\Phi(t,\tau-t,B(\tau-t))=u(\tau,\tau-t,B(\tau-t))\subset K(\tau)
	$$
	holds for any $t\geq T$ and $\varepsilon_1, \varepsilon_2\in (0,\varepsilon_0]$.
	Particularly, by \eqref{LSWs4.5} we get
	\begin{align*}
		\lim_{\tau \rightarrow -\infty}e^{\kappa\tau}\|K(\tau)\|^2_{L^2(\Omega,\mathscr{F}_\tau;H)}\leq \lim_{\tau \rightarrow -\infty}e^{\kappa\tau}\rho_0 +\lim_{\tau \rightarrow -\infty}\rho_0\int_{-\infty}^{\tau} e^{\kappa s}\|f(s)\|_{V'}^2ds=0,
	\end{align*}
	and hence $K=\{K(\tau)\}_{\tau\in \mathbb{R}}\in \mathfrak{D}$. Therefore, $K$ is a weakly compact $\mathfrak{D} $-pullback absorbing set for $\Phi$.
	This completes the proof.
\end{proof}

We now turn to the central objective of establishing the existence and uniqueness of weak $\mathfrak{D}$-pullback mean random attractors for $\Phi$.
\begin{theorem}
	Under the assumptions of Lemma $\ref{2DHD-lem4.1}$. The mean random dynamical system $\Phi=\{\Phi(t,\tau):t\in \mathbb{R}^+,\tau\in \mathbb{R}\}$ associated with system $(\ref{2DHD-1.1})$ possesses a unique weak $\mathfrak{D}$-pullback mean random attractors $\mathscr{A}=\{\mathscr{A}(\tau)\}_{\tau\in \mathbb{R}}\in \mathfrak{D}$ on $L^2(\Omega,\mathscr{F};H)$ over $(\Omega,\mathscr{F},\{\mathscr{F}_t\}_{t\in \mathbb{R}},\mathbb{P})$, and the attractors $\mathscr{A}$ can be given, for every $\tau\in \mathbb{R}$, $\mathscr{A}(\tau)$ is represented as follows:
	$$
	\mathscr{A}(\tau)=\bigcap_{s\geq 0}{\overline{\bigcup_{t\geq s}\Phi(t,{\tau-t})K(\tau-t)}}^w,
	$$
	where the closure is taken with respect to the weak topology of $L^2(\Omega,\mathscr{F}_\tau;H)$.
\end{theorem}
\begin{proof}
	By Lemma \ref{2DHD-lem4.2}, the mean random dynamical system $\Phi$ possesses a weakly compact $\mathfrak{D}$-pullback bounded absorbing set $K=\{K(\tau)\}_{\tau\in \mathbb{R}}$. Then, the existence and uniqueness of weak $\mathfrak{D}$-pullback mean random attractors $\mathscr{A}\in \mathfrak{D}$ of $\Phi$ follows directly from
	 \cite[Theorem 2.13]{WBX2019}.
\end{proof}

\section{Existence and limiting behavior of invariant measures in the autonomous case}\label{2DHD-inv5}
In this section, we will study the existence of invariant measures for system $(\ref{2DHD-1.1})$ in the autonomous case and their limiting behavior with respect to the noise intensity. Specifically, we consider the following system:
\begin{eqnarray}\label{2DHD-1.01}
		d u(t)+\mu \mathcal{A} u(t)d t+B(u(t),u(t))d t = f(x)dt+ \varepsilon_1 h(u(t))d W(t)+\varepsilon_2\int_{\mathcal{Z}}G(u(t-),z)\widetilde{N}(dt,dz)
\end{eqnarray}
with the initial value $u(0)=u_{0}$.  Note that the results (i.e., the existence and uniqueness of solutions) in Section \ref{2DHD-EUS} still hold for autonomous system  \eqref{2DHD-1.01} with initial value $u(0)=u_{0}$. Unless otherwise stated, we assume that the deterministic forcing term \( f \in V' \), $V$ is compactly
embedded in $H$, and that \( h(u) \) and $G(u,z)$ still satisfy Hypothesis $\ref{2DHD-Ass2.4}$.

We introduce the transition operator of system  \eqref{2DHD-1.01} with initial value $u(0)=u_{0}$. Let $C_b(H)$ denote the collection of all  bounded continuous  Borel-measurable functions from $H$ to $\mathbb{R}$, and $\mathcal{B}(H)$ the space of all bounded Borel-measurable functions. 
Then, for any  $\phi\in C_b(H)$, $t\geq s \geq 0$  and $u_0\in H$, we 
define the transition operator (or semigroup) $p_{s,t}$ as follows:
\begin{equation}\label{mfdeqim4.13}
	(p_{s,t}\phi)(u_0)=\mathbb{E}\left[\phi(u(t,s,u_0))\right],
\end{equation}
where $u(t,s,u_0)$ is the solution of system  \eqref{2DHD-1.01} with initial value $u(0)=u_{0}$. Furthermore, for any $A\in \mathcal{B}(H)$, $t\geq s \geq 0$  and $u_0\in H$, we denote by
$
p(s,u_0;t,A)=p_{s,t}\chi_{A} (u_0 )=\mathbb{P}\left(\{\omega\in \Omega:u(t,s,u_0)\in A\}\right)$
the probability distribution for the solution $u$ of system  \eqref{2DHD-1.01} with initial value $u(0)=u_{0}$, where $\chi_{A}$ is the characteristic function of $A$.

Denote by $\mathcal{P}(H)$ the space of all probability measures on $\left(H,\mathcal{B}(H)\right)$. Then, as defined in \cite{DaPrato2014}, a probability measure $\tilde{\mu} \in \mathcal{P}(H)$ is called invariant if
\begin{equation}\label{mfdeqim4.014}
	\int_{H}(p_t \phi)(u_0)\tilde{\mu}(du_0)=\int_{H}\phi(u_0)\tilde{\mu}(du_0),\,\,\forall t> 0,\,\,\phi\in C_b(H),
\end{equation}
where $p_t$ is the
abbreviation of $p_{0,t}$.
\subsection{Existence of invariant measures}\label{Probledist}
 We first present the following moment estimates for autonomous system \eqref{2DHD-1.01} with initial condition $u(0)=u_{0}$, which are crucial for proving the tightness of the family of probability distributions for solutions in $H$.
\begin{lemma}\label{2DHD-lem5.1}
	Assume that $u_0 \in L^2(\Omega,\mathscr{F}_0;H)$ and that Hypotheses $\ref{2DHD-Ass2.0}$, $\ref{2DHD-Ass2.1}$, and $\ref{2DHD-Ass2.4}$ hold. Then, there exists $
	\varepsilon_0=\min \left\{1,\sqrt{\frac{\mu \lambda_1}{2L_g}}\right\}$ such that for any $\varepsilon_1, \varepsilon_2 \in (0,\varepsilon_0]$, $t\geq 0$ and $\kappa \in (0,\frac{\mu\lambda_1}{2})$, the solution $u(t,0,u_0)$ of system  \eqref{2DHD-1.01} with initial value $u(0)=u_{0}$ satisfies
	\begin{align}\label{2DHD-equation5.002}
		\mathbb{E}\left[{\left\| {u(t,0,{u_0})} \right\|_{H}^2} \right]+\int_0^t e^{-\kappa (t-s)} {\mathbb{E}\left[ {\left\| {u(s,0,{u_0})} \right\|_{V}^2} \right]ds}\leq e^{-\kappa t}\mathbb{E}\left[\|u_0\|_{H}^2\right] + \rho_1 t,
	\end{align}
	and
	\begin{align}\label{2DHD-equation5.2}
		\int_0^t {\mathbb{E}\left[ {\left\| {u(s,0,{u_0})} \right\|_{V}^2} \right]ds}  \leq \mathbb{E}\left[\|u_0\|_{H}^2\right] + \rho_1 t,
	\end{align}
	where $\rho_1= \frac{2\kappa+2}{\kappa\min\{2,\mu\}}\left(\frac{2}{\mu}\|f\|_{V'}^2+\frac{\mu\lambda_1}{2}\right)
	$ is a positive constant independent of $\varepsilon_1, \varepsilon_2$.
\end{lemma}
\begin{proof}
	Following the proof of inequality \eqref{2DHD-eq4.10}, we can apply It\^{o}'s formula with jump to the process $\|u(t,0,u_0)\|_H^2$ and combine it with \eqref{2DHD-1.01} to obtain
	\begin{align*}
		\frac{d}{dt}\mathbb{E}\left[\|u(t,0,u_0)\|_H^2\right]+\frac{\mu}{2} \mathbb{E}\left[\|u(t,0,u_{0})\|_{V}^2\right]+\frac{\mu\lambda_1}{2}\mathbb{E}\left[\|u(t,0,u_{0})\|_{H}^2\right]\leq \frac{2}{\mu}\|f\|_{V'}^2+\frac{\mu\lambda_1}{2},
	\end{align*}
	from which we have
	\begin{align}\label{2DHD-equation5.3}
		\frac{d}{dt}\mathbb{E}\left[\|u(t,0,u_0)\|_H^2\right]+\frac{\mu}{2} \mathbb{E}\left[\|u(t,0,u_{0})\|_{V}^2\right]+\kappa\mathbb{E}\left[\|u(t,0,u_{0})\|_{H}^2\right]\leq \frac{2}{\mu}\|f\|_{V'}^2+\frac{\mu\lambda_1}{2}.
	\end{align}
%		\begin{align}\label{2DHD-equation5.3}
%		\frac{d}{dt}\mathbb{E}\left[\|u(t,0,u_0)\|_H^2\right]+\frac{\mu}{2} \mathbb{E}\left[\|u(t,0,u_{0})\|_{V}^2\right]\leq \frac{2}{\mu}\|f\|_{V'}^2+\frac{\mu\lambda_1}{2}.
%	\end{align}
Multiplying \eqref{2DHD-equation5.3} by  $e^{\kappa t}$ and integrating on $[0,t]$, we can infer \eqref{2DHD-equation5.002}.

	Integrating  \eqref{2DHD-equation5.3} on $t$ from $0$ to $t$, it yields
		\begin{align*}%\label{2DHD-equation5.4}
	\mathbb{E}\left[\|u(t,0,u_0)\|_H^2\right]+\frac{\mu}{2}\int_{0}^{t} \mathbb{E}\left[\|u(s,0,u_{0})\|_{V}^2\right]ds\leq \mathbb{E}\left[\|u_0\|_{H}^2\right]+\left(\frac{2}{\mu}\|f\|_{V'}^2+\frac{\mu\lambda_1}{2}\right)t,
	\end{align*}
	as desired. This completes the proof.
\end{proof}

To prove the existence of invariant measures for system \eqref{2DHD-1.01} with initial value $u(0)=u_0$, we need to establish the Feller property of the transition semigroup $p_t$ for $t\geq 0$. This involves a nonstandard estimate. 
\begin{lemma}\label{2DHD-lem5.2}
	Assume that the hypotheses of Lemma $\ref{2DHD-lem5.1}$ hold. Then, there exists $
	\varepsilon_0>$ such that, for any $\varepsilon_1, \varepsilon_2 \in (0,\varepsilon_0]\subseteq (0,1]$ and any bounded and continuous $\phi:H\rightarrow \mathbb{R}$, the function $u_0 \rightarrow \mathbb{E}\left[\phi(u(t,0,u_0))\right]$ is also bounded and continuous for any $t\geq 0$.
\end{lemma}
\begin{proof}
	The proof proceeds in three parts. 
	
	\textbf{Step 1.} Given $t_0>0$, let $u_0^n \rightarrow u_0$ in $H$ as $n\rightarrow \infty$. Denote
	$ u^n (t) := u(t,0, u_{0}^n ) $, 
	$ u(t) := u( t, 0, u_{0} )$, and $ \varpi^n (t) = u^n (t)
	-  u(t)$ for $t\in [0,t_0]$. Then, for all $t\in [0,t_0]$, $\varpi^n (t)$ satisfies 
	\begin{align*}
		&d\varpi^n (t)+\mu \mathcal{A}\varpi^n (t)dt + \left(B(u^n (t),u^n (t))-B(u(t),u(t))\right)dt\\
		=&
		\varepsilon_1 \left(h(u^n (t))-h(u (t))\right)dW(t)+\varepsilon_2\int_{\mathcal{Z}}\left(G(u^n(t-),z)-G(u(t-),z)\right)\widetilde{N}(dt,dz)
	\end{align*} 
	with initial value $\varpi^n(0)=\varpi^n_0={u}_0^n-{u}_0$. Define 
	a sequence of stopping times by
	$$
	\tau_k:=\inf \left\{t\in [0,t_0]: \|u^n (t)\|_{H}^2>k, \text{~or~} \|u (t)\|_{H}^2>k, \text{~or~} \int_{0}^t \|u (s)\|_{V}^2ds>k\right\},
	$$
	if the set is empty, we set $\tau_k=t_0$. 

%	Based the $C_{\mathbf{c}_0,\mu}>0$ in \eqref{2DHD-equ-Fe5.8}, let
%		$$
%	\mathfrak{F}(t):=e^{-C_{\mathbf{c}_0,\mu} \int_0^{t} \|u (s)\|_{V}^2ds}, \quad \forall t\in [0,t_0].
%	$$
%	
%	It follows from \eqref{2DHD-equ-Fe5.7} that for all  $t\in [0,t_0]$,
%	\begin{align}\label{2DHD-equ-Fe5.007}
%			\begin{split}
%			&~\mathfrak{F}(t\wedge \tau_k)\|\varpi^n (t\wedge \tau_k)\|_{H}^2\leq \|\varpi^n_0\|_{H}^2+\varepsilon_0^2\int_0^{t\wedge \tau_k} \|h(u^n(s))-h(u(s))\|_{\mathcal{L}_2(U;H)}^2ds \\
%			&+ \varepsilon_0^2 \int_0^{t\wedge \tau_k} \int_{\mathcal{Z}} \|G(u^n(s-),z)-G(u(s-),z)\|_{H}^2N(ds,dz)\\
%			&+2\varepsilon_0\sup_{0\leq r \leq t} \left|\int_0^{r\wedge \tau_k} \langle h(u^n(s))-h(u(s)),\varpi^n (s)\rangle dW(s)\right|\\
%			&+2\varepsilon_0 \sup_{0\leq r \leq t}\left|\int_0^{r\wedge \tau_k}  \int_{\mathcal{Z}}\langle G(u^n(s),z)-G(u(s),z),\varpi^n (s)\rangle \widetilde{N}(ds,dz)\right|
%		\end{split}
%	\end{align}

	Applying It\^{o}'s formula with jump to the process $\|\varpi^n (t)\|_{H}^2$, we obtain that for any  $t\in [0,t_0]$, $\mathbb{P}$-a.s.
	\begin{align}\label{2DHD-equ-Fe5.7}
		&~\|\varpi^n (t)\|_{H}^2+2\mu \int_0^t \|\varpi^n (s)\|_{V}^2ds\notag \\
		&=\|\varpi^n_0\|_{H}^2-2\int_0^t \langle B(u^n(s),u^n(s))-B(u(s),u(s)),\varpi^n (s)\rangle ds\notag\\
		&+\varepsilon_1^2 \int_0^t \|h(u^n(s))-h(u(s))\|_{\mathcal{L}_2(U;H)}^2ds + \varepsilon_2^2 \int_0^t \int_{\mathcal{Z}} \|G(u^n(s-),z)-G(u(s-),z)\|_{H}^2N(ds,dz)\notag \\
		&+2\varepsilon_1 \int_0^t \langle h(u^n(s))-h(u(s)),\varpi^n (s)\rangle dW(s)\notag\\
		&+2\varepsilon_2 \int_0^t  \int_{\mathcal{Z}}\langle G(u^n(s-),z)-G(u(s-),z),\varpi^n (s-)\rangle \widetilde{N}(ds,dz).
	\end{align}
	For the second term on the right-hand side of \eqref{2DHD-equ-Fe5.7}, by  Remark \ref{Remweee2.3} we get that
	\begin{align}\label{2DHD-equ-Fe5.8}
	\begin{split}
			&~-2\int_0^t \langle B(u^n(s),u^n(s))-B(u(s),u(s)),\varpi^n (s)\rangle ds\\
		&\leq \mu \int_0^{t} \|\varpi^n (s)\|_V^2ds + C_{\mathbf{c}_0,\mu} \int_0^{t} \|\varpi^n (s)\|_{H}^2 \|u (s)\|_{V}^2ds.
	\end{split}
	\end{align}
	In view of \eqref{2DHD-equ-Fe5.8}, using Gronwall's lemma to \eqref{2DHD-equ-Fe5.7}, we can infer that for any  $t\in [0,t_0]$, $\mathbb{P}$-a.s.
	\begin{align}\label{2DHD-equ-Fe5.9}
%		\begin{split}
			&~\sup_{0\leq r \leq t}\|\varpi^n (r\wedge \tau_k)\|_{H}^2\notag \\
			&\leq e^{C_{\mathbf{c}_0,\mu} \int_0^{t\wedge \tau_k} \|u (s)\|_{V}^2ds}\bigg(\|\varpi^n_0\|_{H}^2+\varepsilon_0^2\int_0^{t\wedge \tau_k} \|h(u^n(s))-h(u(s))\|_{\mathcal{L}_2(U;H)}^2ds\notag \\
			&+ \varepsilon_0^2 \int_0^{t\wedge \tau_k} \int_{\mathcal{Z}} \|G(u^n(s-),z)-G(u(s-),z)\|_{H}^2N(ds,dz)\notag \\
			&+2\varepsilon_0\sup_{0\leq r \leq t} \left|\int_0^{r\wedge \tau_k} \langle h(u^n(s))-h(u(s)),\varpi^n (s)\rangle dW(s)\right|\notag \\
			&+2\varepsilon_0 \sup_{0\leq r \leq t}\left|\int_0^{r\wedge \tau_k}  \int_{\mathcal{Z}}\langle G(u^n(s),z)-G(u(s),z),\varpi^n (s)\rangle \widetilde{N}(ds,dz)\right|\bigg)\notag \\
			&\leq e^{C_{\mathbf{c}_0,\mu,k} t}\bigg(\|\varpi^n_0\|_{H}^2+\varepsilon_0^2\int_0^{t\wedge \tau_k} \|h(u^n(s))-h(u(s))\|_{\mathcal{L}_2(U;H)}^2ds\notag \\
			&+ \varepsilon_0^2 \int_0^{t\wedge \tau_k} \int_{\mathcal{Z}} \|G(u^n(s-),z)-G(u(s-),z)\|_{H}^2N(ds,dz)\notag \\
			&+2\varepsilon_0\sup_{0\leq r \leq t} \left|\int_0^{r\wedge \tau_k} \langle h(u^n(s))-h(u(s)),\varpi^n (s)\rangle dW(s)\right|\notag \\
			&+2\varepsilon_0 \sup_{0\leq r \leq t}\left|\int_0^{r\wedge \tau_k}  \int_{\mathcal{Z}}\langle G(u^n(s-),z)-G(u(s-),z),\varpi^n (s-)\rangle \widetilde{N}(ds,dz)\right|\bigg).
%		\end{split}
	\end{align}
	By the BDG inequality and $\mathbf{(C.3)}$, we have
	\begin{align}\label{2DHD-equ-Fe5.10}
		&~2\varepsilon_0\mathbb{E}\left[\sup_{0\leq r \leq t} \left|\int_0^{r\wedge \tau_k} \langle h(u^n(s))-h(u(s)),\varpi^n (s)\rangle dW(s)\right|\right]\notag \\
		&\leq \frac{1}{4}\mathbb{E}\left[\sup_{0\leq r \leq t}\|\varpi^n (r\wedge \tau_k)\|_{H}^2\right]e^{-C_{\mathbf{c}_0,\mu,k} t}+ C_{\varepsilon_0}\mathbb{E}\left[\int_0^{t\wedge \tau_k} \|h(u^n(s))-h(u(s))\|_{\mathcal{L}_2(U;H)}^2ds\right]e^{C_{\mathbf{c}_0,\mu,k} t}\notag \\
		&\leq \frac{1}{4}\mathbb{E}\left[\sup_{0\leq r \leq t}\|\varpi^n (r\wedge \tau_k)\|_{H}^2\right]e^{-C_{\mathbf{c}_0,\mu,k} t}+ C_{\varepsilon_0}L_{k}e^{C_{\mathbf{c}_0,\mu,k} t}\int_0^{t} \mathbb{E}\left[\sup_{0\leq r \leq s}\|\varpi^n (r\wedge \tau_k)\|_{H}^2\right]ds,	
	\end{align}
	and
		\begin{align}\label{2DHD-equ-Fe5.11}
			&~2\varepsilon_0\mathbb{E}\left[\sup_{0\leq r \leq t}\left|\int_0^{r\wedge \tau_k}  \int_{\mathcal{Z}}\langle G(u^n(s-),z)-G(u(s-),z),\varpi^n (s-)\rangle \widetilde{N}(ds,dz)\right|\right]\notag \\
			&\leq \frac{1}{4}\mathbb{E}\left[\sup_{0\leq r \leq t}\|\varpi^n (r\wedge \tau_k)\|_{H}^2\right]e^{-C_{\mathbf{c}_0,\mu,k} t}\notag\\
			&+ C_{\varepsilon_0}\mathbb{E}\left[\int_0^{t\wedge \tau_k} \int_{\mathcal{Z}} \|G(u^n(s-),z)-G(u(s-),z)\|_{H}^2\nu(dz)ds\right]e^{C_{\mathbf{c}_0,\mu,k} t}\notag \\
				&\leq \frac{1}{4}\mathbb{E}\left[\sup_{0\leq r \leq t}\|\varpi^n (r\wedge \tau_k)\|_{H}^2\right]e^{-C_{\mathbf{c}_0,\mu,k} t}+ C_{\varepsilon_0}L_{k}e^{C_{\mathbf{c}_0,\mu,k} t}\int_0^{t} \mathbb{E}\left[\sup_{0\leq r \leq s}\|\varpi^n (r\wedge \tau_k)\|_{H}^2\right]ds.	
	\end{align}
	
	Therefore, combining with \eqref{2DHD-equ-Fe5.9}-\eqref{2DHD-equ-Fe5.11}, $\mathbf{(C.3)}$, and applying Gronwall's lemma, we can infer that there exists $C_k=C(\mathbf{c}_0,\mu,\varepsilon_0,k)>0$ such that for all $n\in \mathbb{N}$,
	 \begin{align}\label{2DHD-equ-Fe5.12}
	 	\mathbb{E}\left[\sup_{0\leq t \leq t_0}\|\varpi^n (t\wedge \tau_k)\|_{H}^2\right]\leq C_k\|\varpi^n_0\|_{H}^2.
	 \end{align}
	
	\textbf{Step 2.} We prove that the family $\{\mathscr{L}(u(t_0,0,u^n_0)): n\in \mathbb{N}\}$ of the distributions of solutions is tight in $H$.
It suffices to show that, for every $t \in [0, t_0]$, $u(t, 0, u_0^n)$ converges in probability to $u(t, 0, u_0)$. For $R>0$, we can obtain
	\begin{align}\label{5.11}
		\begin{split}
			&\mathbb{P}\left(\|u(t,0, u_0^n)-u(t, 0, u_0)\|_H>R\right)
			\leq  \mathbb{P}\left(\sup_{r\in [0,t]}\|u(r,0, u_0^n)-u(r, 0, u_0)\|_H>R\right)\\
			&\leq  \mathbb{P}\left(\left\{\sup_{r\in [0,t]}\|u(r\wedge \tau_k,0, u_0^n)-u(r\wedge \tau_k, 0, u_0)\|_H>R\right\}\right)+\mathbb{P}\left(\tau_k<t\right).
		\end{split}
	\end{align}
	
	We deal with the terms on the right-hand side of \eqref{5.11}. By \eqref{2DHD-eq3.0} we obtain that there exists a constant $C>0$ such that for all $n\in \mathbb{N}$,
	\begin{align}\label{5.--0013}
			\mathbb{E}\left[\sup_{0\leq t\leq t_0}{\left\| {u(t,0,{u_0^n})} \right\|_{H}^2} \right]+\mathbb{E}\left[\sup_{0\leq t\leq t_0}{\left\| {u(t,0,{u_0})} \right\|_{H}^2} \right]+\mathbb{E} \left[\int_0^{t_0}  {\left\| {u(s,0,{u_0})} \right\|_{V}^2} ds\right]\leq C.
	\end{align} 
	which together with \eqref{2DHD-equ-Fe5.12} and Chebyshev's inequality can derive
	\begin{align}\label{5.0013}
		\begin{split}
			&~\mathbb{P}\left(\left\{\sup_{r\in [0,t]}\|u(r\wedge \tau_k,0, u_0^n)-u(r\wedge \tau_k, 0, u_0)\|_H>R\right\}\right)\\
			&\leq \frac{1}{R^2}\mathbb{E}\left[\sup_{r\in [0,t]}\|u(r\wedge \tau_k,0, u_0^n)-u(r\wedge \tau_k, 0, u_0)\|_H^2\right]\\
			&\leq \frac{1}{R^2}\mathbb{E}\left[\sup_{t\in [0,t_0]}\|u(t\wedge \tau_k,0, u_0^n)-u(t\wedge \tau_k, 0, u_0)\|_H^2\right]\leq \frac{C_k}{R^2}\|u_0^n-u_0\|_{H}^2,
		\end{split}
	\end{align}
	and
	\begin{align}\label{5.13}
		\begin{split}
			&~~\mathbb{P}\left(\tau_k<t\right)\leq \mathbb{P}\left(\sup_{r\in [0,t]}\int_0^r \|u(s,0,u_0)\|_V^2ds>k^2\right)\\
			&+\mathbb{P}\left(\sup_{r\in [0,t]}\|u(r,0,u^n_0)\|_{H}^2>k\right)+\mathbb{P}\left(\sup_{r\in [0,t]}\|u(r,0,u_0)\|_{H}^2>k\right)\\
			&\leq \frac{1}{k^2}\mathbb{E}\left[\int_0^{t_0} \|u(s,0,u_0)\|_V^2ds\right]+\frac{1}{k}\mathbb{E}\left[\sup_{t\in [0,t_0]}\left(\|u(t,0,u^n_0)\|_{H}^2+\|u(t,0,u_0)\|_{H}^2\right)\right]\\
			&\leq \frac{C}{k^2}+\frac{C}{k}.
		\end{split}
	\end{align}
	By \eqref{5.11}-\eqref{5.13}, we obtain
	\begin{align}\label{5.14}
		\begin{split}
			\mathbb{P}\left(\|u(t,0, u_0^n)-u(t, 0, u_0)\|_H^2>R\right)
			\leq \frac{C_k}{R}\|u_0^n-u_0\|_{H}^2+\frac{C}{k^2}+\frac{C}{k}.
		\end{split}
	\end{align}
%	 Since
%	$u_{0}^n \to u_0$ as $n\rightarrow \infty$, which along with Vitali's theorem implies that $u_{0}^n \to u_0$ in   $L^2(\Omega,\mathscr{F}_0; H)$. Passing $n\rightarrow \infty$ and then $k\rightarrow +\infty$ in \eqref{5.14}, we conclude that for any $t\in [0,t_0]$,
Since $u_{0}^n \to u_0$ as $n\rightarrow \infty$, we can first pass to the limit as $n\rightarrow \infty$ and subsequently let $k\rightarrow +\infty$ in \eqref{5.14}. This yields that, for any $t\in [0,t_0]$,
	$$
	\mathbb{P}\left(\|u(t,0, u_0^n)-u(t, 0, u_0)\|_H^2>R\right)=0.
	$$
 
	\textbf{Step 3.}  For $\phi\in C_b(H)$, we shall prove the convergence of the following expression:
	\begin{align}\label{lemma5.4(5.35)}
		\lim_{n\rightarrow\infty} \mathbb{ E}\left[\phi(u(t_0,0,u^n_0))\right]=\mathbb{ E}\left[\phi(u(t_0,0,u_0))\right].
	\end{align}
	By \eqref{5.--0013}, similar to \eqref{5.13}, we obtain that for any $\epsilon>0$, there exists $R_0:=R_0(\epsilon)>0$ such that  
	\begin{align}\label{2DHD-equ-Fe5.18}
		\mathbb{P}\left(\sup_{t\in [0,t_0]}\|u(t,0,u^n_0)\|_{H}>R_0\right)<\frac{\epsilon}{5},\quad \text{for all } n\in \mathbb{N},
	\end{align}
	and
	\begin{align}\label{2DHD-equ-Fe5.19}
		\mathbb{P}\left(\int_0^{t_0} \|u(s,0,u_0)\|_V^2ds>R_0^2\right)<\frac{\epsilon}{5},\quad  \mathbb{P}\left(\sup_{t\in [0,t_0]}\|u(t,0,u_0)\|_{H}>R_0\right)<\frac{\epsilon}{5}.
	\end{align}
	From \textbf{Step 2}, we know that there exists a compact subset $\mathfrak{K}^{\epsilon}$ of $H$ such that for all $n\in \mathbb{N}$,
	\begin{align}\label{2DHD-equ-Fe5.20}
		\mathbb{P}(u(t_0,0,u^n_0)\in \mathfrak{K}^{\epsilon})>1-\frac{\epsilon}{5} \text{~~and~~}
		\mathbb{P}(u(t_0,0,u_0)\in \mathfrak{K}^{\epsilon})>1-\frac{\epsilon}{5}.
	\end{align}  
	Since $\psi:H\rightarrow \mathbb{R}$ is continuous, it is easy to find that $\psi$ is uniformly continuous in $\mathfrak{K}^{\epsilon}$. Then, there exists $\widehat{\delta}>0$ such that, for all $v_1, v_2\in \mathfrak{K}^{\epsilon}$ with $\|v_1-v_2\|<\widehat{\delta}$, there holds
	\begin{align}\label{lemma5.4(5.37)}
		\left|\phi(v_1)-\phi(v_2)\right|<\frac{\epsilon}{2}.
	\end{align}

	Given $n\in \mathbb{N}$, we denote
	\begin{align}\label{2DHD-equ-Fe5.21}
		\begin{split}
			&\widetilde{\Omega}_{n}^{\epsilon}=\left\{\omega \in \Omega: u(t_0,0,u^n_0)\in \mathfrak{K}^{\epsilon} \text{~~and~} \sup_{t\in [0,t_0]}\|u(t,0,u^n_0)\|_{H}\leq R_0\right\},\\
			&\widetilde{\Omega}_{0}^{\epsilon}=\bigg\{\omega \in \Omega: u(t_0,0,u_0)\in \mathfrak{K}^{\epsilon} \text{~~and~} \sup_{t\in [0,t_0]}\|u(t,0,u_0)\|_{H}\leq R_0, \\
			&\qquad ~~~ \text{~~and~} \int_0^{t_0} \|u(s,0,u_0)\|_V^2ds\leq R_0^2\bigg\}.
		\end{split}
	\end{align}
	Let $\Omega_n^{\epsilon} =\widetilde{\Omega}_{n}^{\epsilon}\cap \widetilde{\Omega}_{0}^{\epsilon}$. Then, by \eqref{2DHD-equ-Fe5.18}-\eqref{2DHD-equ-Fe5.20} we obtain that for all $n\in \mathbb{N}$,
	\begin{align}\label{2DHD-equ-Fe5.24}
		\mathbb{P}(\Omega\backslash \Omega_n^{\epsilon})\leq \mathbb{P}(\Omega\backslash \widetilde{\Omega}_{n}^{\epsilon})+\mathbb{P}(\Omega\backslash \widetilde{\Omega}_{0}^{\epsilon})<\epsilon.
	\end{align}
	Since $\phi$ is bounded, there exists $\widehat{\textbf{c}}:=\widehat{\textbf{c}}(\phi)>0$ such that
		$|\phi(v)|\leq \widehat{\textbf{c}}$ for all $v\in H$.
	Thus, for all $n\in \mathbb{N}$, we have
	\begin{align}\label{2DHD-equ-Fe5.25}
		\int_{\Omega\backslash \Omega_n^{\epsilon}} \left|\phi(u(t_0,0,u^n_0))-\phi(u(t_0,0,u_0))\right|d\mathbb{P}\leq 2\widehat{\textbf{c}}\epsilon.
	\end{align}	
	By \eqref{lemma5.4(5.37)} we obtain that for all $n\in \mathbb{N}$,
		\begin{align}\label{lemma5.4(5.48)}
			\begin{split}
				&~~\int_{\Omega_{n}^{\epsilon}} \left|\phi(u(t_0,0,u^n_0))-\phi(u(t_0,0,u_0))\right|d\mathbb{P}\\
				&\leq  \int_{\Omega_{n}^{\epsilon}\cap \{\omega\in \Omega:\|u(t_0,0,u^n_0)-u(t_0,0,u_0)\|_{H}\geq \widehat{\delta}\}} \left|\phi(u(t_0,0,u^n_0))-\phi(u(t_0,0,u_0))\right|d\mathbb{P}\\
				&+ \int_{\Omega_{n}^{\epsilon}\cap \{\omega\in \Omega:\|u(t_0,0,u^n_0)-u(t_0,0,u_0)\|_{H}<\widehat{\delta}\}} \left|\phi(u(t_0,0,u^n_0))-\phi(u(t_0,0,u_0))\right|d\mathbb{P}\\
				&\leq \int_{\Omega_{n}^{\epsilon}\cap \{\omega\in \Omega:\|u(t_0,0,u^n_0)-u(t_0,0,u_0)\|_{H}\geq \widehat{\delta}\}} \left|\phi(u(t_0,0,u^n_0))-\phi(u(t_0,0,u_0))\right|d\mathbb{P}+\frac{\epsilon}{2}.
			\end{split}
		\end{align}
	According to the definition of the stopping time $\tau_k$ in {\textbf{Step 1}}, we have $\tau_k(\omega)=t_0$ for all $\omega \in \Omega_{n}^{\epsilon}$. Hence, we can obtain from \eqref{2DHD-equ-Fe5.12} that
	\begin{align*}
		&~\mathbb{P}\left(\left\{\omega\in \Omega_{n}^{\epsilon}: \|u(t_0,0,u^n_0)-u(t_0,0,u_0)\|_{H}\geq \widehat{\delta}\right\}\right)
		\\
		&=\mathbb{P}\left(\left\{\omega\in \Omega_{n}^{\epsilon}: \|u(t_0\wedge \tau_k,0,u^n_0)-u(t_0\wedge \tau_k,0,u_0)\|_{H}\geq \widehat{\delta}\right\}\right)\\
		&\leq \mathbb{P}\left(\left\{\omega\in \Omega: \|u(t_0\wedge \tau_k,0,u^n_0)-u(t_0\wedge \tau_k,0,u_0)\|_{H}\geq \widehat{\delta}\right\}\right)\\
		&\leq \frac{1}{\widehat{\delta}^2}\mathbb{E}\left[\|u(t_0\wedge \tau_k,0,u^n_0)-u(t_0\wedge \tau_k,0,u_0)\|_{H}^2\right] \leq \frac{C_k}{\widehat{\delta}^2}\mathbb{E}\left[\|u^n_0-u_0\|_{H}^2\right].
	\end{align*}
   Since
   $u_{0}^n \to u_0$ in $H$ as $n\rightarrow \infty$, there exists $N\in \mathbb{N}$ such that for all $n\geq N$,
   \begin{align*}
   	\mathbb{P}\left(\left\{\omega\in \Omega_{n}^{\epsilon}: \|u(t_0,0,u^n_0)-u(t_0,0,u_0)\|_{H}\geq \widehat{\delta}\right\}\right)\leq \frac{\epsilon}{2},
   \end{align*}
   which along with \eqref{lemma5.4(5.48)} yields
   \begin{align}\label{2DHD-equ-Fe5.27}
   	\int_{\Omega_{n}^{\epsilon}} \left|\phi(u(t_0,0,u^n_0))-\phi(u(t_0,0,u_0))\right|d\mathbb{P}\leq \epsilon.
   \end{align}
   It follows from \eqref{2DHD-equ-Fe5.25} and \eqref{2DHD-equ-Fe5.27} that
   \begin{align*}
   	\limsup_{n\rightarrow \infty}\int_{\Omega} \left|\phi(u(t_0,0,u^n_0))-\phi(u(t_0,0,u_0))\right|d\mathbb{P}\leq \epsilon+2\widehat{\textbf{c}}\epsilon.
   \end{align*}
   Thanks to the arbitrariness of $\epsilon$, we have
   $$
   \limsup_{n\rightarrow \infty}\int_{\Omega} \left|\phi(u(t_0,0,u^n_0))-\phi(u(t_0,0,u_0))\right|d\mathbb{P}\leq 0,
   $$
   which implies \eqref{lemma5.4(5.35)}. This completes the proof.
\end{proof}

\begin{remark}\label{rem5.3sdfsdf}
	 The inequality \eqref{2DHD-equ-Fe5.12} can also be obtained through alternative a priori estimation techniques. Specifically, by introducing a weighting function $	\mathfrak{F}(t):=e^{-C_{\mathbf{c}_0,\mu} \int_0^{t} \|u (s)\|_{V}^2ds}$, $\forall t\in [0,t_0]$, here $C_{\mathbf{c}_0,\mu}$ is the same number as \eqref{2DHD-equ-Fe5.8}, and applying It\^{o}’s formula to the process $\mathfrak{F}(t)\|\varpi^n (t)\|_{H}^2$, one can derive the desired estimate. This approach shares some similarities with the estimation presented later in Lemma \ref{2DHD-lemma5.5}.
\end{remark}

\begin{theorem}\label{2DHD-the5.3}
		Under the hypotheses of Lemma $\ref{2DHD-lem5.1}$. There exists $
		\varepsilon_0=\min \left\{1,\sqrt{\frac{\mu \lambda_1}{2L_g}}\right\}$ such that for any $\varepsilon_1, \varepsilon_2 \in (0,\varepsilon_0]$, system  \eqref{2DHD-1.01} with initial value $u(0)=u_{0}$ has an invariant measure $\tilde{\mu}$ on $H$, that is to say, there exists a probability measure $\tilde{\mu}$ in $H$ such that for any $\phi\in C_b(H)$ and $t\geq 0$,
		\begin{equation}\label{mfdeqim4.016}
			\int_{H}\left(\int_{H}\phi(x)p(0,y;t,dx)\right)d\tilde{\mu}(y)=\int_{H} \phi(y)d\tilde{\mu}(y).
		\end{equation}
\end{theorem}
\begin{proof}
	We know from Lemma \ref{2DHD-lem5.2} that the transition semigroup  $\{p_{s,t}\}_{0\leq s \leq r\leq t}$ is Feller. Using the standard argument of \cite[Theorem 9.14]{DaPrato2014}, it is easy to prove that the Markov property of the solution $u(t,s,u_0)$ of system \eqref{2DHD-1.01} with initial value $u(0)=u_{0}$; that is, for any $0\leq s \leq r\leq t \leq T$, it holds
	$\mathbb{E}\left[\phi(u(t,s,u_0))|\mathscr{F}_r\right]=(p_{r,t}\phi)(u(r,s,u_0))$ $\mathbb{P}$-a.s.,
	which implies that for any $\phi\in \mathcal{B}(H)$, $u_0 \in H$ and $0\leq s\leq r\leq t\leq T $, $(p_{s,t}\phi)(u_0)=(p_{s,r}(p_{r,t}\phi))(u_0)$,
	and for every Borel set $A$, there is the following Chapman-Kolmogorov equation
	\begin{equation*}%\label{mfdeqim4.17}
		p(s,u_0;t,A)=\int_{H}p(s,u_0;r,dy)p(r,y;t,A),\quad \forall u_0 \in H,\, 0\leq s\leq r\leq t\leq T.
	\end{equation*}
	In addition, the construction of the solution of system \eqref{2DHD-1.01}, we find that the transition semigroup is homogeneous; that is, for any $u_0 \in H$ and $0\leq s \leq r\leq t \leq T$, $p(s,u_0;t,\cdot)=p(0,u_0;t-s,\cdot)$. 
	
	 It remains to prove the sequence of probability measures $\tilde{\mu}_{t,u_0}$ is tight in $H$, where $\tilde{\mu}_{t,u_0}(A):=\frac{1}{t}\int_{0}^{t} p(0,u_0;s,A)$. By Chebychev's inequality and Lemma \ref{2DHD-lem5.1}, we can infer that for any $T_0>0$ and $\mathfrak{R}>0$,
	\begin{align*}
		&~\sup_{t\geq T_0}\frac{1}{t}\int_0^t \mathbb{P}\left(\|u(s,0,u_0)\|_V> \mathfrak{R}\right)ds\\
		&\leq \sup_{t\geq T_0}\frac{1}{\mathfrak{R}^2 t}\int_0^t \mathbb{E}\left[\|u(s,0,u_0)\|_V^2\right]ds\leq \sup_{t\geq T_0}\frac{1}{\mathfrak{R}^2}\left(\frac{\mathbb{E}\left[\|u_0\|_{H}^2\right] }{t}+ \rho_1\right)\\
		&\leq \frac{1}{\mathfrak{R}^2 T_0}\left(\mathbb{E}\left[\|u_0\|_{H}^2\right] + T_0\rho_1\right),
	\end{align*}
	which means that for all $t\geq T_0$ and $\delta>0$, there exists $\mathfrak{R}_0:=\mathfrak{R}(u_0,T_0,\delta)>0$ such that for any $\mathfrak{R}\geq \mathfrak{R}_0$,
	\begin{align}\label{2DHD-equation5.8}
		\begin{split}
			\tilde{\mu}_{t,u_0}&=\frac{1}{t}\int_{0}^{t} \mathbb{P}\left(\{\omega\in\Omega: u(s,0,u_0)\in A \}\right)ds\\
			&\geq  \frac{1}{t}\int_{0}^{t} \mathbb{P}\left(\{\omega\in\Omega: \|u(s,0,u_0)\|_{V}\leq  \mathfrak{R}_0\}\right)ds\\
			&\geq 1- \sup_{t\geq T_0}\frac{1}{t}\int_{0}^{t} \mathbb{P}\left(\{\omega\in\Omega: \|u(s,0,u_0)\|_{V}>  \mathfrak{R}_0\}\right)ds> 1-\delta,
		\end{split}
	\end{align}
	where $A$ denotes the ball $\mathbb{B}(0,\mathfrak{R})$ centered at $0$ with radius $\mathfrak{R}$ in $V$. Note that the compactness of $V\hookrightarrow H$, we know \eqref{2DHD-equation5.8} that for any $\delta>0$, there exists a compact set $\mathscr{Z} \in H$ such that $\tilde{\mu}(\mathscr{Z})> 1-\delta$ for all $t\geq T_0$. This shows that the tightness of the sequence of probability measures $\tilde{\mu}_{t,u_0}$ in $H$. 
	Therefore, by Krylov-Bogolyubov’s theorem we know that there exists a sequence $t_n \rightarrow \infty$ as $n\rightarrow \infty$ such that $\tilde{\mu}_{t_n,u_0}\rightarrow \tilde{\mu}$ weakly as $n\rightarrow \infty$. 
	
	To prove \eqref{mfdeqim4.016}, we need to make use of the Chapman-Kolmogorov equation, the homogeneity of transition semigroup, and the weak convergence $\tilde{\mu}_{t_n,u_0}\xrightarrow[]{w} \tilde{\mu}$. Following the proof of \cite[Theorem $8.3$]{WBX2019b}, for any $s,t\geq 0$ and $\phi\in C_b({H})$, one can derive
	\begin{align*}
		\int_{H} \phi  (y)d\tilde \mu (y) \hfill
		&= \int_{H} {\left( {\int_{H} {\phi (x)p(0,y;t,dx)} } \right)} d\tilde \mu (y),
	\end{align*}
	which implies \eqref{mfdeqim4.016}. This completes the proof.
\end{proof}

\subsection{Limiting behavior of invariant measures}
This subsection mainly studies the limiting behavior of invariant measures of system  \eqref{2DHD-1.01} driven by L\'{e}vy noise with initial value $u(0)=u_{0}$ as the noise intensities $(\varepsilon_{1,n}, \varepsilon_{2,n})\in (0,\varepsilon_0]^2\subset (0,1]^2$ tends to $(\widehat{\varepsilon}_1,\widehat{\varepsilon}_2)\in [0,\varepsilon_0]^2\subset  [0,1]^2$. To this end, we denote by $u^{\varepsilon_{1}, \varepsilon_{2}}(t,0,u_0)$ the solution of \eqref{2DHD-1.01} with initial value $u(0)=u_{0}$,  and $u^{\widehat{\varepsilon}_1,\widehat{\varepsilon}_2}(t,0,u_0)$ the solution of the limiting system (as $(\varepsilon_1, \varepsilon_2) \to (\widehat{\varepsilon}_1,\widehat{\varepsilon}_2)$) of \eqref{2DHD-1.01}, with the same initial value $u(0)=u_{0}$.

In what follows, we derive the convergence in probability of solutions of \eqref{2DHD-1.01} with respect to the Gaussian noise intensity $\varepsilon_{1}$ and L\'{e}vy noise intensity $\varepsilon_{2}$. 

%To this end, we denote by $u^{\widehat{\varepsilon}_1,\widehat{\varepsilon}_2}$ the solution of the following limiting system:to denote the solutions of \eqref{2DHD-1.01} with initial value $u(0)=u_{0}$ and the limiting system of \eqref{2DHD-1.01} with initial value $u(0)=u_{0}$, respectively.
%\begin{align}\label{2DHD-1.01-lim}
%	\begin{split}
%		&d u^{\widehat{\varepsilon}_1,\widehat{\varepsilon}_2}(t)+\mu \mathcal{A} u^{\widehat{\varepsilon}_1,\widehat{\varepsilon}_2}(t)d t+B(u^{\widehat{\varepsilon}_1,\widehat{\varepsilon}_2}(t),u^{\widehat{\varepsilon}_1,\widehat{\varepsilon}_2}(t))d t \\
%		&= f(x)dt+ \varepsilon_1 h(u^{\widehat{\varepsilon}_1,\widehat{\varepsilon}_2}(t))d W(t)+\varepsilon_2\int_{\mathcal{Z}}G(u^{\widehat{\varepsilon}_1,\widehat{\varepsilon}_2}(t-),z)\widetilde{N}(dt,dz)
%	\end{split}
%\end{align}
%with the initial value $u^{\widehat{\varepsilon}_1,\widehat{\varepsilon}_2}(0)=u_0$. 
%Note that the conclusions in Section \ref{2DHD-EUS} remains valid for $(\widehat{\varepsilon}_1,\widehat{\varepsilon}_2)\in [0,\varepsilon_0]^2 \subseteq [0,1]^2$.

\begin{lemma}\label{2DHD-lemma5.5}
	Suppose that the hypothesis of Lemma $\ref{2DHD-lem5.1}$ hold. Then, there exists $
	\varepsilon_0>0$ such that for any bounded subset $\mathfrak{E}\subset H$, $ T>0$,  $\varkappa>0$, $(\varepsilon_{1}, \varepsilon_{2}) \in (0,\varepsilon_0]^2$ and $(\widehat{\varepsilon}_1,\widehat{\varepsilon}_2)\in [0,\varepsilon_0]^2$, 
	\begin{equation*}
		\lim _{(\varepsilon_{1}, \varepsilon_{2}) \rightarrow (\widehat{\varepsilon}_1,\widehat{\varepsilon}_2)} \sup _{u_0\in \mathfrak{E}} \mathbb{P}\left( \left\{ \omega\in\Omega:   \sup_{0\leq t\leq T}\left\|u^{\varepsilon_{1}, \varepsilon_{2}}(t,0,u_0)-u^{\widehat{\varepsilon}_1,\widehat{\varepsilon}_2}(t,0,u_0)\right\|_{H} \geq\varkappa      \right\} \right)=0.
	\end{equation*}
\end{lemma}
\begin{proof}
	For simplicity, we write $u^{\varepsilon_{1}, \varepsilon_{2}}(t):=u^{\varepsilon_{1}, \varepsilon_{2}}(t,0,u_0)$ and $u^{\widehat{\varepsilon}_1,\widehat{\varepsilon}_2}(t):=u^{\widehat{\varepsilon}_1,\widehat{\varepsilon}_2}(t,0,u_0)$. By Theorem \ref{2DHD-the3.3}, we find that for any bounded subset $\mathfrak{E}\subset H$ and any $ T>0$, there exist $
	\varepsilon_0>0$ and $C:=C_{\mathfrak{E},T}>0$ (independent of the noise intensities $\varepsilon_{1}, \varepsilon_{2}, \widehat{\varepsilon}_1,\widehat{\varepsilon}_2$) such that for all $(\varepsilon_{1}, \varepsilon_{2}) \in (0,\varepsilon_0]^2$ and $(\widehat{\varepsilon}_1,\widehat{\varepsilon}_2)\in [0,\varepsilon_0]^2$,  
	\begin{align}\label{2DHD-eq5.30}
		\begin{split}
			&~~\sup _{u_0 \in \mathfrak{E}}\mathbb{ E}\left[\sup_{0\leq t\leq T} \|u^{\varepsilon_{1}, \varepsilon_{2}}(t)\|^{2}_{H}\right]+\sup _{u_0 \in \mathfrak{E}}\mathbb{ E}\left[\sup_{0\leq t\leq T} \|u^{\widehat{\varepsilon}_1,\widehat{\varepsilon}_2}(t)\|_{H}^2\right]
			+\sup _{u_0 \in \mathfrak{E}}\mathbb{E}\left[\int_{0}^{T} \|u^{\varepsilon_{1}, \varepsilon_{2}}(s)\|_{V}^2ds\right]\leq C,
		\end{split}
	\end{align}
	which together with Chebyshev's inequality can deduce that for any $\epsilon>0$, there exists a $\widehat{R}:=\widehat{R}(\epsilon)>0$ such that 
	\begin{align}\label{2DHD-eq5.31}
		\begin{split}
			&\sup _{u_0 \in \mathfrak{E}}\mathbb{P}\left(\left\{\omega\in \Omega:\sup_{0\leq t\leq T}\|u^{\varepsilon_{1}, \varepsilon_{2}}(t)\|_{H}>\widehat{R}\right\}\right)<\frac{\epsilon}{3},\\
			& \sup _{u_0 \in \mathfrak{E}}\mathbb{P}\left(\left\{\omega\in \Omega:\int_{0}^{T} \|u^{\varepsilon_{1}, \varepsilon_{2}}(s)\|_{V}^2ds>\widehat{R}^2\right\}\right)<\frac{\epsilon}{3},\\
			&\sup _{u_0 \in \mathfrak{E}}\mathbb{P}\left(\left\{\omega\in \Omega:\sup_{0\leq t\leq T}\|u^{\widehat{\varepsilon}_1,\widehat{\varepsilon}_2}(t)\|_{H}>\widehat{R}\right\}\right)<\frac{\epsilon}{3}.
		\end{split}
	\end{align} 
	
	Denote
	\begin{align*}
		&\Omega_{\epsilon}=\left\{\omega\in \Omega: \sup_{0\leq t\leq T}\|u^{\varepsilon_{1}, \varepsilon_{2}}(t)\|_{H}\leq \widehat{R}, \text{ and } \sup_{0\leq t\leq T}\|u^{\widehat{\varepsilon}_1,\widehat{\varepsilon}_2}(t)\|_{H}\leq \widehat{R},  \text{ and }  \int_{0}^{T} \|u^{\varepsilon_{1}, \varepsilon_{2}}(s)\|_{V}^2ds\leq \widehat{R}^2\right\}.
	\end{align*}
Then, we can get $\mathbb{P}\left(\Omega \backslash \Omega_\epsilon\right)< \epsilon$. We define the stopping time
	$$
	\tau^{\epsilon}_{\widehat{R}}=\inf \left\{t\geq 0: \|u^{\varepsilon_{1}, \varepsilon_{2}}(t)\|_{H}>\widehat{R}, \text{ or } \|u^{\widehat{\varepsilon}_1,\widehat{\varepsilon}_2}(t)\|_{H}>\widehat{R}, \text{ or } \int_{0}^{T} \|u^{\varepsilon_{1}, \varepsilon_{2}}(s)\|_{V}^2ds>\widehat{R}^2 \right\},
	$$	 
it is obvious that $\tau^{\epsilon}_{\widehat{R}}\geq T$ for any $\omega \in \Omega_{\epsilon}$.
For any $ T>0$ and  $\varkappa>0$, we have
	\begin{align}\label{2DHD-eq5.32}
		\begin{split}
			&~~\mathbb{P}\left( \left\{ \omega\in\Omega:   \sup_{0\leq t\leq T}\left\|u^{\varepsilon_{1}, \varepsilon_{2}}\left(t\right)-u^{\widehat{\varepsilon}_1,\widehat{\varepsilon}_2}\left(t \right)\right\|_{H} \geq\varkappa      \right\} \right)\\
			&\leq \mathbb{P}\left( \left\{ \omega\in\Omega_\epsilon:   \sup_{0\leq t\leq T}\left\|u^{\varepsilon_{1}, \varepsilon_{2}}\left(t\right)-u^{\widehat{\varepsilon}_1,\widehat{\varepsilon}_2}\left(t \right)\right\|_{H} \geq\varkappa      \right\}\cap \left\{\tau^{\epsilon}_{\widehat{R}}\geq t\right\}\ \right)\\
			&+\mathbb{P}\left( \left\{ \omega\in\Omega \backslash \Omega_\epsilon:   \sup_{0\leq t\leq T}\left\|u^{\varepsilon_{1}, \varepsilon_{2}}\left(t\right)-u^{\widehat{\varepsilon}_1,\widehat{\varepsilon}_2}\left(t \right)\right\|_{H} \geq\varkappa      \right\} \cap \left\{\tau^{\epsilon}_{\widehat{R}}<t\right\} \right)\\
			&<\mathbb{P}\left( \left\{ \omega\in\Omega:   \sup_{0\leq t\leq T}\left\|u^{\varepsilon_{1}, \varepsilon_{2}}(t\wedge \tau^{\epsilon}_{\widehat{R}})-u^{\widehat{\varepsilon}_1,\widehat{\varepsilon}_2}(t\wedge \tau^{\epsilon}_{\widehat{R}} )\right\|_{H} \geq\varkappa      \right\} \right)+\epsilon.
		\end{split}
	\end{align}	 
	Next, we prove that 
	for any bounded subset $\mathfrak{E}\subset H$, as $(\varepsilon_{1}, \varepsilon_{2}) \rightarrow (\widehat{\varepsilon}_1,\widehat{\varepsilon}_2)$, for any $u_0\in \mathfrak{E}$, it holds
	\begin{align}\label{2DHD-eq5.33}
		\mathbb{P}\left( \left\{ \omega\in\Omega:   \sup_{0\leq t\leq T}\left\|u^{\varepsilon_{1}, \varepsilon_{2}}(t\wedge \tau^{\epsilon}_{\widehat{R}})-u^{\widehat{\varepsilon}_1,\widehat{\varepsilon}_2}(t\wedge \tau^{\epsilon}_{\widehat{R}} )\right\|_{H} \geq\varkappa      \right\} \right)\rightarrow 0.
	\end{align}
	
	As stated in Remark \ref{rem5.3sdfsdf}, let $	\mathfrak{F}^{\varepsilon_{1}, \varepsilon_{2}}(t):=e^{-\beta \int_0^{t} \|u^{\varepsilon_{1}, \varepsilon_{2}} (s)\|_{V}^2ds}$ for all $t\in [0,T]$ and $u_0 \in \mathfrak{E}$, where $\beta>0$ is an undetermined constant. 
	By \eqref{2DHD-1.01} we have
	\begin{align*}
		&~d\left(u^{\varepsilon_{1}, \varepsilon_{2}}(t)-u^{\widehat{\varepsilon}_1,\widehat{\varepsilon}_2}(t)\right)+\mu \left(\mathcal{A}u^{\varepsilon_{1}, \varepsilon_{2}}(t)-\mathcal{A}u^{\widehat{\varepsilon}_1,\widehat{\varepsilon}_2}(t)\right)dt\\
		&=-\left(B(u^{\varepsilon_{1}, \varepsilon_{2}}(t)-u^{\widehat{\varepsilon}_1,\widehat{\varepsilon}_2}(t),u^{\varepsilon_{1}, \varepsilon_{2}}(t))-B(u^{\varepsilon_{1}, \varepsilon_{2}}(t),u^{\varepsilon_{1}, \varepsilon_{2}}(t)-u^{\widehat{\varepsilon}_1,\widehat{\varepsilon}_2}(t))\right)dt\\
		&+\left(\varepsilon_1-\widehat{\varepsilon}_1\right)h\left(u^{\varepsilon_{1}, \varepsilon_{2}}(t)\right)dW(t)
		+\widehat{\varepsilon}_1\left(h\left(u^{\varepsilon_{1}, \varepsilon_{2}}(t)\right)-h(u^{\widehat{\varepsilon}_1,\widehat{\varepsilon}_2}(t))\right)dW(t)\\
		&+\left(\varepsilon_2-\widehat{\varepsilon}_2\right)\int_{\mathcal{Z}}G(u^{\varepsilon_{1}, \varepsilon_{2}}(t-),z)\widetilde{N}(dt,dz) +\widehat{\varepsilon}_2\int_{\mathcal{Z}}\left(G(u^{\varepsilon_{1}, \varepsilon_{2}}(t-),z)-G(u^{\widehat{\varepsilon}_1,\widehat{\varepsilon}_2}(t-),z)\right)\widetilde{N}(dt,dz).
	\end{align*} 
	Applying It\^{o}'s formula with jump to $\mathfrak{F}^{\varepsilon_{1}, \varepsilon_{2}}(t)\|u^{\varepsilon_{1}, \varepsilon_{2}}(t)-u^{\widehat{\varepsilon}_1,\widehat{\varepsilon}_2}(t)\|_{H}^2$ in the above inequality, we obtain that for all $t\in [0,T]$,
	\begin{align}\label{2DHD-eq5.34}
		&\mathfrak{F}^{\varepsilon_{1}, \varepsilon_{2}}(t\wedge \tau^{\epsilon}_{\widehat{R}})\|u^{\varepsilon_{1}, \varepsilon_{2}}(t\wedge \tau^{\epsilon}_{\widehat{R}})-u^{\widehat{\varepsilon}_1,\widehat{\varepsilon}_2}(t\wedge \tau^{\epsilon}_{\widehat{R}})\|_{H}^2
		+2\mu\int_{0}^{t\wedge \tau^{\epsilon}_{\widehat{R}}}\mathfrak{F}^{\varepsilon_{1}, \varepsilon_{2}}(s)\|u^{\varepsilon_{1}, \varepsilon_{2}}(s)-u^{\widehat{\varepsilon}_1,\widehat{\varepsilon}_2}(s)\|_{V}^2ds\notag \\
		&=-\beta \int_{0}^{t\wedge \tau^{\epsilon}_{\widehat{R}}}\mathfrak{F}^{\varepsilon_{1}, \varepsilon_{2}}(s)\|u^{\varepsilon_{1}, \varepsilon_{2}}(s)\|_{V}^2\|u^{\varepsilon_{1}, \varepsilon_{2}}(s)-u^{\widehat{\varepsilon}_1,\widehat{\varepsilon}_2}(s)\|_{H}^2ds\notag\\
		&-2\int_{0}^{t\wedge \tau^{\epsilon}_{\widehat{R}}}\mathfrak{F}^{\varepsilon_{1}, \varepsilon_{2}}(s) \langle B(u^{\varepsilon_{1}, \varepsilon_{2}}(t)-u^{\widehat{\varepsilon}_1,\widehat{\varepsilon}_2}(t),u^{\varepsilon_{1}, \varepsilon_{2}}(t)),u^{\varepsilon_{1}, \varepsilon_{2}}(s)-u^{\widehat{\varepsilon}_1,\widehat{\varepsilon}_2}(s)\rangle ds\notag\\
		&+\int_{0}^{t\wedge \tau^{\epsilon}_{\widehat{R}}}\mathfrak{F}^{\varepsilon_{1}, \varepsilon_{2}}(s)\left(\left(\varepsilon_1-\widehat{\varepsilon}_1\right)^2\|h\left(u^{\varepsilon_{1}, \varepsilon_{2}}(s)\right)\|_{\mathcal{L}_2(U;H)}^2+\widehat{\varepsilon}_1^2\|h\left(u^{\varepsilon_{1}, \varepsilon_{2}}(s)\right)-h(u^{\widehat{\varepsilon}_1,\widehat{\varepsilon}_2}(s))\|_{\mathcal{L}_2(U;H)}^2\right)ds\notag\\
		&+2\left(\varepsilon_1-\widehat{\varepsilon}_1\right)\int_{0}^{t\wedge \tau^{\epsilon}_{\widehat{R}}}\mathfrak{F}^{\varepsilon_{1}, \varepsilon_{2}}(s)\langle h\left(u^{\varepsilon_{1}, \varepsilon_{2}}(s)\right),u^{\varepsilon_{1}, \varepsilon_{2}}(s)-u^{\widehat{\varepsilon}_1,\widehat{\varepsilon}_2}(s)\rangle dW(s)\notag\\
		&+2\widehat{\varepsilon}_1\int_{0}^{t\wedge \tau^{\epsilon}_{\widehat{R}}}\mathfrak{F}^{\varepsilon_{1}, \varepsilon_{2}}(s)\langle h\left(u^{\varepsilon_{1}, \varepsilon_{2}}(s)\right)-h(u^{\widehat{\varepsilon}_1,\widehat{\varepsilon}_2}(s)),u^{\varepsilon_{1}, \varepsilon_{2}}(s)-u^{\widehat{\varepsilon}_1,\widehat{\varepsilon}_2}(s)\rangle dW(s)\notag\\
		&+\left(\varepsilon_2-\widehat{\varepsilon}_2\right)^2\int_{0}^{t\wedge \tau^{\epsilon}_{\widehat{R}}}\int_{\mathcal{Z}}\mathfrak{F}^{\varepsilon_{1}, \varepsilon_{2}}(s)\|G(u^{\varepsilon_{1}, \varepsilon_{2}}(s-),z)\|_{H}^2 N(ds,dz)\notag\\
		&+\widehat{\varepsilon}_2^2\int_{0}^{t\wedge \tau^{\epsilon}_{\widehat{R}}}\int_{\mathcal{Z}}\mathfrak{F}^{\varepsilon_{1}, \varepsilon_{2}}(s)\|G(u^{\varepsilon_{1}, \varepsilon_{2}}(s-),z)-G(u^{\widehat{\varepsilon}_1,\widehat{\varepsilon}_2}(s-),z)\|_{H}^2 N(ds,dz)\notag\\
		&+2\left(\varepsilon_2-\widehat{\varepsilon}_2\right)\int_{0}^{t\wedge \tau^{\epsilon}_{\widehat{R}}}\int_{\mathcal{Z}}\mathfrak{F}^{\varepsilon_{1}, \varepsilon_{2}}(s)\langle G(u^{\varepsilon_{1}, \varepsilon_{2}}(s-),z),u^{\varepsilon_{1}, \varepsilon_{2}}(s-)-u^{\widehat{\varepsilon}_1,\widehat{\varepsilon}_2}(s-)\rangle\widetilde{N}(ds,dz)\notag\\
		&+2\widehat{\varepsilon}_2\int_{0}^{t\wedge \tau^{\epsilon}_{\widehat{R}}}\int_{\mathcal{Z}}\mathfrak{F}^{\varepsilon_{1}, \varepsilon_{2}}(s)\langle G(u^{\varepsilon_{1}, \varepsilon_{2}}(s-),z)-G(u^{\widehat{\varepsilon}_1,\widehat{\varepsilon}_2}(s-),z),u^{\varepsilon_{1}, \varepsilon_{2}}(s-)-u^{\widehat{\varepsilon}_1,\widehat{\varepsilon}_2}(s-)\rangle\widetilde{N}(ds,dz)\notag\\
		&:=-\beta \int_{0}^{t\wedge \tau^{\epsilon}_{\widehat{R}}}\mathfrak{F}^{\varepsilon_{1}, \varepsilon_{2}}(s)\|u^{\varepsilon_{1}, \varepsilon_{2}}(s)\|_{V}^2\|u^{\varepsilon_{1}, \varepsilon_{2}}(s)-u^{\widehat{\varepsilon}_1,\widehat{\varepsilon}_2}(s)\|_{H}^2ds+ \sum_{i=1}^{8} \mathscr{J}_{i}(t\wedge \tau^{\epsilon}_{\widehat{R}}).
	\end{align}
	
Similar to \eqref{2DHD-equ-Fe5.8}, by Remark \ref{Remweee2.3} we infer that there exists $C_{\mathbf{c}_0,\mu}>0$ such that for all $t\in [0,T]$,
\begin{align}\label{minus-bin}
		\begin{split}
			\left|\mathscr{J}_{1}(t\wedge \tau^{\epsilon}_{\widehat{R}})\right|&\leq \mu \int_0^{t\wedge \tau^{\epsilon}_{\widehat{R}}} \mathfrak{F}^{\varepsilon_{1}, \varepsilon_{2}}(s)\|u^{\varepsilon_{1}, \varepsilon_{2}}(s)-u^{\widehat{\varepsilon}_1,\widehat{\varepsilon}_2}(s)\|_V^2ds \\
			&+ C_{\mathbf{c}_0,\mu} \int_0^{t\wedge \tau^{\epsilon}_{\widehat{R}}} \mathfrak{F}^{\varepsilon_{1}, \varepsilon_{2}}(s)\|u^{\varepsilon_{1}, \varepsilon_{2}} (s)\|_{V}^2\|u^{\varepsilon_{1}, \varepsilon_{2}}(s)-u^{\widehat{\varepsilon}_1,\widehat{\varepsilon}_2}(s)\|_{H}^2ds.
		\end{split}
\end{align}	
We choose $\beta= C_{\mathbf{c}_0,\mu}$, then it holds
\begin{align*}
	&-\beta \int_{0}^{t\wedge \tau^{\epsilon}_{\widehat{R}}}\mathfrak{F}^{\varepsilon_{1}, \varepsilon_{2}}(s)\|u^{\varepsilon_{1}, \varepsilon_{2}}(s)\|_{V}^2\|u^{\varepsilon_{1}, \varepsilon_{2}}(s)-u^{\widehat{\varepsilon}_1,\widehat{\varepsilon}_2}(s)\|_{H}^2ds+\left|\mathscr{J}_{1}(t\wedge \tau^{\epsilon}_{\widehat{R}})\right|\\
	& \leq \mu \int_0^{t\wedge \tau^{\epsilon}_{\widehat{R}}} \mathfrak{F}^{\varepsilon_{1}, \varepsilon_{2}}(s)\|u^{\varepsilon_{1}, \varepsilon_{2}}(s)-u^{\widehat{\varepsilon}_1,\widehat{\varepsilon}_2}(s)\|_V^2ds,
\end{align*}
which along with \eqref{2DHD-eq5.34} can infer that for all $t\in [0,T]$,
\begin{align}\label{2DHD-eq5.35}
	\mathbb{E}\left[\sup_{0\leq r\leq t} \mathfrak{F}^{\varepsilon_{1}, \varepsilon_{2}}(r\wedge \tau^{\epsilon}_{\widehat{R}})\|u^{\varepsilon_{1}, \varepsilon_{2}}(r\wedge \tau^{\epsilon}_{\widehat{R}})-u^{\widehat{\varepsilon}_1,\widehat{\varepsilon}_2}(r\wedge \tau^{\epsilon}_{\widehat{R}})\|_{H}^2\right]
	\leq  \sum_{i=2}^{8}\mathbb{E}\left[\sup_{0\leq r\leq t} \left|\mathscr{J}_{i}(r\wedge \tau^{\epsilon}_{\widehat{R}})\right|\right].
\end{align}

From \eqref{2DHD-eq5.30} and hypotheses $\mathbf{(C.2)}$ and $\mathbf{(C.3)}$, along with $\mathfrak{F}^{\varepsilon_{1}, \varepsilon_{2}}(\cdot)\in (0,1]$, we deduce that for all $t\in [0,T]$,
%Since $\mathfrak{F}^{\varepsilon_{1}, \varepsilon_{2}}(\cdot)\in (0,1]$, by the hypotheses $\mathbf{(C.2)}$, $\mathbf{(C.3)}$, we get that for all $t\in [0,T]$,
\begin{align}\label{2DHD-eq5.36}
	&~~\mathbb{E}\left[\sup_{0\leq r\leq t}\left|\mathscr{J}_{2}(r\wedge \tau^{\epsilon}_{\widehat{R}})\right|\right]\notag\\
	&\leq C_{L_{g},T}|\varepsilon_1-\widehat{\varepsilon}_1|^2 + L_{\widehat{R}}\int_{0}^{t}\mathbb{E}\left[\sup_{0\leq r \leq s} \mathfrak{F}^{\varepsilon_{1}, \varepsilon_{2}}(r\wedge \tau^{\epsilon}_{\widehat{R}})\|u^{\varepsilon_{1}, \varepsilon_{2}}(r\wedge \tau^{\epsilon}_{\widehat{R}})-u^{\widehat{\varepsilon}_1,\widehat{\varepsilon}_2}(r\wedge \tau^{\epsilon}_{\widehat{R}})\|_{H}^2\right]ds,
%	&\leq L_{g}|\varepsilon_1-\widehat{\varepsilon}_1|^2\left(T+\int_{0}^{t\wedge \tau^{\epsilon}_{\widehat{R}}}\|u^{\varepsilon_{1}, \varepsilon_{2}}(s)\|_{H}^2ds\right)+\varepsilon_0^2L_{\widehat{R}}\int_{0}^{t\wedge \tau^{\epsilon}_{\widehat{R}}}\mathfrak{F}^{\varepsilon_{1}, \varepsilon_{2}}(s)\|u^{\varepsilon_{1}, \varepsilon_{2}}(s)-u^{\widehat{\varepsilon}_1,\widehat{\varepsilon}_2}(s)\|_{H}^2ds\notag\\
%	&\leq L_{g}|\varepsilon_1-\widehat{\varepsilon}_1|^2\left(T+\int_{0}^{T}\|u^{\varepsilon_{1}, \varepsilon_{2}}(s)\|_{H}^2ds\right)+L_{\widehat{R}}\int_{0}^{t\wedge \tau^{\epsilon}_{\widehat{R}}}\mathfrak{F}^{\varepsilon_{1}, \varepsilon_{2}}(s)\|u^{\varepsilon_{1}, \varepsilon_{2}}(s)-u^{\widehat{\varepsilon}_1,\widehat{\varepsilon}_2}(s)\|_{H}^2ds,
\end{align}
and
\begin{align}\label{2DHD-eq5.37}
&~~\mathbb{E}\left[\sup_{0\leq r\leq t}\left|\mathscr{J}_{5}(r\wedge \tau^{\epsilon}_{\widehat{R}})\right| \right]+\mathbb{E}\left[\sup_{0\leq r\leq t}\left|\mathscr{J}_{6}(r\wedge \tau^{\epsilon}_{\widehat{R}})\right| \right]\notag\\
&\leq \left|\varepsilon_2-\widehat{\varepsilon}_2\right|^2\mathbb{E}\left[\int_{0}^{t}\mathfrak{F}^{\varepsilon_{1}, \varepsilon_{2}}(s)\int_{\mathcal{Z}}\|G(u^{\varepsilon_{1}, \varepsilon_{2}}(s-),z)\|_{H}^2 \nu(dz) ds\right]\notag\\
&+\widehat{\varepsilon}_0^2 \mathbb{E}\left[\int_{0}^{t\wedge \tau^{\epsilon}_{\widehat{R}}}\mathfrak{F}^{\varepsilon_{1}, \varepsilon_{2}}(s)\int_{\mathcal{Z}}\|G(u^{\varepsilon_{1}, \varepsilon_{2}}(s-),z)-G(u^{\widehat{\varepsilon}_1,\widehat{\varepsilon}_2}(s-),z)\|_{H}^2\nu(dz) ds\right]\notag\\
&\leq C_{L_{g},T}\left|\varepsilon_2-\widehat{\varepsilon}_2\right|^2+L_{\widehat{R}}\int_{0}^{t}\mathbb{E}\left[\sup_{0\leq r \leq s} \mathfrak{F}^{\varepsilon_{1}, \varepsilon_{2}}(r\wedge \tau^{\epsilon}_{\widehat{R}})\|u^{\varepsilon_{1}, \varepsilon_{2}}(r\wedge \tau^{\epsilon}_{\widehat{R}})-u^{\widehat{\varepsilon}_1,\widehat{\varepsilon}_2}(r\wedge \tau^{\epsilon}_{\widehat{R}})\|_{H}^2\right]ds.
\end{align}

Similar to \eqref{2DHD-equ-Fe5.10} and \eqref{2DHD-equ-Fe5.11}, by the BDG inequality and $\mathbf{(C.3)}$, we can obtain that for all $t\in [0,T]$,
\begin{align}\label{2DHD-eq5.38}
	&~~\mathbb{E}\left[\sup_{0\leq r\leq t}\left|\mathscr{J}_{3}(r\wedge \tau^{\epsilon}_{\widehat{R}})\right| \right]+\mathbb{E}\left[\sup_{0\leq r\leq t}\left|\mathscr{J}_{4}(r\wedge \tau^{\epsilon}_{\widehat{R}})\right| \right]\notag\\
	&\leq \frac{1}{4}\mathbb{E}\left[\sup_{0\leq r \leq t}{\mathfrak{F}^{\varepsilon_{1}, \varepsilon_{2}}(r\wedge \tau^{\epsilon}_{\widehat{R}})\|u^{\varepsilon_{1}, \varepsilon_{2}}(r\wedge \tau^{\epsilon}_{\widehat{R}})-u^{\widehat{\varepsilon}_1,\widehat{\varepsilon}_2}(r\wedge \tau^{\epsilon}_{\widehat{R}})\|_{H}^2}\right]\notag\\
	&+C|\varepsilon_1-\widehat{\varepsilon}_1|^2\mathbb{E}\left[\int_{0}^{t\wedge \tau^{\epsilon}_{\widehat{R}}}\mathfrak{F}^{\varepsilon_{1}, \varepsilon_{2}}(s)\|h\left(u^{\varepsilon_{1}, \varepsilon_{2}}(s)\right)\|_{\mathcal{L}_2(U;H)}^2ds\right]\notag\\
	&+C\widehat{\varepsilon}_0^2\mathbb{E}\left[\int_{0}^{t\wedge \tau^{\epsilon}_{\widehat{R}}}\mathfrak{F}^{\varepsilon_{1}, \varepsilon_{2}}(s)\|h\left(u^{\varepsilon_{1}, \varepsilon_{2}}(s)\right)-h(u^{\widehat{\varepsilon}_1,\widehat{\varepsilon}_2}(s))\|_{\mathcal{L}_2(U;H)}^2ds\right]\notag\\
	&\leq \frac{1}{4}\mathbb{E}\left[\sup_{0\leq r \leq t}{\mathfrak{F}^{\varepsilon_{1}, \varepsilon_{2}}(r\wedge \tau^{\epsilon}_{\widehat{R}})\|u^{\varepsilon_{1}, \varepsilon_{2}}(r\wedge \tau^{\epsilon}_{\widehat{R}})-u^{\widehat{\varepsilon}_1,\widehat{\varepsilon}_2}(r\wedge \tau^{\epsilon}_{\widehat{R}})\|_{H}^2}\right]+C_{L_{g},T}|\varepsilon_1-\widehat{\varepsilon}_1|^2\notag\\
	&+C_{L_{\widehat{R}}}\int_{0}^{t}\mathbb{E}\left[\sup_{0\leq r \leq s} \mathfrak{F}^{\varepsilon_{1}, \varepsilon_{2}}(r\wedge \tau^{\epsilon}_{\widehat{R}})\|u^{\varepsilon_{1}, \varepsilon_{2}}(r\wedge \tau^{\epsilon}_{\widehat{R}})-u^{\widehat{\varepsilon}_1,\widehat{\varepsilon}_2}(r\wedge \tau^{\epsilon}_{\widehat{R}})\|_{H}^2\right]ds,
\end{align}
and
\begin{align}\label{2DHD-eq5.39}
	&~~\mathbb{E}\left[\sup_{0\leq r\leq t}\left|\mathscr{J}_{7}(r\wedge \tau^{\epsilon}_{\widehat{R}})\right| \right]+\mathbb{E}\left[\sup_{0\leq r\leq t}\left|\mathscr{J}_{8}(r\wedge \tau^{\epsilon}_{\widehat{R}})\right| \right]\notag\\
	&\leq \frac{1}{4}\mathbb{E}\left[\sup_{0\leq r \leq t}{\mathfrak{F}^{\varepsilon_{1}, \varepsilon_{2}}(r\wedge \tau^{\epsilon}_{\widehat{R}})\|u^{\varepsilon_{1}, \varepsilon_{2}}(r\wedge \tau^{\epsilon}_{\widehat{R}})-u^{\widehat{\varepsilon}_1,\widehat{\varepsilon}_2}(r\wedge \tau^{\epsilon}_{\widehat{R}})\|_{H}^2}\right]\notag\\
	&+C|\varepsilon_2-\widehat{\varepsilon}_2|^2\mathbb{E}\left[\int_{0}^{t\wedge \tau^{\epsilon}_{\widehat{R}}}\int_{\mathcal{Z}}\mathfrak{F}^{\varepsilon_{1}, \varepsilon_{2}}(s)\|G(u^{\varepsilon_{1}, \varepsilon_{2}}(s-),z)\|_{H}^2\nu(dz)ds\right]\notag\\
	&+C\widehat{\varepsilon}_0^2\mathbb{E}\left[\int_{0}^{t\wedge \tau^{\epsilon}_{\widehat{R}}}\int_{\mathcal{Z}}\mathfrak{F}^{\varepsilon_{1}, \varepsilon_{2}}(s)\|G(u^{\varepsilon_{1}, \varepsilon_{2}}(s-),z)-G(u^{\widehat{\varepsilon}_1,\widehat{\varepsilon}_2}(s-),z)\|_{H}^2\nu(dz)ds\right]\notag\\
	&\leq \frac{1}{4}\mathbb{E}\left[\sup_{0\leq r \leq t}{\mathfrak{F}^{\varepsilon_{1}, \varepsilon_{2}}(r\wedge \tau^{\epsilon}_{\widehat{R}})\|u^{\varepsilon_{1}, \varepsilon_{2}}(r\wedge \tau^{\epsilon}_{\widehat{R}})-u^{\widehat{\varepsilon}_1,\widehat{\varepsilon}_2}(r\wedge \tau^{\epsilon}_{\widehat{R}})\|_{H}^2}\right]+C_{L_{g},T}\left|\varepsilon_2-\widehat{\varepsilon}_2\right|^2\notag\\
	&+C_{L_{\widehat{R}}}\int_{0}^{t}\mathbb{E}\left[\sup_{0\leq r \leq s} \mathfrak{F}^{\varepsilon_{1}, \varepsilon_{2}}(r\wedge \tau^{\epsilon}_{\widehat{R}})\|u^{\varepsilon_{1}, \varepsilon_{2}}(r\wedge \tau^{\epsilon}_{\widehat{R}})-u^{\widehat{\varepsilon}_1,\widehat{\varepsilon}_2}(r\wedge \tau^{\epsilon}_{\widehat{R}})\|_{H}^2\right]ds.
\end{align}

Substituting \eqref{2DHD-eq5.36}-\eqref{2DHD-eq5.39} into \eqref{2DHD-eq5.35}, we have
\begin{align}\label{2DHD-eq5.40}
	&~~\frac{1}{2}\mathbb{E}\left[\sup_{0\leq r\leq t} \mathfrak{F}^{\varepsilon_{1}, \varepsilon_{2}}(r\wedge \tau^{\epsilon}_{\widehat{R}})\|u^{\varepsilon_{1}, \varepsilon_{2}}(r\wedge \tau^{\epsilon}_{\widehat{R}})-u^{\widehat{\varepsilon}_1,\widehat{\varepsilon}_2}(r\wedge \tau^{\epsilon}_{\widehat{R}})\|_{H}^2\right]\notag\\
	&\leq C_{L_{g},T}\left(\left|\varepsilon_1-\widehat{\varepsilon}_1\right|^2+\left|\varepsilon_2-\widehat{\varepsilon}_2\right|^2\right)\notag\\
	&+C_{L_{\widehat{R}}}\int_{0}^{t}\mathbb{E}\left[\sup_{0\leq r \leq s} \mathfrak{F}^{\varepsilon_{1}, \varepsilon_{2}}(r\wedge \tau^{\epsilon}_{\widehat{R}})\|u^{\varepsilon_{1}, \varepsilon_{2}}(r\wedge \tau^{\epsilon}_{\widehat{R}})-u^{\widehat{\varepsilon}_1,\widehat{\varepsilon}_2}(r\wedge \tau^{\epsilon}_{\widehat{R}})\|_{H}^2\right]ds.
\end{align}
Using Gronwall's lemma to \eqref{2DHD-eq5.40} yields that for all $t\in [0,T]$,
\begin{align}\label{2DHD-eq5.41}
	\mathbb{E}\left[\sup_{0\leq r\leq t} \mathfrak{F}^{\varepsilon_{1}, \varepsilon_{2}}(r\wedge \tau^{\epsilon}_{\widehat{R}})\|u^{\varepsilon_{1}, \varepsilon_{2}}(r\wedge \tau^{\epsilon}_{\widehat{R}})-u^{\widehat{\varepsilon}_1,\widehat{\varepsilon}_2}(r\wedge \tau^{\epsilon}_{\widehat{R}})\|_{H}^2\right]\leq C_{L_{g},T}e^{C_{L_{\widehat{R}}} t}\left(\left|\varepsilon_1-\widehat{\varepsilon}_1\right|^2+\left|\varepsilon_2-\widehat{\varepsilon}_2\right|^2\right).
\end{align}
Since $t\rightarrow \frac{1}{\mathfrak{F}^{\varepsilon_{1}, \varepsilon_{2}}(t)}$ is increasing, and $\tau^{\epsilon}_{\widehat{R}}\geq T$ for any $\epsilon\in \Omega_{\epsilon}$, we know that for all $u_0\in \mathfrak{E}$ and $t\in [0,T]$,
\begin{align}\label{2DHD-eq5.42}
	\sup_{0\leq t\leq T}\frac{1}{\mathfrak{F}^{\varepsilon_{1}, \varepsilon_{2}}(t\wedge \tau^{\epsilon}_{\widehat{R}})}= e^{\beta \int_0^{T\wedge \tau^{\epsilon}_{\widehat{R}}} \|u^{\varepsilon_{1}, \varepsilon_{2}} (s)\|_{V}^2ds}\leq  e^{\beta \widehat{R}^2}.
\end{align}

It follows from \eqref{2DHD-eq5.41} and \eqref{2DHD-eq5.41} that  for all $u_0\in \mathfrak{E}$,
\begin{align}\label{2DHD-eqsfsf5.42}
		&~~\mathbb{E}\left[\sup_{0\leq t\leq T} \|u^{\varepsilon_{1}, \varepsilon_{2}}(t\wedge \tau^{\epsilon}_{\widehat{R}})-u^{\widehat{\varepsilon}_1,\widehat{\varepsilon}_2}(t\wedge \tau^{\epsilon}_{\widehat{R}})\|_{H}^2\right]\notag\\
		&\leq \mathbb{E}\left[\left(\sup_{0\leq t\leq T} \frac{1}{\mathfrak{F}^{\varepsilon_{1}, \varepsilon_{2}}(t\wedge \tau^{\epsilon}_{\widehat{R}})}\right) \left(\sup_{0\leq t\leq T} \mathfrak{F}^{\varepsilon_{1}, \varepsilon_{2}}(t\wedge \tau^{\epsilon}_{\widehat{R}})\|u^{\varepsilon_{1}, \varepsilon_{2}}(t\wedge \tau^{\epsilon}_{\widehat{R}})-u^{\widehat{\varepsilon}_1,\widehat{\varepsilon}_2}(t\wedge \tau^{\epsilon}_{\widehat{R}})\|_{H}^2\right)\right]\notag\\
		&\leq C_{L_{g},L_{\widehat{R}},\widehat{R},\mathbf{c}_0,\mu,T}\left(\left|\varepsilon_1-\widehat{\varepsilon}_1\right|^2+\left|\varepsilon_2-\widehat{\varepsilon}_2\right|^2\right).
\end{align}
By \eqref{2DHD-eqsfsf5.42} we have
\begin{align*}
	&\sup_{u_0\in \mathfrak{E}} \mathbb{P}\left( \left\{ \omega\in\Omega:   \sup_{0\leq t\leq T}\left\|u^{\varepsilon_{1}, \varepsilon_{2}}(t\wedge \tau^{\epsilon}_{\widehat{R}})-u^{\widehat{\varepsilon}_1,\widehat{\varepsilon}_2}(t\wedge \tau^{\epsilon}_{\widehat{R}} )\right\|_{H} \geq\varkappa      \right\} \right)\\
	&\leq \frac{1}{\varkappa^2}\mathbb{E}\left[\sup_{0\leq t\leq T} \|u^{\varepsilon_{1}, \varepsilon_{2}}(t\wedge \tau^{\epsilon}_{\widehat{R}})-u^{\widehat{\varepsilon}_1,\widehat{\varepsilon}_2}(t\wedge \tau^{\epsilon}_{\widehat{R}})\|_{H}^2\right]\\
	&\leq \frac{C_{L_{g},L_{\widehat{R}},\widehat{R},\mathbf{c}_0,\mu,T}}{\varkappa^2}\left(\left|\varepsilon_1-\widehat{\varepsilon}_1\right|^2+\left|\varepsilon_2-\widehat{\varepsilon}_2\right|^2\right)\rightarrow 0 \text{~as~} (\varepsilon_{1}, \varepsilon_{2}) \rightarrow (\widehat{\varepsilon}_1,\widehat{\varepsilon}_2),
\end{align*}
which implies \eqref{2DHD-eq5.33}.  Therefore, by \eqref{2DHD-eq5.32} and \eqref{2DHD-eq5.33} we can obtain the desired conclusion. This completes the proof.
\end{proof}	

We now investigate the limiting behavior of the invariant measures for system \eqref{2DHD-1.01} driven by L\'{e}vy noise with initial condition $u(0)=u_{0}$ 
as the noise intensities $(\varepsilon_{1,n}, \varepsilon_{2,n})\rightarrow (\widehat{\varepsilon}_1,\widehat{\varepsilon}_2)\in [0,\varepsilon_0]^2\subset  [0,1]^2$.
Let $\mathscr{S}^{\varepsilon_{1},\varepsilon_{2}}$ denote the collection of all invariant measures for this system. By Theorem \ref{2DHD-the5.3}, It is not difficult to find that $\mathscr{S}^{\varepsilon_{1},\varepsilon_{2}}$ is nonempty.
	\begin{theorem}\label{2DHD-the5.5}
		Under the hypotheses of Lemma $\ref{2DHD-lem5.1}$. If the noise intensities $(\varepsilon_{1,n}, \varepsilon_{2,n})\rightarrow (\widehat{\varepsilon}_1,\widehat{\varepsilon}_2)$, and $\tilde{\mu}^{\varepsilon_{1,n}, \varepsilon_{2,n}}\in \mathscr{S}^{\varepsilon_{1,n}, \varepsilon_{2,n}}$, then there exist a subsequence $\{\tilde{\mu}^{\varepsilon_{1,n_m}, \varepsilon_{2,n_m}}\}\subset \{\tilde{\mu}^{\varepsilon_{1,n}, \varepsilon_{2,n}}\}$ and an invariant measure $\tilde{\mu}^{\widehat{\varepsilon}_1,\widehat{\varepsilon}_2}\in \mathscr{S}^{\varepsilon_{1}, \varepsilon_{2}}$ such that $\tilde{\mu}^{\varepsilon_{1,n_m}, \varepsilon_{2,n_m}}\rightarrow \tilde{\mu}^{\widehat{\varepsilon}_1,\widehat{\varepsilon}_2}$ weakly.
	\end{theorem}
\begin{proof}
	In Theorem \ref{2DHD-the5.3}, we have proven that the family $\{\tilde{\mu}^{\varepsilon_{1,n}, \varepsilon_{2,n}}\}_{n=1}^\infty$ is tight, and hence there exists a subsequence $\{\tilde{\mu}^{\varepsilon_{1,n_m}, \varepsilon_{2,n_m}}\}_{m=1}^\infty$ such that the convergence $\tilde{\mu}^{\varepsilon_{1,n_m}, \varepsilon_{2,n_m}}\rightarrow \tilde{\mu}^{\widehat{\varepsilon}_1,\widehat{\varepsilon}_2}$ is weakly for some probability measure $\tilde{\mu}^{\tilde{\mu}^{\widehat{\varepsilon}_1,\widehat{\varepsilon}_2}}$. Consequently, as $({\varepsilon_{1,n_m}, \varepsilon_{2,n_m}}) \rightarrow ({\widehat{\varepsilon}_1,\widehat{\varepsilon}_2})$, by Lemma \ref{2DHD-lemma5.5} and \cite[Theorem $2.1$]{ChenDong2020} we can obtain that $\tilde{\mu}^{{\widehat{\varepsilon}_1,\widehat{\varepsilon}_2}}$ is an invariant measure of corresponding limit system, and $\tilde{\mu}^{{\widehat{\varepsilon}_1,\widehat{\varepsilon}_2}} \in \mathscr{S}^{{\widehat{\varepsilon}_1,\widehat{\varepsilon}_2}}$. This proof completes the proof.
	%{mftheim4.3}
\end{proof}

As a direct consequence of Theorems \ref{2DHD-the5.3} and \ref{2DHD-the5.5}, we have the following result.
\begin{theorem}\label{2DHD-the5.7}
	Assume that Hypotheses $\ref{2DHD-Ass2.0}$, $\ref{2DHD-Ass2.1}$, and $\ref{2DHD-Ass2.4}$ hold, and let ${\varepsilon_{1,n}, \varepsilon_{2,n}} \in (0,{\varepsilon}_0]$, ${\widehat{\varepsilon}_1,\widehat{\varepsilon}_2}\in [0,\varepsilon_0]$. If $\tilde{\mu}^{{\varepsilon_{1,n}, \varepsilon_{2,n}}}$ is the sequence of invariant measures of system \eqref{2DHD-1.01} driven by L\'{e}vy noise with initial condition $u(0)=u_{0}$, then $\tilde{\mu}^{\varepsilon_{1,n}, \varepsilon_{2,n}}$  weakly converges to $\tilde{\mu}^{\widehat{\varepsilon}_1,\widehat{\varepsilon}_2}$ as $(\varepsilon_{1,n}, \varepsilon_{2,n})\rightarrow (\widehat{\varepsilon}_1,\widehat{\varepsilon}_2)$.
\end{theorem}

\section{Pullback measure attractors}\label{2DHD-PMAs6}
In this section, we focus on the existence, uniqueness and asymptotic autonomy of pullback measure attractors of system (\ref{2DHD-1.1}). To this end, we first introduce some notations on the space of probability measures. 

$\bullet$ For the space $C_b(H)$, we endow with the supremum norm $\|\phi\|_{C_b}=\sup_{x\in H}|\phi(x)|$. Let $L_b(H)$ denote the space of all bounded Lipschitz functions on $H$, which consists of all functions $\phi\in C_b(H)$ such that
\begin{align*}
	\text{Lip}(\phi): = \sup\limits_{x_{1},x_{2} \in H,x_{1} \neq x_{2}}\ \frac{\left| \phi\left( x_{1} \right) - \phi\left( x_{2} \right) \right|}{{\| x_{1} - x_{2}\|}_{H}} < \infty,
\end{align*}
and the space $L_b(H)$ is endowed with the norm
$$
\|\phi\|_{L_b}=\|\phi\|_{C_b}+\text{Lip}(\phi),~~~\forall \phi\in L_b(H).
$$

$\bullet$ For given $\phi\in C_b(H)$ and $\tilde{\mu}\in \mathcal{P}(H)$, we define
$$
(\phi,\tilde{\mu}):=\int_{H} \phi(x)\tilde{\mu}(dx).
$$
The probability measure sequence $\{\tilde{\mu}_n\}_{n=1}^\infty\subset \mathcal{P}(H)$ is weakly convergent to $\tilde{\mu} \in \mathcal{P}(H)$, if for all $\phi\in C_b(H)$,
$
\lim\limits_{n\rightarrow \infty} \left(\phi,\tilde{\mu}_n\right)=\left(\phi,\tilde{\mu}\right)$.
We define the metric of $\mathcal{P}(H)$ by
$$
d_{\mathcal{P}(H)}\left( \tilde{\mu}_{1},\tilde{\mu}_{2} \right) = \sup\limits_{{\phi \in L_{b}{(H)}},\,{\|\phi\|_{L_b} \leq 1}}\ \left| \left( \phi,\tilde{\mu}_{1} \right) - \left( \phi,\tilde{\mu}_{2} \right) \right|,\quad\forall \tilde{\mu}_{1},\tilde{\mu}_{2} \in \mathcal{P}(H).
$$
It is known that $\left(\mathcal{P}(H),d_{\mathcal{P}(H)}\right)$ is a Polish space. For every $p\geq 1$, the space $\mathcal P_p \left( H \right)$ is defined by
It is known that $\left(\mathcal{P}(H),d_{\mathcal{P}(H)}\right)$ is a Polish space. For every $p\geq 1$, the space $\mathcal P_p \left( H \right)$ is defined by
$$
\mathcal P_p \left( H \right) = \left\{ {\tilde{\mu}  \in \mathcal
	P\left( H \right):
	\int_H {\|x\|_H^p \tilde{\mu} \left( {dx} \right) <    \infty } } \right\}.
$$
Recall that the Hausdorff semi-metric of between two nonempty sets $Y,Z\in \mathcal{P}_p(X)$, it is defined as follows
$$
d_{\mathcal{P}_{p}(X)}(Y,Z) = \sup\limits_{\mu_1 \in Y}\ \inf\limits_{\mu_2 \in Z}\ d_{\mathcal{P}(X)}(\mu_1,\mu_2).
$$
%Recall that the $p$-Wasserstein distance $\mathbb{W}_p$, $p\geq 1$, is given by
%$$
%\mathbb{W}_p ( \tilde{\mu} , \nu ) =
%\inf\limits_{ \pi \in \mathscr{C} ( \tilde{\mu}, \nu ) }
%\Big (
%\int_{ H \times H}
%\| x-y  \|_H^p \pi (dx, dy)
%\Big )^{ \frac{1}{p} },\quad
%\forall \tilde{\mu}, \nu \in \mathcal{P}_p ( H ),
%$$
%where $ \mathscr{C} ( \tilde{\mu}, \nu ) $ is the  set of all couplings of $\tilde{\mu}$ and $\nu$. Then, $(\mathcal P_p(H), \mathbb{W}_p )$ is 
%Polish space.

$\bullet$ Given $r>0$, we define the ball $\mathbb{B}_{\mathcal{P}_{p}(H)}(r)$ as
\begin{align*}
	\mathbb{B}_{\mathcal{P}_{p}(H)}(r) = \left\{ \tilde{\mu} \in \mathcal{P}_{p}(H):\left( \int_{H} \|x\|_{H}^{p}\tilde{\mu}(dx) \right)^{\frac{1}{p}} \leq r \right\}.
\end{align*}
A  subset $\mathfrak{B}\subseteq {\mathcal P}_p \left( H \right)$ is
bounded  if there exists $r>0$ such that
$\mathfrak{B}\subseteq \mathbb{B}_{\mathcal{P}_{p}(H)}(r)$. If   $\mathfrak{B}$ is bounded in   $ {\mathcal P}_p ( H )$,
then we set
$$
\|\mathfrak{B}\|_{{\mathcal P}_p ( H )}
=
\sup_{
	\tilde{\mu} \in \mathfrak{B}
} \left(\int_{H} \|x\|_{H}^p \tilde{\mu}
( {dx}  )\right)^{\frac{1}{p}}<\infty.
$$

\begin{remark}
Note that 
$(\mathcal {P}_p ({H}), d_{\mathcal{P}({H})})$ is a metric space but not complete. Since
for every $r>0$, $\mathbb{B}_ {\mathcal P_p({H})} (r)$ is a closed subset of  $\mathcal{P}({H})$ with respect to the metric  $d_{\mathcal{P}({H})}$, we know that
$(\mathbb{B}_ {\mathcal P_p({H})} (r), \ d_{\mathcal{P}({H})})$
is complete for every $r>0$.	
\end{remark}

Throughout this section, we need to make some additional hypotheses on the bilinear mapping $B$ and the nonlinear terms $h$, $G$, and on the forcing term, assuming that $f\in L_{loc}^2(\mathbb{R},H)$.

\begin{description}
	\item[(B.4)] For the bilinear mapping $B$, we  assume that there exists a constant $\mathbf{c}_0>0$ such that
	\begin{align*}
		\|B(u,v)\|_{H}^2 \leq \mathbf{c}_0\|u\|_{H} \|u\|_{V}\|\mathcal{A} v\|_{H} \|v\|_{V},
		\text{~  for all } u,v\in V.
	\end{align*}
	
	\item[(C.4)] There exists a positive constant $\widetilde{L}_{g}>0$ such that for all $u\in H$,		
	$$
	\int_{\mathcal{Z}}\|G(u,z)\|_{H}^{p}\nu(dz)\leq \widetilde{L}_{g}\left(1+\|u\|_{H}^{p}\right), \quad \forall p\geq 2.
	$$
	
	\item[(C.5)] There exists a positive constant $\widehat{L}_{gv}>0$ such that for all $t\in \mathbb{R}$ and $u\in V$,		
	$$
	\|h(t,u)\|^2_{\mathcal{L}_2(U;V)}
	+
	\int_{\mathcal{Z}}\|G(u,z)\|_{V}^2\nu(dz)\leq \widehat{L}_{gv}\left(1+\|u\|_{V}^2\right).
	$$
\end{description}

We point out that, based on {Hypotheses} $\ref{2DHD-Ass2.0}$, $\ref{2DHD-Ass2.1}$, and $\ref{2DHD-Ass2.4}$, along with the condition $f\in L_{\text{loc}}^2(\mathbb{R},H)$, the results of Theorem \ref{2DHD-the3.3} remain valid even when the additional, stronger assumptions $\mathbf{(B.4)}$, $\mathbf{(C.4)}$  and $\mathbf{(C.5)}$ are imposed. Specifically, the results under stronger assumptions can be stated as follows, see \cite[Theorem 1.2]{BreLZ-NA-2014}.

\begin{theorem}\label{2DHD-the3.3-0}
	Assume that the hypotheses of Theorem \ref{2DHD-the3.3} hold, and that $\mathbf{(B.4)}$, $\mathbf{(C.4)}$  and $\mathbf{(C.5)}$  are also satisfied.
	Then, for any $u_{\tau}\in L^{p}(\Omega,\mathscr{F}_{\tau};H)$, $p\geq 2$, and $f\in L_{\text{loc}}^{p}(\mathbb{R},H)$, there exists a $\varepsilon_0>0$ such that for any $\varepsilon_1, \varepsilon_2\in (0,\varepsilon_0]\subseteq (0,1]$,
	\eqref{2DHD-1.1} has a unique  solution  $\{u(t),t\in [\tau,\tau+T]\}$ in the sense of Definition \ref{2DHD-Def3.1}, and for any $T>0$, it holds
	\begin{eqnarray*}%\label{2DHD-eq3.0-00}
		\mathbb{ E}\left[\sup_{t\in [\tau, \tau+T]}\|u(t)\|_{H}^{p}\right]+\mathbb{ E}\left[\int_{\tau}^{\tau+T} \|u(t)\|_{H}^{p-2}\|u(t)\|_{V}^2d t\right]\leq C_T\left(1+\mathbb{ E} [\|u_{\tau}\|_{H}^{p}]+\int_{\tau}^{\tau+T}\|f(s)\|^{p}_{H}ds\right).
	\end{eqnarray*}
\end{theorem}

\subsection{Long-time uniform estimates of solutions}

We now establish some uniform moment estimates of (\ref{2DHD-1.1}), which is necessary for proving the existence and uniqueness of pullback measure attractors. 
%For any $p\geq 2$, we further assume that 
%\begin{align}\label{2DHD-eq5.1}
%	\nu>\frac{8\mathbf{c}_3}{p\lambda_1}.
%\end{align}
%From \eqref{2DHD-eq5.1}, it is easy to deduce that there exists a sufficiently small constant $\gamma\in (0,1)$ such that
%\begin{align}\label{2DHD-eq5.2}
%	\frac{\mu \lambda_1}{2}-\widetilde{\kappa}>0.
%\end{align}
Assume that the deterministic forcing term $f$ satisfies
\begin{align}\label{LSWs6.3}
	\int_{-\infty}^{\tau}e^{\widetilde{\kappa} s} \|f(s)\|_{H}^{4} ds<\infty, \quad \forall \tau\in\mathbb{R},
\end{align}
where $\widetilde{\kappa}\in (0,\frac{\mu \lambda_1}{2})$.
It follows from Taylor formula that there exists a constant $\mathbf{c}_p>0$ such that for every $p\geq 2$,
\begin{align}\label{LSWs3.meanin24}
	\left|\|\wp+\varsigma\|_{H}^{p}-\|\wp\|_{H}^{p}-{p}\|\wp\|_{H}^{{p}-2}\left(\wp,\varsigma\right)\right|\leq \mathbf{c}_p\left(\|\wp\|_{H}^{p-2}\|\varsigma\|_{H}^2+\|\varsigma\|_{H}^{p}\right), \quad \forall \wp,\varsigma \in H.
\end{align}
%we define $\mathscr{D}$ as
%\begin{align*}
%	{\mathscr{D}}=\big\{&{\mathcal{D}}=\left\{\mathcal{D}(\tau)\subseteq \mathcal{P}_4(H):  \mathcal{D}(\tau)\neq\emptyset \text{ bounded in } \mathcal{P}_4(H), \tau\in \mathbb{R}\right\}:\\
%	&\lim_{\tau\rightarrow -\infty} e^{2\gamma \tau}\|D(\tau)\|^4_{\mathcal{P}_4(H)}=0\big\}.
%\end{align*}
\begin{lemma}\label{2DHD-lemma6.3}
Assume that the hypotheses of Theorem \ref{2DHD-the3.3-0} and condition \eqref{LSWs6.3} hold. There exists a $\varepsilon_0=\min \left\{1,\sqrt{\frac{\mu\lambda_1}{12L_g+2\mathbf{c}_4L_g+8\mathbf{c}_4\widetilde{L}_g}}\right\}$ such that for all $\tau \in \mathbb{R}$, $t>1$ and $\varepsilon_1, \varepsilon_2 \in (0,\varepsilon_0]$, the solution $ u$ of system (\ref{2DHD-1.1}) satisfies
\begin{align}\label{LSWs6.4}
	\begin{split}
		&~~\int_{\tau-1}^{\tau}\mathbb{ E}\left[\|u(s,\tau-t,u_{\tau-t})\|_{V}^2\right]ds\\
		&\leq \rho_2 e^{-\widetilde{\kappa} t}\mathbb{ E}\left[\|u_{\tau-t}\|_{H}^2\right]+\rho_2\left(\int_{-\infty}^{\tau} e^{-\widetilde{\kappa} (\tau-s)}\|f(s)\|_{H}^{4}ds\right)^{1/2}+\rho_2,
	\end{split}
\end{align}
and
\begin{align}\label{LSWs6.5}
	&~\mathbb{ E}\left[\|u(\tau,\tau-t,u_{\tau-t})\|_{H}^4\right]+\int_{\tau-t}^{\tau}e^{-\widetilde{\kappa} (\tau-s)}\mathbb{ E}\left[\|u(s,\tau-t,u_{\tau-t})\|_{H}^2\|u(s,\tau-t,u_{\tau-t})\|_{V}^2\right]ds\notag\\
	&\leq \rho_3 e^{-\widetilde{\kappa} t}\mathbb{ E}\left[\|u_{\tau-t}\|_{H}^4\right]+\rho_3\int_{-\infty}^{\tau}e^{-\widetilde{\kappa} (\tau-s)} \|f(s)\|_{H}^{4} ds+\rho_3,
\end{align}
where $u_{\tau-t}\in L^{4}(\Omega,\mathscr{F}_{\tau-t};H)$, $\rho_2, \rho_3>0$ are the positive numbers independent of $\varepsilon_1, \varepsilon_2$, $\tau$, $t$ and $u_{\tau-t}$, and $\widetilde{\kappa}$ is the same number as in \eqref{LSWs6.3}.
\end{lemma}
\begin{proof}
	By \eqref{2DHD-eq4.8} and \eqref{2DHD-eq4.9}, we have
	\begin{align}\label{LSWs6.6}
		&~\mathbb{E}\left[\|u(\tau,\tau-t,u_{\tau-t})\|_{H}^2\right]+\frac{\mu}{2} \int_{\tau-t}^{\tau} e^{-\widetilde{\kappa} (\tau-s)}\mathbb{E}\left[\|u(s,\tau-t,u_{\tau-t})\|_{V}^2\right]ds\notag\\
		&+\left(\frac{\mu\lambda_1}{2}  -\widetilde{\kappa} \right)\int_{\tau-t}^{\tau}e^{-\widetilde{\kappa} (\tau-s)}\mathbb{E}\left[\|u(s,\tau-t,u_{\tau-t})\|_{H}^2\right]ds\notag \\
		&\leq e^{-\widetilde{\kappa} t}\mathbb{E}\left[\|u_{\tau-t}\|_{H}^2\right]+\frac{2}{\mu\lambda_1\widetilde{\kappa}^{1/2}}\left(\int_{-\infty}^{\tau} e^{-\widetilde{\kappa} (\tau-s)}\|f(s)\|_{H}^2ds\right)^{1/2} + \frac{\mu\lambda_1 }{2\widetilde{\kappa}}.
	\end{align}
	It is obvious that for all $t> 1$,
	\begin{align*}
	\int_{\tau-1}^{\tau}\mathbb{ E}\left[\|u(s,\tau-t,u_{\tau-t})\|_{V}^2\right]ds
	\leq e^{\widetilde{\kappa}}\int_{\tau-t}^{\tau} 
	e^{-\widetilde{\kappa} (\tau-s)}\mathbb{E}\left[\|u(s,\tau-t,u_{\tau-t})\|_{V}^2\right]ds, 
\end{align*}
which along with \eqref{LSWs6.6} can infer that \eqref{LSWs6.4}, we only need to set $\rho_2=\frac{e^{\widetilde{\kappa}}}{\min\left\{1,\frac{\mu}{2}\right\}}\max\left\{1,\frac{2}{\mu\lambda_1},\frac{\mu\lambda_1 \widetilde{\kappa}^{1/2}}{2\widetilde{\kappa}}\right\}$.

Next, we prove \eqref{LSWs6.5}. Applying It\^{o}'s formula with jump to the process $e^{\widetilde{\kappa} \tau}\|u(\tau,\tau-t,u_{\tau-t})\|_{H}^{4}$, we can obtain
\begin{align}\label{LSWs6.06}
	&~~\mathbb{E}\left[\|u(\tau,\tau-t,u_{\tau-t})\|_{H}^{4}\right]
	+2\mu \int_{\tau-t}^{\tau} e^{-\widetilde{\kappa} (\tau-s)}\mathbb{E}\left[\|u(s,\tau-t,u_{\tau-t})\|_{H}^2\|u(s,\tau-t,u_{\tau-t})\|_{V}^2\right]ds\notag\\
	&+\left(2{\mu\lambda_1}  -\widetilde{\kappa} \right)\int_{\tau-t}^{\tau}e^{-\widetilde{\kappa} (\tau-s)}\mathbb{E}\left[\|u(s,\tau-t,u_{\tau-t})\|_{H}^4\right]ds\notag \\
	&\leq e^{-\widetilde{\kappa} t}\mathbb{E}\left[\|u_{\tau-t}\|_{H}^4\right]+4\int_{\tau-t}^{\tau}e^{-\widetilde{\kappa} (\tau-s)}\mathbb{E}\left[\|u(s,\tau-t,u_{\tau-t})\|_{H}^2\left|\langle f(s),u(s,\tau-t,u_{\tau-t})\rangle\right|\right]ds\notag \\
	&+6\varepsilon_1^2 \int_{\tau-t}^{\tau}e^{-\widetilde{\kappa} (\tau-s)}\mathbb{E}\left[\|u(s,\tau-t,u_{\tau-t})\|_{H}^2\|h(s,u(s,\tau-t,u_{\tau-t}))\|_{\mathcal{L}_2(U;H)}^2\right]ds\notag \\
	&+\mathbb{E}\bigg[\int_{\tau-t}^{\tau}e^{-\widetilde{\kappa} (\tau-s)}  \int_{\mathcal{Z}} \|u(s-,\tau-t,u_{\tau-t})+\varepsilon_2 G(u(s-,\tau-t,u_{\tau-t}),z)\|_H^4-\|u(s-,\tau-t,u_{\tau-t})\|_{H}^4\notag \\
	&-4\varepsilon_2\|u(s-,\tau-t,u_{\tau-t})\|_{H}^2\langle u(s-,\tau-t,u_{\tau-t}),G(u(s-,\tau-t,u_{\tau-t}),z) \rangle N(ds,dz)\bigg]\notag \\
	&:=e^{-\widetilde{\kappa} t}\mathbb{E}\left[\|u_{\tau-t}\|_{H}^4\right]+\mathscr{J}_{9}+\mathscr{J}_{10}+\mathscr{J}_{11}.
\end{align}
By H\"{o}lder's inequality and Young inequality we have
\begin{align}\label{LSWs6.07}
	\mathscr{J}_{9} &\leq 4\int_{\tau-t}^{\tau}e^{-\widetilde{\kappa} (\tau-s)}\mathbb{E}\left[\|u(s,\tau-t,u_{\tau-t})\|_{H}^{3}\|f(s)\|_{H}\right]ds\notag\\
	&\leq \frac{\mu \lambda_1}{2}\int_{\tau-t}^{\tau}e^{-\widetilde{\kappa} (\tau-s)}\mathbb{E}\left[\|u(s,\tau-t,u_{\tau-t})\|_{H}^{4}\right]ds
	+\frac{216}{\mu^3\lambda_1^3}\int_{-\infty}^{\tau}e^{-\widetilde{\kappa} (\tau-s)}\|f(s)\|_{H}^{4}ds.
\end{align}
Combining with $\mathbf{(C.2)}$, $\mathbf{(C.4)}$ and \eqref{LSWs3.meanin24} yields
\begin{align}\label{LSWs6.08}
\mathscr{J}_{10}+&\mathscr{J}_{11}
\leq 6\varepsilon_1^2 \int_{\tau-t}^{\tau}e^{-\widetilde{\kappa} (\tau-s)}\mathbb{E}\left[\|u(s,\tau-t,u_{\tau-t})\|_{H}^2\|h(s,u(s,\tau-t,u_{\tau-t}))\|_{\mathcal{L}_2(U;H)}^2\right]ds\notag \\	
& +\mathbf{c}_4\int_{\tau-t}^{\tau}e^{-\widetilde{\kappa} (\tau-s)}  \mathbb{E}\bigg[\int_{\mathcal{Z}} \varepsilon_2^2\|u(s-,\tau-t,u_{\tau-t})\|_H^2\|G(u(s-,\tau-t,u_{\tau-t}),z)\|_{H}^2\notag \\
& \qquad \qquad\quad \qquad \qquad+\varepsilon_2^4\|G(u(s-,\tau-t,u_{\tau-t}),z)\|_{H}^{4}\nu(dz)\bigg] ds\notag \\
&\leq \left(6L_g+\mathbf{c}_4L_g+\mathbf{c}_4\widetilde{L}_g\right)\varepsilon_0^2\int_{\tau-t}^{\tau}e^{-\widetilde{\kappa} (\tau-s)}\mathbb{E}\left[\|u(s,\tau-t,u_{\tau-t})\|_{H}^4\right]ds\notag\\
&+L_g\left(6+\mathbf{c}_4\right)\varepsilon_0^2\int_{\tau-t}^{\tau}e^{-\widetilde{\kappa} (\tau-s)}\mathbb{E}\left[\|u(s,\tau-t,u_{\tau-t})\|_{H}^{2}\right]ds+\mathbf{c}_4\widetilde{L}_g \varepsilon_0^2\int_{\tau-t}^{\tau}e^{-\widetilde{\kappa} (\tau-s)}ds\notag \\
&\leq \left(12L_g+2\mathbf{c}_4L_g+\mathbf{c}_4\widetilde{L}_g\right)\varepsilon_0^2\int_{\tau-t}^{\tau}e^{-\widetilde{\kappa} (\tau-s)}\mathbb{E}\left[\|u(s,\tau-t,u_{\tau-t})\|_{H}^4\right]ds+\frac{L_g\left(6+\mathbf{c}_4\right)+4{c}_4\widetilde{L}_g}{4\widetilde{\kappa}}\varepsilon_0^2\notag\\
&\leq \mu\lambda_1\int_{\tau-t}^{\tau}e^{-\widetilde{\kappa} (\tau-s)}\mathbb{E}\left[\|u(s,\tau-t,u_{\tau-t})\|_{H}^4\right]ds
+\mu\lambda_1\frac{L_g\left(6+\mathbf{c}_4\right)+4{c}_4\widetilde{L}_g}{4\widetilde{\kappa}\left(12L_g+2\mathbf{c}_4L_g+8\mathbf{c}_4\widetilde{L}_g\right)}\notag\\
&= \mu\lambda_1\int_{\tau-t}^{\tau}e^{-\widetilde{\kappa} (\tau-s)}\mathbb{E}\left[\|u(s,\tau-t,u_{\tau-t})\|_{H}^4\right]ds
+\frac{\mu\lambda_1}{8\widetilde{\kappa}}.
\end{align}
where we use the fact $\varepsilon_1,\varepsilon \in(0,\varepsilon_0]\subseteq (0,1]$.
\end{proof}

By \eqref{LSWs6.06}-\eqref{LSWs6.08} and $\widetilde{\kappa}\in (0,\frac{\mu \lambda_1}{2})$ we have
\begin{align}
	&~~\mathbb{E}\left[\|u(\tau,\tau-t,u_{\tau-t})\|_{H}^{4}\right]
	+2\mu \int_{\tau-t}^{\tau} e^{-\widetilde{\kappa} (\tau-s)}\mathbb{E}\left[\|u(s,\tau-t,u_{\tau-t})\|_{H}^2\|u(s,\tau-t,u_{\tau-t})\|_{V}^2\right]ds\notag\\
	&\leq e^{-\widetilde{\kappa} t}\mathbb{E}\left[\|u_{\tau-t}\|_{H}^4\right]+\frac{216}{\mu^3\lambda_1^3}\int_{-\infty}^{\tau}e^{-\widetilde{\kappa} (\tau-s)}\|f(s)\|_{H}^{4}ds+\frac{\mu\lambda_1}{8\widetilde{\kappa}}.
\end{align}
Let $\rho_3=\frac{1}{\min\left\{1,2\mu\right\}}\max\left\{1,\frac{216}{\mu^3\lambda_1^3},\frac{\mu\lambda_1}{8\widetilde{\kappa}}\right\}$, then the desired estimate \eqref{LSWs6.5} can be obtained. This completes the proof.
%$\varepsilon_0\leq \sqrt{\frac{\mu\lambda_1}{12L_g+2\mathbf{c}_4L_g+8\mathbf{c}_4\widetilde{L}_g}}$

\begin{lemma}\label{2DHD-lemma6.4}
	Under the hypotheses of Lemma \ref{2DHD-lemma6.3}. There exists a $$
	\varepsilon_0=\min\left\{1,\sqrt{\frac{\mu\lambda_1}{2\widehat{L}_{gv}}},\sqrt{\frac{\mu\lambda_1}{12L_g+2\mathbf{c}_4L_g+8\mathbf{c}_4\widetilde{L}_g}}\right\}
	$$
	 such that for all $\tau \in \mathbb{R}$, $t>1$ and $\varepsilon_1, \varepsilon_2 \in (0,\varepsilon_0]$, the solution $ u$ of system (\ref{2DHD-1.1}) satisfies
	\begin{align*}%\label{LSWs6.10}
			&~\mathbb{E}\left[\mathfrak{P}(\tau,\tau-t,u_{\tau-t})\| u(\tau,\tau-t,u_{\tau-t})\|_V^2\right]\\
			&\leq \rho_4 e^{-\widetilde{\kappa} t}\mathbb{ E}\left[\|u_{\tau-t}\|_{H}^2\right]+\rho_4\left(\int_{-\infty}^{\tau} e^{-\widetilde{\kappa} (\tau-s)}\|f(s)\|_{H}^{4}ds\right)^{1/2}+\rho_4
		\end{align*}
		with
		\begin{align}\label{NSE-disghjmathfrak{F}}
			\mathfrak{P}(\tau,\tau-t,u_{\tau-t})=e^{-\frac{27
					\mathbf{c}_0^2}{2\mu^3}\int_{\tau-t}^{\tau} \|u(s,\tau-t,u_{\tau-t})\|_H^2 \|u(s,\tau-t,u_{\tau-t})\|_V^2ds},
		\end{align}
	where $u_{\tau-t}\in L^{4}(\Omega,\mathscr{F}_{\tau-t};H)$, $\rho_4>0$ is a positive numbers independent of $\varepsilon_1, \varepsilon_2$, $\tau$, $t$ and $u_{\tau-t}$, and $\widetilde{\kappa}$ is the same number as in \eqref{LSWs6.3}.
\end{lemma}
\begin{proof}
Applying It\^{o}'s formula with jump to $\| u(\tau,\tau-t,u_{\tau-t})\|_V^2$, we can obtain that for all $\tau\in \mathbb{R}$, $t>1$ and $\varrho \in (\tau-1,\tau)$,	
\begin{align}\label{2DHD-eq6.10}
		&\,\, \| u(\tau,\tau-t,u_{\tau-t})\|_V^2
		+2\mu\int_\varrho^\tau \|\mathcal{A} u(s,\tau-t,u_{\tau-t})\|_{H}^2 ds\notag \\
		&=\| u(\varrho,\tau-t,u_{\tau-t})\|_V^2-2\int_\varrho^\tau \langle B(u(s,\tau-t,u_{\tau-t}),u(s,\tau-t,u_{\tau-t})), \mathcal{A}u(s,\tau-t,u_{\tau-t})\rangle ds\notag \\
		&+2\int_\varrho^\tau\langle f(s),\mathcal{A}u(s,\tau-t,u_{\tau-t})\rangle ds+2\varepsilon_1\int_\varrho^\tau \langle \mathcal{A}u(s,\tau-t,u_{\tau-t}),h(s,u(s,\tau-t,u_{\tau-t}))\rangle d W(s)\notag \\
		&+2\varepsilon_2\int_\varrho^\tau \int_{\mathcal{Z}}\langle \mathcal{A}u(s-,\tau-t,u_{\tau-t}),G(u(s-,\tau-t,u_{\tau-t}),z)\rangle \widetilde{N}(ds,dz)\notag \\
		&+\varepsilon_1^2 \int_\varrho^\tau\|h(s,u(s,\tau-t,u_{\tau-t}))\|_{\mathcal{L}_2(U;V)}^2ds+\varepsilon_2^2 \int_\varrho^\tau \int_{\mathcal{Z}}\|G(u(s-,\tau-t,u_{\tau-t}),z)\|_{V}^2N(ds,dz).
\end{align}
By H\"{o}lder's inequality, Young's inequality and $\mathbf{(C.5)}$, for all $\varrho \in (\tau-1,\tau)$, we have
\begin{align*}
	&2\mathbb{E}\left[\int_\varrho^\tau\langle f(s),\mathcal{A}u(s,\tau-t,u_{\tau-t})\rangle ds\right]\leq \frac{\mu}{2}\int_\varrho^\tau \mathbb{E}\left[\|\mathcal{A} u(s,\tau-t,u_{\tau-t})\|_{H}^2\right] ds+ \frac{2}{\mu} \int_\varrho^\tau \|f(s)\|_{H}^2ds,
\end{align*}
and
\begin{align*}
		&~~\varepsilon_1^2 \mathbb{E}\left[\int_\varrho^\tau\|h(s,u(s,\tau-t,u_{\tau-t}))\|_{\mathcal{L}_2(U;V)}^2ds\right]+\varepsilon_2^2 \mathbb{E}\left[\int_\varrho^\tau \int_{\mathcal{Z}}\|G(u(s-,\tau-t,u_{\tau-t}),z)\|_{V}^2N(ds,dz)\right]\notag\\
	&\leq \widehat{L}_{gv} \varepsilon_0^2\int_\varrho^\tau \mathbb{E}\left[\left(1+\|u(s,\tau-t,u_{\tau-t})\|_{V}^2\right)\right]\leq \frac{\mu}{2}\int_\varrho^\tau \mathbb{E}\left[\|\mathcal{A} u(s,\tau-t,u_{\tau-t})\|_{H}^2\right] ds+\frac{\mu\lambda_1}{2}.
\end{align*}
For the second term on the right-hand side of \eqref{2DHD-eq6.10}, by $\mathbf{(B.4)}$ we infer
\begin{align*}
		&-2\mathbb{E}\left[\int_\varrho^\tau \langle B(u(s,\tau-t,u_{\tau-t}),u(s,\tau-t,u_{\tau-t})), \mathcal{A}u(s,\tau-t,u_{\tau-t})\rangle ds\right]\\
		&\leq 2\int_\varrho^\tau \mathbb{E}\left[\|B(u(s,\tau-t,u_{\tau-t}),u(s,\tau-t,u_{\tau-t}))\|_{H}\|\mathcal{A} u(s,\tau-t,u_{\tau-t})\|_{H}\right] ds\\
		&\leq 2\sqrt{\mathbf{c}_0}\int_\varrho^\tau \mathbb{E}\left[\|u(s,\tau-t,u_{\tau-t})\|_H^{1/2}\| u(s,\tau-t,u_{\tau-t})\|_V \|\mathcal{A} u(s,\tau-t,u_{\tau-t})\|_{H}^{3/2}\right]ds\\
		&\leq \frac{\mu}{2}\int_\varrho^\tau \mathbb{E}\left[\|\mathcal{A}u(s,\tau-t,u_{\tau-t})\|_{H}^2\right]ds+\frac{27
			\mathbf{c}_0^2}{2\mu^3}\int_\varrho^\tau \mathbb{E}\left[\|u(s,\tau-t,u_{\tau-t})\|_H^2 \| u(s,\tau-t,u_{\tau-t})\|_V^{4}\right]ds.
\end{align*}

Therefore, by taking the expectation of \eqref{2DHD-eq6.10}, for all $\tau\in \mathbb{R}$, $t>1$ and $\varrho \in (\tau-1,\tau)$, we get
\begin{align}\label{2DHD-eq6.11}
	&\,\, \mathbb{E}\left[\mathfrak{P}(\tau,\tau-t,u_{\tau-t})\| u(\tau,\tau-t,u_{\tau-t})\|_V^2\right]
	\notag \\
	&\leq \mathbb{E}\left[\mathfrak{P}(\varrho,\tau-t,u_{\tau-t})\| u(\varrho,\tau-t,u_{\tau-t})\|_V^2\right]+ \frac{2}{\mu} \int_\varrho^\tau \|f(s)\|_{H}^2ds+\frac{\mu\lambda_1}{2}.
\end{align}
Integrating \eqref{2DHD-eq6.11} on $\varrho$ from $\tau-1$ to $\tau$, we have
\begin{align*}
	&\,\, \mathbb{E}\left[\mathfrak{P}(\tau,\tau-t,u_{\tau-t})\| u(\tau,\tau-t,u_{\tau-t})\|_V^2\right]
	\notag \\
	&\leq \int_{\tau-1}^\tau\mathbb{E}\left[\mathfrak{P}(\varrho,\tau-t,u_{\tau-t})\| u(\varrho,\tau-t,u_{\tau-t})\|_V^2\right]d\varrho+ \frac{2e^{\widetilde{\kappa}}}{\mu} \int_{\tau-1}^\tau e^{-\widetilde{\kappa}(\tau-s)}\|f(s)\|_{H}^2ds+\frac{\mu\lambda_1}{2}\\
	&\leq \int_{\tau-1}^\tau\mathbb{E}\left[\| u(\varrho,\tau-t,u_{\tau-t})\|_V^2\right]d\varrho+ \frac{2e^{\widetilde{\kappa}}}{\mu} \int_{-\infty}^\tau e^{-\widetilde{\kappa}(\tau-s)}\|f(s)\|_{H}^2ds+\frac{\mu\lambda_1}{2},
\end{align*}
which together with \eqref{LSWs6.4} derives
\begin{align*}
&\,\, \mathbb{E}\left[\mathfrak{P}(\tau,\tau-t,u_{\tau-t})\| u(\tau,\tau-t,u_{\tau-t})\|_V^2\right]
\notag \\
&\leq \rho_2 e^{-\widetilde{\kappa} t}\mathbb{ E}\left[\|u_{\tau-t}\|_{H}^2\right]+\left(\rho_2+\frac{2e^{\widetilde{\kappa}}}{\mu\widetilde{\kappa}^{1/2}}\right)\left(\int_{-\infty}^{\tau} e^{-\widetilde{\kappa} (\tau-s)}\|f(s)\|_{H}^{4}ds\right)^{1/2}+\rho_2
+\frac{\mu\lambda_1}{2}.
\end{align*}
Let $\rho_4=\max\left\{\rho_2+\frac{2e^{\widetilde{\kappa}}}{\mu\widetilde{\kappa}^{1/2}},\rho_2
+\frac{\mu\lambda_1}{2}\right\}$, then the desired result holds. This completes the proof.
\end{proof}

\subsection{Existence of pullback measure attractors}
To prove the existence of $\mathscr{D}$-pullback measure attractors of system (\ref{2DHD-1.1}), we first need to construct a non-autonomous dynamical system $\Psi$ on $(\mathcal{P}_4(H),d_{\mathcal{P}(H)})$. Then we prove the existence of a closed $\mathscr{D}$-pullback absorbing set and finally establish the $\mathscr{D}$-pullback asymptotic compactness of $\Psi$ in $(\mathcal{P}_4(H),d_{\mathcal{P}(H)})$. We will also retain the notation for probability measures and transition semigroups introduced in Section \ref{Probledist}, precisely, the transition semigroup is denoted by
$(p_{\tau,t}\phi)(u_{\tau})=\mathbb{E}\left[\phi(u(t,\tau,u_{\tau}))\right]$ for any $\phi\in C_b(H)$, and $p(\tau,u_{\tau};t,A)=p_{\tau,t}\chi_{A} (u_{\tau} )$ for any $A\in \mathcal{B}(H)$. %$(p(t,s)\phi)(u_{\tau}):=(p_{s,t}\phi)(u_{\tau})=\mathbb{E}\left[\phi(u(t,s,u_{\tau}))\right]$ for any $\phi\in C_b(H)$, and $p(s,u_{\tau};t,A)=p_{s,t}\chi_{A} (u_{\tau} )$ for any $A\in \mathcal{B}(H)$.

According to  Lemma \ref{2DHD-lem5.2} and Theorem \ref{2DHD-the5.3}, we can also prove that the family $\{p_{\tau,t}\}_{t\geq \tau}$ has the following properties.
\begin{lemma}\label{MarkovFell}
	Assume that the hypotheses of Theorem \ref{2DHD-the3.3-0} and \eqref{LSWs6.3} hold. Then,
	
	$(i)$ the family $\{p_{\tau,t}\}_{t\geq \tau}$ is Feller, i.e., for any $\tau\leq t$, if $\phi\in C_b({H})$, then $p_{\tau,t}\phi\in C_b({H})$;
	
	$(ii)$ for any $\tau\in \mathbb{R}$ and $u_{\tau}\in {H}$, the stochastic process $\{u(t,\tau,u_{\tau})\}_{t\geq \tau}$ is a ${H}$-valued Markov process. 
\end{lemma}

For any $t\geq \tau$, the dual operator  $p^*_{\tau,t}:\mathcal{P}({H})\rightarrow \mathcal{P}({H})$ of $p_{\tau,t}$ is defined as follows:
$$
p^*_{\tau,t}\tilde{\mu}(\cdot)=\int_{{H}} p(\tau,u_{\tau};t,\cdot) \tilde{\mu}(du_{\tau}), \quad \forall \tilde{\mu}\in \mathcal{P}({H}).
$$
From Theorem \ref{2DHD-the3.3-0} we see that $p^*_{\tau,t}$ maps $\mathcal{P}_4({H})$ to $\mathcal{P}_4({H})$ for any $t\geq \tau$. Now, we may define a non-autonomous dynamical system $\Psi(t,\tau)$, $t\geq \tau$. Given $t\in \mathbb{R}^+$ and $\tau\in \mathbb{R}$, let the mapping $\Psi(t,\tau):\mathcal{P}_4({H})\rightarrow \mathcal{P}_4({H})$ be given by
\begin{align}\label{SLRSSs3.19}
	\Psi(t,\tau) \tilde{\mu}= p^*_{\tau,\tau+t}\tilde{\mu},\quad \forall \tilde{\mu}\in \mathcal{P}_4({H}).
\end{align}

We can prove that $\Psi(t,\tau)$, $t \geq \tau$, constitutes a continuous non-autonomous dynamical system on $(\mathcal{P}_4({H}),d_{\mathcal{P}({H})})$. Specifically, the corresponding results stated below.
\begin{lemma}\label{MarkovFell***}
	Assume that the hypotheses of Theorem \ref{2DHD-the3.3-0} and \eqref{LSWs6.3} hold. Then $\Psi(t,\tau), t\geq \tau,$ is a continuous non-autonomous dynamical system associated with (\ref{2DHD-1.1}). More precisely, $\Psi(t,\tau):\mathcal{P}_4({H})\rightarrow \mathcal{P}_4({H})$ satisfies that for any $\tau \in \mathbb{R}$,
	
	$(i)$  $\Psi(0,\tau)=I_{\mathcal{P}_{4}({H})}$;
	
	$(ii)$ $\Psi(t+s,\tau)=\Psi(t,s+\tau)\Psi(s,\tau)$ for any $\tau \in \mathbb{R}$ and $s, t \in \mathbb{R}^+$;
	
	$(iii)$ $\Psi(t,\tau): \mathcal{P}_{4}({H})\rightarrow \mathcal{P}_{4}({H})$ is continuous for any $\tau \in \mathbb{R}$ and $t\in \mathbb{R}^+$. 
\end{lemma}   
\begin{proof}
	Note that $\Psi(t,\tau)=p^*_{\tau,\tau+t}$.
	$(i)$ By \eqref{SLRSSs3.19} we have $\Psi(0,\tau)\tilde{\mu}=p^*_{\tau,\tau}\tilde{\mu}=\tilde{\mu}$ for any $\tau \in \mathbb{R}$ and $\tilde{\mu}\in \mathcal{P}_4({H})$, which implies $\Psi(0,\tau)=I_{\mathcal{P}_{4}({H})}$. 
	$(ii)$ From the Markov property of solutions given in Lemma \ref{MarkovFell}, it follows that $p_{\tau,s}p_{s,t}=p_{\tau,t}$ for any $t\geq s \geq \tau$. Then, we have
	$\Psi(t,s+\tau)\Psi(s,\tau) =p^*_{s+\tau,\tau+s+t}p^*_{\tau,\tau+s}=\left(p_{\tau,\tau+s} p_{s+\tau,\tau+s+t}\right)^*=p_{\tau,\tau+s+t}^*=\Psi(t+s,\tau)$.
	$(iii)$ Suppose that $\tilde{\mu}_n\rightarrow \tilde{\mu}$ in $\mathcal{P}_4({H})$, it suffices to prove that  $\Psi(t,\tau) \tilde{\mu}_n\rightarrow \Psi(t,\tau) \tilde{\mu}$ in $(\mathcal{P}_4({H}),d_{\mathcal{P}({H})})$. By the Feller property of $\{p_{\tau,\tau+t}\}_{t\in \mathbb{R}^+}$ established in Lemma \ref{MarkovFell}, we find that  $p_{\tau,\tau+t}\phi\in C_b({H})$ for all $\phi\in C_b({H})$. Hence, by \eqref{SLRSSs3.19} we infer that for all $\tau\in \mathbb{R}$ and $t\in \mathbb{R}^+$,
	\begin{align*}
		\begin{split}
			\lim_{n\rightarrow \infty}\left(\phi,\Psi(t,\tau)\tilde{\mu}_n\right)
			=\lim_{n\rightarrow \infty}\left(p_{\tau,\tau+t}\phi,\tilde{\mu}_n\right)
			=\left(p_{\tau,\tau+t}\phi,\tilde{\mu}\right)
			=&\left(\phi,\Psi(t,\tau)\tilde{\mu}\right),
		\end{split}
	\end{align*}
	as desired. This completes the proof.
\end{proof}

To construct a closed $\mathscr{D}$-pullback absorbing set. Let $\mathscr{D}$ be the collection of families of bounded nonempty subsets of $\mathcal{P}_4(H)$, it is given by
\begin{align*}
	\mathscr{D}=\big\{&D=\left\{D(\tau)\subseteq \mathcal{P}_4(H): \emptyset\neq D(\tau) \text{ bounded in } \mathcal{P}_4(H), \tau\in \mathbb{R}\right\}:\lim_{\tau\rightarrow -\infty} e^{\widetilde{\kappa} \tau}\|D(\tau)\|^2_{\mathcal{P}_4(H)}=0\big\},
\end{align*} 
here $\widetilde{\kappa}>0$ is a constant appearing in \eqref{LSWs6.3}.

\begin{corollary}\label{2DHD-cor6.7}
	Under the hypotheses of Lemma \ref{2DHD-lemma6.3}. For every $\tau \in \mathbb{R}$, $D=\{D(t):t\in \mathbb{R}\}\in \mathscr{D}$, there exists $\varepsilon_0=\min\left\{1,\sqrt{\frac{\mu\lambda_1}{2\widehat{L}_{gv}}},\sqrt{\frac{\mu\lambda_1}{12L_g+2\mathbf{c}_4L_g+8\mathbf{c}_4\widetilde{L}_g}}\right\}$ and $T:=T(\tau,D)> 1$ such that for all $t\geq T$ and $\varepsilon_1, \varepsilon_2 \in (0,\varepsilon_0]$, the solution $ u$ of system (\ref{2DHD-1.1}) satisfies
	\begin{align}\label{LSWs6.-005}
		&~\mathbb{ E}\left[\|u(\tau,\tau-t,u_{\tau-t})\|_{H}^4\right]\leq \rho_5 +\rho_5\int_{-\infty}^{\tau}e^{-\widetilde{\kappa} (\tau-s)} \|f(s)\|_{H}^{4} ds,
	\end{align}
	where $u_{\tau-t}\in L^{4}(\Omega,\mathscr{F}_{\tau-t};H)$ with $\mathscr{L}(u_{\tau-t}) \in D(\tau-t)$, and $\rho_5>0$ is a constant independent of $u_{\tau-t}$, $\varepsilon_1$, $\varepsilon_{2}$, $\tau$ and $D$.
\end{corollary}
\begin{proof}
	Thanks to $\mathscr{L}(u_{\tau-t}) \in D(\tau-t)$, then we have
	$$
	\rho_3 e^{-\widetilde{\kappa} t}\mathbb{ E}\left[\|u_{\tau-t}\|_{H}^2\right] \leq
	\rho_3 e^{-\widetilde{\kappa} t}\|D(\tau-t)\|_{H}^2 \rightarrow 0, \text{ as } t\rightarrow \infty,
	$$
	which implies that there exists $T:=T(\tau,D)>1$ such that for any $t\geq T$,
	\begin{align}\label{SLRSSs3.9}
		\begin{split}
			\rho_3 e^{-\widetilde{\kappa} t}\mathbb{ E}\left[\|u_{\tau-t}\|_{H}^2\right]
			\leq \rho_3. 
		\end{split}
	\end{align}
	By \eqref{LSWs6.5} and \eqref{SLRSSs3.9}, let $\rho_5=2\rho_3$, we can infer the \eqref{LSWs6.-005}. The proof is finished.
\end{proof}

\begin{lemma}\label{sDHD-lemma6.8}
	Under the hypotheses of Lemma \ref{2DHD-lemma6.3}. For given $\tau \in \mathbb{R}$, let
	$$
	\mathscr{M}(\tau):=\rho_5 +\rho_5\int_{-\infty}^{\tau}e^{-\widetilde{\kappa} (\tau-s)} \|f(s)\|_{H}^{4} ds,
	$$
	and denote $\mathcal{K}(\tau)=\mathbb{B}_{\mathcal{P}_4(H)}\left(\sqrt[4]{\mathscr{M}(\tau)}\right)$, where $\rho_5$ is from Corollary \ref{2DHD-cor6.7}. Then 
	$\mathcal{K}=\left\{\mathcal{K}(\tau):\tau\in \mathbb{R}\right\}\in \mathscr{D}$ is a closed $\mathscr{D}$-pullback absorbing set for $\Psi$ on $(\mathcal{P}_4(H),d_{\mathcal{P}(H)})$.
\end{lemma}
\begin{proof}
	By Corollary \ref{2DHD-cor6.7} and the definition of $\mathscr{M}(\tau)$, we know that for any $\tau \in \mathbb{R}$ and $D=\{D(t):t\in \mathbb{R}\}\in \mathscr{D}$, there exists $T:=T(\tau,D)>1$ such that for any $t\geq T$, 
	$$
	\Psi(t,\tau-t)D(\tau-t)\subset \mathcal{K}(\tau).
	$$ 
	It remains to verify that  $\mathcal{K}=\{\mathcal{K}(\tau):\tau\in \mathbb{R}\}$ belongs to the universe $\mathscr{D}$. From \eqref{LSWs6.3}, we can deduce
	\begin{align*}
		&\lim_{\tau\rightarrow -\infty} e^{\widetilde{\kappa} \tau}\|\mathcal{K}(\tau)\|_{\mathcal{P}_4(H)}^4
		=\lim_{\tau\rightarrow -\infty} e^{\widetilde{\kappa} \tau} \mathscr{M}(\tau)\\
		&\leq \lim_{\tau\rightarrow -\infty} e^{\widetilde{\kappa} \tau} \rho_5 +\lim_{\tau\rightarrow -\infty} \rho_5\int_{-\infty}^{\tau}e^{\widetilde{\kappa} s} \|f(s)\|_{H}^{4} ds=0,
	\end{align*}
	which shows $\mathcal{K}=\{\mathcal{K}(\tau):\tau\in \mathbb{R}\}\in \mathscr{D}$. This completes the proof.
\end{proof}

Next, we verify the $\mathscr{D}$-pullback asymptotic compactness of $\Psi$ in $(\mathcal{P}_4(H),d_{\mathcal{P}(H)})$. 
\begin{lemma}\label{sDHD-lemma6.9}
	Under the hypotheses of Lemma \ref{2DHD-lemma6.3}. The non-autonomous dynamical system $\Psi$ is $\mathscr{D}$-pullback asymptotically compact in $(\mathcal{P}_4(H),d_{\mathcal{P}(H)})$; precisely, for any $\tau \in \mathbb{R}$,  $t_n \to +\infty$ and $\tilde{\mu}_n\in D(\tau-t_n)$ with $D\in \mathscr{D}$, $\left\{\Psi\left(t_n, \tau-t_n\right) \tilde{\mu}_n \right\}_{n=1}^\infty$ has a convergent subsequence in $(\mathcal{P}_4(H),d_{\mathcal{P}(H)})$.
\end{lemma}
\begin{proof}
%	To achieve our goal, it suffices, by Prohorov's theorem, to prove that the sequence of distributions $\{\mathscr{L}{u(\tau,\tau-t_n,u_{\tau-t_n})}\}_{n=1}^\infty$ is tight in $H$.  
	
	Using similar the argument of \eqref{LSWs6.-005}, by Lemma \ref{2DHD-lemma6.4} we can obtain that for every $\tau \in \mathbb{R}$, $D=\{D(t):t\in \mathbb{R}\}\in \mathscr{D}$, there exists  $T:=T(\tau,D)> 1$ such that for all $t\geq T$,
	$$
	\mathbb{E}\left[\mathfrak{P}(\tau,\tau-t,u_{\tau-t})\| u(\tau,\tau-t,u_{\tau-t})\|_V^2\right]
	\leq \rho_6 +\rho_6\left(\int_{-\infty}^{\tau} e^{-\widetilde{\kappa} (\tau-s)}\|f(s)\|_{H}^{4}ds\right)^{1/2},
	$$
	where $\rho_6$ is independent of $u_{\tau-t}$, $\varepsilon_1$, $\varepsilon_{2}$, $\tau$ and $D$. It follows that there exists $N_1=N_1(\tau,{D})\in \mathbb{N}$ such that for all $n\geq N_1$,  
	\begin{align}\label{NSE-disghjmea5.8}
		\mathbb{E}\left[\mathfrak{P}(\tau,\tau-2,u(\tau-2,\tau-t_n,u_{\tau-t_n}))\| u(\tau,\tau-2,u(\tau-2,\tau-t_n,u_{\tau-t_n}))\|_V^2\right]\leq \mathscr{M}_1,
	\end{align}
	where $\mathscr{M}_1:=\mathscr{M}_1(\tau)>0$ is a constant depending only on $\tau$, but not on $n$ or ${D}$. 
	
	Following an argument analogous to that for \eqref{LSWs6.4} and using Lemma \ref{2DHD-lemma6.3}, we obtain  that there exist $\mathscr{M}_2:=\mathscr{M}_2(\tau)>0$ and $N_2=N_2(\tau,{D})\in \mathbb{N}$ such that for all $n\geq N_2$,
	\begin{align}\label{NSE-disghjmea5.9}
		\begin{split}
			&\int_{\tau-2}^\tau \mathbb{E}\left[\|u(s,\tau-t_n,u_{\tau-t_n})\|_H^2\|u(s,\tau-t_n,u_{\tau-t_n})\|_V^2\right]ds\leq \mathscr{M}_2.
		\end{split}
	\end{align}
	
	For any $R>2$, let $\mathbb{B}_{V}(R):=\{u:\|u\|_{V}\leq R\}$. Since $V$ is compactly embedded in $H$, the set $\mathbb{B}_{V}(R)$ is compact in $H$.
	By \eqref{NSE-disghjmea5.8}, \eqref{NSE-disghjmea5.9} and Chebyshev's inequality, we obtain that there exists $N_3=\max\{N_1,N_2\}$ such that for  all $n\geq N_3$ and $u_{\tau-t_n}\in H$, 
		\begin{align}\label{ojodfkpdfkp}
			%\begin{split}
			&\mathbb{P}\left(\|u(\tau,\tau-t_n,u_{\tau-t_n})\|_V\notin \mathbb{B}_{V}(R)\right)=\mathbb{P}\left(\|u(\tau,\tau-2, u(\tau-2, \tau-t_n, u_{\tau-t_n}))\|_V>{R}\right)\notag \\
			&\leq  \mathbb{P}\left(\mathfrak{P}^{1/2}(\tau,\tau-2,u(\tau-2, \tau-t_n, u_{\tau-t_n}))\|u(\tau,\tau-2,u(\tau-2, \tau-t_n, u_{\tau-t_n}))\|_V>{R}^{1/2}\right)\notag \\
			&\quad+\mathbb{P}\left(\mathfrak{P}^{-1/2}(\tau,\tau-2,u(\tau-2, \tau-t_n, u_{\tau-t_n}))>{R}^{1/2}\right)\notag \\
			&\leq \frac{\mathbb{E}\left(\mathfrak{P}(\tau,\tau-2,u(\tau-2, \tau-t_n, u_{\tau-t_n}))\|u(\tau,\tau-2,u(\tau-2, \tau-t_n, u_{\tau-t_n}))\|_V^2\right)}{{R}}\notag \\
			&\quad+\mathbb{P}\left(\int^\tau_{\tau-2}\|u(s,\tau-2,u(\tau-2, \tau-t_n, u_{\tau-t_n}))\|_H^2\|u(s,\tau-2,u(\tau-2, \tau-t_n, u_{\tau-t_n}))\|_V^2ds>\frac{2\mu^3\ln{R}}{27
				\mathbf{c}_0^2}\right)\notag \\
			&\leq \frac{\mathbb{E}\left[\mathfrak{P}(\tau,\tau-2,u(\tau-2, \tau-t_n, u_{\tau-t_n}))\|u(\tau,\tau-2,u(\tau-2, \tau-t_n, u_{\tau-t_n}))\|_V^2\right]}{{R}}\notag \\
			&\quad +\frac{27\mathbf{c}_0^2}{2\mu^3\ln{R}}\int^\tau_{\tau-2}\mathbb{E}\left(\|u(s,\tau-t_n,u_{\tau-t_n})\|_H^2\|u(s,\tau-t_n,u_{\tau-t_n})\|_V^2\right)ds\notag \\
			&\leq \frac{\mathscr{M}_1}{{R}}+\frac{27\mathbf{c}_0^2\mathscr{M}_2}{2\mu^3\ln{{R}}}\rightarrow 0 \quad \text{as~~} {R}\rightarrow \infty.
			%\end{split}
		\end{align}
	 By \eqref{ojodfkpdfkp}, for any $\tau\in \mathbb{R}$ and $\epsilon>0$, there exists $\widetilde{R}:=\widetilde{R}(\tau,\epsilon)>2$ such that for all $n\geq N_3$ and $\tilde{\mu}_n \in \mathcal{K}(\tau-t_n)$,
	$$
	\left(\Psi\left(t_n, \tau-t_n\right) \tilde{\mu}_n\right)(H\backslash \mathbb{B}_{V}(R) )=\int_{H} \mathbb{P}\left(\|u(\tau,\tau-t_n,u_{\tau-t_n})\|_V\notin \mathbb{B}_{V}(R)\right)\tilde{\mu}_n(du_{\tau-t_n})< \epsilon.
	$$
	Therefore, $\{\Psi\left(t_n, \tau-t_n\right) \tilde{\mu}_n:n\geq N_3\}$ is precompact, which along with the compactness of the finite set $\{\Psi\left(t_n, \tau-t_n\right) \tilde{\mu}_n:n< N_3\}$ concludes that the whole sequence $\left\{\Psi\left(t_n, \tau-t_n\right) \tilde{\mu}_n \right\}_{n=1}^\infty$ is precompact in $(\mathcal{P}_4(H),d_{\mathcal{P}(H)})$. This completes the proof.
\end{proof}

Finally, we establish the existence and uniqueness of $\mathscr{D}$-pullback measure attractors of system (\ref{2DHD-1.1}) in $(\mathcal{P}_4(H), d_{\mathcal{P} (H)})$.

\begin{theorem}\label{sDHD-th6.10}
	Assume that the hypotheses of Theorem \ref{2DHD-the3.3-0} and condition \eqref{LSWs6.3} hold. Then the non-autonomous dynamical system $\Psi$ generated by (\ref{2DHD-1.1}) possesses a unique $\mathscr{D}$-pullback measure attractor $\mathscr{A}=\{\mathscr{A}(\tau):\tau\in \mathbb{R}\}\in \mathscr{D}$ in $\mathcal {P}_4 (H)$. Moreover, for every $\tau\in \mathbb{R}$, $\mathscr{A}(\tau)$ has the following characterizations:
	\begin{align*}
		\begin{split}
			\mathscr{A}(\tau)
			=\omega(\mathcal{K},\tau)&=\{\psi(0,\tau): \psi \text{ is a } \mathscr{D}\text{-complete orbit of } \Psi\}\\
			&=\{\widehat{u}(\tau): \widehat{u} \text{ is a }  \mathscr{D}\text{-complete solution of } \Psi\},
		\end{split}
	\end{align*}
	where $\mathcal{K}=\{\mathcal{K}(\tau):\tau\in \mathbb{R}\}$ denotes the $\mathscr{D}$-pullback absorbing set of $\Psi$ provided by Lemma \ref{sDHD-lemma6.8}.
\end{theorem}
\begin{proof}
	On the one hand, we know from Lemma \ref{MarkovFell***} that $\Psi$ is a continuous non-autonomous dynamical system on  $\mathcal {P}_4 (H)$. On the other hand, Lemmas \ref{sDHD-lemma6.8} and \ref{sDHD-lemma6.9} imply that $\Psi$ has a closed $\mathscr{D}$-pullback absorbing set $\mathcal{K}=\left\{\mathcal{K}(\tau):\tau\in \mathbb{R}\right\}$ and is $\mathscr{D}$-pullback asymptotically compact in $\mathcal{P}_4(H)$. Therefore, by \cite[Proposition 2.10]{LDS-JDE-2024} (or \cite[Proposition 3.6]{WBX-JDE-2012}) we immediately conclude that $\Psi$ possesses a unique $\mathscr{D}$-pullback measure attractor. This completes the proof.
\end{proof}

%\subsection{Upper semicontinuity of pullback measure attractors}
%In this section, we consider the upper semicontinuity of 

\subsection{Asymptotic autonomy of pullback measure attractors}
In this section, we consider the asymptotically autonomous robustness of pullback measure attractors $\mathscr{A}=\{\mathscr{A}(\tau):\tau\in \mathbb{R}\}$ in Theorem \ref{sDHD-th6.10} when the time parameter $\tau$ tends to negative infinity. To achieve this goal, we introduce the following limiting system
\begin{align}\label{2DHD-asy1.1}
	\left\{
\begin{aligned}
		& d v(t)+\mu \mathcal{A} v(t)d t+B(v(t),v(t))d t
		= f_{\infty} dt+ \varepsilon_1 h_{\infty}(v(t))d W(t)+\varepsilon_2\int_{\mathcal{Z}}G(v(t-),z)\widetilde{N}(dt,dz),
		\\
		&v(0)=v_{0}.
	\end{aligned}
\right.
\end{align}
For the time-dependent forcing term $f$ and the time-independent forcing term $f_{\infty}$, we assume that
\begin{align}\label{2DHD-asy1.2}
\lim\limits_{\tau\to -\infty} \int_{-\infty}^{\tau} \|f(s)-f_{\infty}\|_{H}^2 ds=0.
\end{align}
Moreover, we assume that for every $r>0$, there exists a positive constant
$L_{r}$ such that, for all $t\in \mathbb{R}$, $u,v\in H$ with $\|u\|_{H}\vee \|v\|_{H}\leq r$, $h$ and $h_{\infty}$ satisfy
\begin{align}\label{2DHD-asy1.3}
	\|h(t,u)-h_{\infty}(v)\|_{\mathcal{L}_2(U;H)}^2
	+
	\int_\mathcal{Z}\|G(u,z)-G(v,z)\|_{H}^2\nu(dz)\leq L_{r}\|u-v\|_{H}^2.
\end{align}

Similar to Theorem \ref{2DHD-the3.3-0}, it is clear that for any $v_{0}\in L^{p}(\Omega,\mathscr{F}_{\tau};H)$, $p\geq 2$, and $f_\infty\in H$, there exists a $\varepsilon_0>0$ such that for any $\varepsilon_1, \varepsilon_2\in (0,\varepsilon_0]\subseteq (0,1]$ and $T>0$, the system \eqref{2DHD-asy1.1} has a unique solution $v\in \mathcal{D}([0,T],H)\cap L^2([0,T],V)$, $\mathbb{P}$-almost surely, and the solution $v$ satisfies
\begin{eqnarray}\label{2DHD-asy1.4}
	\mathbb{ E}\left[\sup_{t\in [0, T]}\|v(t)\|_{H}^{p}\right]+\mathbb{ E}\left[\int_{0}^{T} \|v(t)\|_{H}^{p-2}\|v(t)\|_{V}^2d t\right]\leq C_T\left(1+\|f_{\infty}\|_{H}^p+\mathbb{ E} [\|v_{0}\|_{H}^{p}]\right).
\end{eqnarray}

By virtue of the Feller and Markov properties established in Lemma \ref{MarkovFell}, and following a procedure analogous to the definition in \eqref{SLRSSs3.19} and the proof of Lemma \ref{MarkovFell***}, we can show that system \eqref{2DHD-asy1.1} generates a continuous dynamical system $S(t): \mathcal{P}_4(H)\to \mathcal{P}_4(H)$ defined by
$$
\int_{H}\phi(v_0)\left(S(t)\tilde{\mu}\right)(dv_0)=\int_{H} \mathbb{E}\left[\phi(v(t,v_0))\right]\tilde{\mu}(dv_0) , \quad \forall t\in \mathbb{R}^+, \phi\in C_b(H), \tilde{\mu} \in\mathcal{P}_4(H).
$$
As a consequence of Theorem \ref{sDHD-th6.10}, we deduce that system \eqref{2DHD-asy1.1} possesses a unique measure attractor $\mathscr{A}_{\infty}$ in $(\mathcal {P}_4 (H), d_{\mathcal{P}(H)})$, which is compact, invariant, and attracts the bounded
sets in $H$ (see\cite{Mar-EJP-1998} for the related concept).

Next, we derive the following uniform convergence between the solutions of \eqref{2DHD-1.1} and \eqref{2DHD-asy1.1}
\begin{lemma}\label{2DHD-lemma6.11}
	Assume that the hypotheses of Theorem \ref{2DHD-the3.3-0} and \eqref{2DHD-asy1.2}-\eqref{2DHD-asy1.3} hold. Then, there exists a $\varepsilon_{0}\in (0,1]$ such that for any bounded set $\mathfrak{E}\subset H$, $T>0$, $\varkappa>0$  and $\varepsilon_{1}, \varepsilon_{2}\in (0,\varepsilon_{0}]$,
	\begin{align*}
		\lim\limits_{\tau_n\rightarrow -\infty}\sup_{\varsigma\in \mathfrak{E}}\mathbb{P}\left(\left\{\omega \in \Omega: \sup_{0\leq t\leq T}\|u(t+\tau_{n},\tau_{n},\varsigma)-v(t,\varsigma)\|_{H}\geq \varkappa\right\}\right)=0.
	\end{align*}
\end{lemma}
\begin{proof}
	The proof follows the same methodology as in Lemma \ref{2DHD-lemma5.5}, combined with conditions \eqref{2DHD-asy1.2} and \eqref{2DHD-asy1.3}. 
	Define the stopping time 
	$$
	\zeta_{\widehat{R}}=\inf \left\{t\in [0,T]: \|u(t+\tau_n,\tau_n,\varsigma)\|_{H}>\widehat{R}, \text{ or } \|v(t,\varsigma)\|_{H}>\widehat{R}, \text{ or } \int_{0}^{T} \|v(s,\varsigma)\|_{V}^2ds>\widehat{R}^2 \right\},
	$$	
	and let $\mathcal{Y}_{n}^{\varsigma}(t)=u(t+\tau_{n},\tau_n,\varsigma)-v(t,\varsigma)$.
	It suffices to prove that for any $\varkappa>0$, as $\tau_{n}\rightarrow -\infty$,
	\begin{align}\label{2DHD-eq6.323}
		\sup_{\varsigma\in \mathfrak{E}}\mathbb{P}\left( \left\{ \omega\in\Omega:   \sup_{0\leq t\leq T}\left\|\mathcal{Y}_{n}^{\varsigma}(t\wedge \zeta_{\widehat{R}})\right\|_{H} \geq\varkappa      \right\} \right)\rightarrow 0.
	\end{align}
	
	Let $\mathfrak{F}(t):=e^{-C_{\mathbf{c}_0,\mu} \int_0^{t} \|v(s,\varsigma)\|_{V}^2ds}$ for all $t\in [0,T]$ and $\varsigma \in \mathfrak{E}$, where $C_{\mathbf{c}_0,\mu}>0$ is a constant that appears in the subsequent inequality \eqref{minus-bin1}.  By \eqref{2DHD-1.1} and \eqref{2DHD-asy1.1} we have
	\begin{align}\label{2DHD-eq6.324}
		\begin{split}
			&~d\mathcal{Y}_{n}^{\varsigma}(t)+\mu\mathcal{A}\mathcal{Y}_{n}^{\varsigma}(t)dt+\left(B(u(t+\tau_{n},\tau_n,\varsigma),\mathcal{Y}_{n}^{\varsigma}(t))+B(\mathcal{Y}_{n}^{\varsigma}(t),v(t,\varsigma))\right)dt\\
			&=(f(t+\tau_{n})-f_{\infty})dt+\varepsilon_{1} \left(h(t+\tau_n,u(t+\tau_{n},\tau_n,\varsigma))-h_{\infty}(v(t,\varsigma))\right)dW(t)\\
			&+\varepsilon_{2} \int_{\mathcal{Z}} \left(G(u((t+\tau_{n})-,\tau_n,\varsigma)),z)-G(v(t-,\varsigma),z)\right)\widetilde{N}(dt,dz).
		\end{split}
	\end{align}
	Applying It\^{o}'s formula to $\mathfrak{F}(t)\|\mathcal{Y}_{n}^{\varsigma}(t)\|_{H}^2$, we can get that for all $t\in [0,T]$,
	\begin{align*}
		&~~\mathfrak{F}(t\wedge \zeta_{\widehat{R}})\|\mathcal{Y}_{n}^{\varsigma}(t\wedge \zeta_{\widehat{R}})\|_{H}^2+2\mu \int_{0}^{t\wedge \zeta_{\widehat{R}}}\mathfrak{F}(s)\|\mathcal{Y}_{n}^{\varsigma}(s)\|_{V}^2ds \notag\\
		&=-C_{\mathbf{c}_0,\mu} \int_0^{t\wedge \tau^{\epsilon}_{\widehat{R}}} \mathfrak{F}(s)\|v(s,\varsigma))\|_{V}^2\|\mathcal{Y}_{n}^{\varsigma}(s)\|_{H}^2ds\notag\\
		&-2 \int_{0}^{t\wedge \zeta_{\widehat{R}}}\mathfrak{F}(s) \langle B(\mathcal{Y}_{n}^{\varsigma}(s),v(s,\varsigma)),
		\mathcal{Y}_{n}^{\varsigma}(s)\rangle ds+2 \int_{0}^{t\wedge \zeta_{\widehat{R}}}\mathfrak{F}(s) \langle f(s+\tau_{n})-f_{\infty}, \mathcal{Y}_{n}^{\varsigma}(s)\rangle ds\notag\\
		&+2\varepsilon_{1} \int_{0}^{t\wedge \zeta_{\widehat{R}}}\mathfrak{F}(s) \langle
		\mathcal{Y}_{n}^{\varsigma}(s),h(s+\tau_n,u(s+\tau_{n},\tau_n,\varsigma))-h_{\infty}(v(s,\varsigma))\rangle dW(s)\notag\\
		&+\varepsilon_{1}^2 \int_{0}^{t\wedge \zeta_{\widehat{R}}}\mathfrak{F}(s)\|h(s+\tau_n,u(s+\tau_{n},\tau_n,\varsigma))-h_{\infty}(v(s,\varsigma))\|_{\mathcal{L}_2(U;H)}^2 ds\notag\\
		&+2\varepsilon_{2} \int_{0}^{t\wedge \zeta_{\widehat{R}}}\int_{\mathcal{Z}} \mathfrak{F}(s) \langle
		\mathcal{Y}_{n}^{\varsigma}(s),G(u((s+\tau_{n})-,\tau_n,\varsigma)),z)-G(v(s-,\varsigma),z)\rangle \widetilde{N}(dt,dz)\notag\\
		&+\varepsilon_{2}^2\int_{0}^{t\wedge \zeta_{\widehat{R}}}\int_{\mathcal{Z}} \mathfrak{F}(s)\|G(u((s+\tau_{n})-,\tau_n,\varsigma)),z)-G(v(s-,\varsigma),z)\|_{H}^2N(ds,dz)\notag\\
		&:=-C_{\mathbf{c}_0,\mu} \int_0^{t\wedge \tau^{\epsilon}_{\widehat{R}}} \mathfrak{F}(s)\|v(s,\varsigma))\|_{V}^2\|\mathcal{Y}_{n}^{\varsigma}(s)\|_{H}^2ds+\sum_{i=12}^{17} \mathscr{J}_{i}(t\wedge \tau^{\epsilon}_{\widehat{R}}),
	\end{align*}
	which, along with Poincar\'{e}'s inequality and the estimate
	\begin{align}\label{minus-bin1}
		\left| \mathscr{J}_{12}(r\wedge \tau^{\epsilon}_{\widehat{R}})\right| \leq \mu \int_0^{t\wedge \tau^{\epsilon}_{\widehat{R}}} \mathfrak{F}(s)\|\mathcal{Y}_{n}^{\varsigma}(s)\|_V^2ds+ C_{\mathbf{c}_0,\mu} \int_0^{t\wedge \tau^{\epsilon}_{\widehat{R}}} \mathfrak{F}(s)\|v(s,\varsigma))\|_{V}^2\|\mathcal{Y}_{n}^{\varsigma}(s)\|_{H}^2ds,
	\end{align}	
	can deduce that for all $t\in [0,T]$,
	{\small
		\begin{align}\label{2DHD-eq6.326}
		\mathbb{E}\left[\sup_{0\leq r \leq t} \mathfrak{F}(r\wedge \zeta_{\widehat{R}})\|\mathcal{Y}_{n}^{\varsigma}(r\wedge \zeta_{\widehat{R}})\|_{H}^2\right]+\mu\lambda_1 \mathbb{E}\left[\int_{0}^{t\wedge \zeta_{\widehat{R}}}\mathfrak{F}(s)\|\mathcal{Y}_{n}^{\varsigma}(s)\|_{H}^2ds\right] \leq \sum_{i=13}^{17}\mathbb{E}\left[\sup_{0\leq r \leq t} \left| \mathscr{J}_{i}(r\wedge \tau^{\epsilon}_{\widehat{R}})\right|\right].
%		&\leq2 \mathbb{E}\left[\int_{0}^{t\wedge \zeta_{\widehat{R}}}\mathfrak{F}(s) \left|\langle B(\mathcal{Y}_{n}^{\varsigma}(s),v(s,\varsigma)),
%		\mathcal{Y}_{n}^{\varsigma}(s)\rangle \right|ds\right]+2 \mathbb{E}\left[\int_{0}^{t\wedge \zeta_{\widehat{R}}}\mathfrak{F}(s) \left|\langle f(s+\tau_{n})-f_{\infty}, \mathcal{Y}_{n}^{\varsigma}(s)\rangle \right|ds\right]\notag\\
	\end{align}
}	

We now deal with the right-hand terms of \eqref{2DHD-eq6.326}. By H\"{o}lder's inequality and Young's inequality we obtain that for all $t\in [0,T]$,
\begin{align}\label{2DHD-eq6.327}
	\mathbb{E}\left[\sup_{0\leq r \leq t} \left| \mathscr{J}_{13}(r\wedge \tau^{\epsilon}_{\widehat{R}})\right|\right]\leq \frac{\mu \lambda_1}{2}\mathbb{E}\left[\int_{0}^{t\wedge \zeta_{\widehat{R}}}\mathfrak{F}(s)\|\mathcal{Y}_{n}^{\varsigma}(s)\|_{H}^2ds\right]
	+\frac{2}{\mu\lambda_1}\int_0^{T}\|f(s+\tau_{n})-f_{\infty}\|_{H}^2ds.
\end{align}
By the BDG inequality and \eqref{2DHD-asy1.3}, we have
\begin{align}\label{2DHD-eq6.328}
	\begin{split}
		&~\mathbb{E}\left[\sup_{0\leq r \leq t} \left| \mathscr{J}_{15}(r\wedge \tau^{\epsilon}_{\widehat{R}})\right|\right]+\mathbb{E}\left[\sup_{0\leq r \leq t} \left| \mathscr{J}_{17}(r\wedge \tau^{\epsilon}_{\widehat{R}})\right|\right]\\
		&\leq
		L_{\widehat{R}}\int_0^{t}\mathbb{E}\left[\sup_{0\leq r \leq s} \mathfrak{F}(r\wedge \zeta_{\widehat{R}})\|\mathcal{Y}_{n}^{\varsigma}(r\wedge \zeta_{\widehat{R}})\|_{H}^2\right]ds,
	\end{split}
\end{align}	
and	
\begin{align}\label{2DHD-eq6.329}
	\begin{split}
		&~~\mathbb{E}\left[\sup_{0\leq r \leq t} \left| \mathscr{J}_{14}(r\wedge \tau^{\epsilon}_{\widehat{R}})\right|\right]+\mathbb{E}\left[\sup_{0\leq r \leq t} \left| \mathscr{J}_{16}(r\wedge \tau^{\epsilon}_{\widehat{R}})\right|\right]\\
		&\leq \frac{1}{2}\mathbb{E}\left[\sup_{0\leq r \leq t} \mathfrak{F}(r\wedge \zeta_{\widehat{R}})\|\mathcal{Y}_{n}^{\varsigma}(r\wedge \zeta_{\widehat{R}})\|_{H}^2\right]+
		C_{L_{\widehat{R}}}\int_0^{t}\mathbb{E}\left[\sup_{0\leq r \leq s} \mathfrak{F}(r\wedge \zeta_{\widehat{R}})\|\mathcal{Y}_{n}^{\varsigma}(r\wedge \zeta_{\widehat{R}})\|_{H}^2\right]ds.
	\end{split}
\end{align}			
Combining with \eqref{2DHD-eq6.326}-\eqref{2DHD-eq6.329} yields		
	\begin{align}\label{2DHD-eq6.330}
		\begin{split}
			&~~\frac{1}{2}\mathbb{E}\left[\sup_{0\leq r \leq t} \mathfrak{F}(r\wedge \zeta_{\widehat{R}})\|\mathcal{Y}_{n}^{\varsigma}(r\wedge \zeta_{\widehat{R}})\|_{H}^2\right]\\
			& \leq C_{L_{\widehat{R}}}\int_0^{t}\mathbb{E}\left[\sup_{0\leq r \leq s} \mathfrak{F}(r\wedge \zeta_{\widehat{R}})\|\mathcal{Y}_{n}^{\varsigma}(r\wedge \zeta_{\widehat{R}})\|_{H}^2\right]ds+ \frac{2}{\mu\lambda_1}\int_0^{T}\|f(s+\tau_{n})-f_{\infty}\|_{H}^2ds.
		\end{split}
	\end{align}
	
Using Gronwall's lemma for \eqref{2DHD-eq6.330}, we infer that for all $n\in \mathbb{N}$ and $\varsigma \in \mathfrak{E}$,	
	\begin{align*}
		\mathbb{E}\left[\sup_{0\leq r \leq t} \mathfrak{F}(r\wedge \zeta_{\widehat{R}})\|\mathcal{Y}_{n}^{\varsigma}(r\wedge \zeta_{\widehat{R}})\|_{H}^2\right] \leq C_{\mu,T,L_{\widehat{R}}} \int_{\tau_{n}}^{T+\tau_{n}}\|f(s)-f_{\infty}\|_{H}^2ds.
\end{align*}
Following the approach used for inequality \eqref{2DHD-eqsfsf5.42}, we have
	\begin{align}\label{2DHD-eq6.331}
	\sup_{\varsigma \in \mathfrak{E}}\mathbb{E}\left[\sup_{0\leq r \leq t} \|\mathcal{Y}_{n}^{\varsigma}(r\wedge \zeta_{\widehat{R}})\|_{H}^2\right] \leq C_{\mu,T,\widehat{R},L_{\widehat{R}}} \int_{-\infty}^{T+\tau_{n}}\|f(s)-f_{\infty}\|_{H}^2ds.
\end{align}
which, together with \eqref{2DHD-asy1.2} and Chebyshev's inequality, concludes \eqref{2DHD-eq6.323}. This completes the proof.
% Since the procedure is entirely analogous, we omit the details here.
\end{proof}

Since the tightness of a set in $\mathcal{P}_4(H)$ is not determined by its boundedness, it necessitates the introduction of a weaker notion of quasi-tightness, as defined in \cite{LYR-QTDS-2025}. 
\begin{definition}
 A set $\mathfrak{U}$ in $\mathcal{P}(H)$ is called \textit{quasi-tight} if, for any $\epsilon \in (0,1)$, there is a bounded subset $\mathfrak{E}_{\epsilon} \subset H$ such that $\tilde{\mu}(\mathfrak{E}_{\epsilon}) \geq 1 - \epsilon$ for all $\tilde{\mu} \in \mathfrak{U}$.
\end{definition}

Clearly, every tight set is quasi-tight. However, the two concepts coincide if and only if the state space is finite-dimensional.

\begin{lemma}\label{2DHD-lemma6.12}
	Under the hypotheses of Lemma \ref{2DHD-lemma6.11}. For any quasi-tight set $\mathfrak{U}\subset \mathcal{P}(H)$, $T>0$, and $\tau_n \rightarrow -\infty$ as $n\rightarrow -\infty$,
	$$
	\lim\limits_{n\rightarrow \infty} \sup_{\tilde{\mu}\in \mathfrak{U}}d_{\mathcal{P}(H)}\left(\Psi(T,\tau_n)\tilde{\mu},S(T)\tilde{\mu} \right)=0.
	$$
\end{lemma}
\begin{proof}
	Given $\epsilon \in (0,1)$, by the quasi-tightness of $\mathfrak{U}$, we obtain that there exists a bounded subset $\mathfrak{E}_{\epsilon}$ of $H$ such that $\sup_{\tilde{\mu}\in \mathfrak{U}}\tilde{\mu}(H\backslash\mathfrak{E}_{\epsilon})\leq \epsilon$. By Lemma \ref{2DHD-lemma6.11}, we see that there exists $N\in \mathbb{N}$ such that for all $n\geq N$,
	$$
	\sup_{n\geq N}\sup_{\varsigma\in \mathfrak{E}}\mathbb{P}\left(\left\{\omega \in \Omega: \sup_{0\leq t\leq T}\|u(t+\tau_{n},\tau_{n},\varsigma)-v(t,\varsigma)\|_{H}\geq \varkappa\right\}\right)<\frac{\epsilon}{4}.
	$$
	from which we can derive that for any $\varkappa>0$ and $n\geq N$,
	\begin{align}\label{2DHD-asy6.24}
		\begin{split}
			&~~\sup\limits_{{\phi \in L_{b}{(H)}},\,{\|\phi\|_{L_b} \leq 1}}\ \int_{\mathfrak{E}_{\epsilon}} \mathbb{E}\left[\left|\phi(u(T+\tau_n,\tau_n,\varsigma))-\phi(v(T,\varsigma))\right|\right]\tilde{\mu}(d\varsigma)\\
			&\leq \sup\limits_{{\phi \in L_{b}{(H)}},\,{\|\phi\|_{L_b} \leq 1}}\ \int_{\mathfrak{E}_{\epsilon}} \int_{\{\|u(T+\tau_n,\tau_n,\varsigma)-v(T,\varsigma)\|_H\geq \varkappa\}}\left|\phi(u(T+\tau_n,\tau_n,\varsigma))-\phi(v(T,\varsigma))\right|d\mathbb{P}\tilde{\mu}(d\varsigma)\\
			&+\sup\limits_{{\phi \in L_{b}{(H)}},\,{\|\phi\|_{L_b} \leq 1}}\ \int_{\mathfrak{E}_{\epsilon}} \int_{\{\|u(T+\tau_n,\tau_n,\varsigma)-v(T,\varsigma)\|_H< \varkappa\}}\left|\phi(u(T+\tau_n,\tau_n,\varsigma))-\phi(v(T,\varsigma))\right|d\mathbb{P}\tilde{\mu}(d\varsigma)\\
			&\leq 2\sup_{\varsigma\in \mathfrak{E}}\mathbb{P}\left(\left\{\omega \in \Omega: \|u(T+\tau_{n},\tau_{n},\varsigma)-v(T,\varsigma)\|_{H}\geq \varkappa\right\}\right)+\frac{\epsilon}{2}<\epsilon.
		\end{split}
	\end{align}

	Note that 
	$$
	\int_{H}\phi(\varsigma)\left(\Psi(t,\tau)\tilde{\mu}\right)(d\varsigma)=\int_{H} \mathbb{E}\left[\phi(u(t+\tau,\tau,\varsigma))\right]\tilde{\mu}(d\varsigma) , \quad \forall t\in \mathbb{R}^+, \phi\in C_b(H), \tilde{\mu} \in\mathcal{P}_4(H).
	$$
	Thus, for any $\tilde{\mu}\in \mathfrak{U}$ and $n\geq N$, by \eqref{2DHD-asy6.24} we have
	\begin{align*}
%	\begin{split}
		&~~d_{\mathcal{P}(H)}\left(\Psi(T,\tau_n)\tilde{\mu},S(T)\tilde{\mu} \right)\\
		&= \sup\limits_{{\phi \in L_{b}{(H)}},\,{\|\phi\|_{L_b} \leq 1}}\ \left|\int_{H} \phi(\varsigma)\left(\Psi(T,\tau_n)\tilde{\mu}\right)(d\varsigma)-\int_{H} \phi(\varsigma)\left(S(T)\tilde{\mu}\right)(d\varsigma)\right|\\
		&\leq  \sup\limits_{{\phi \in L_{b}{(H)}},\,{\|\phi\|_{L_b} \leq 1}}\ \int_{H}\left| \mathbb{E}\left[\phi(u(T+\tau_n,\tau_n,\varsigma))\right]- \mathbb{E}\left[\phi(v(T,\varsigma))\right]\right|\tilde{\mu}(d\varsigma)\\
		&\leq \sup\limits_{{\phi \in L_{b}{(H)}},\,{\|\phi\|_{L_b} \leq 1}}\ \int_{\mathfrak{E}_{\epsilon}} \mathbb{E}\left[\left|\phi(u(T+\tau_n,\tau_n,\varsigma))-\phi(v(T,\varsigma))\right|\right]\tilde{\mu}(d\varsigma)+2\epsilon<3\epsilon,
%	\end{split}	
	\end{align*}
	as desired. This completes the proof.
\end{proof}

\begin{theorem}\label{2DHD-lemma6.13}
	Under the hypotheses of Lemma \ref{2DHD-lemma6.11}. For the $\mathscr{D}$-pullback measure attractor $\mathscr{A}=\{\mathscr{A}(\tau):\tau\in \mathbb{R}\}$ of 
	system \eqref{2DHD-1.1} and the measure attractor $\mathscr{A}_{\infty}$ of system \eqref{2DHD-asy1.1}, it holds
\begin{align*}
	\lim\limits_{\tau\rightarrow -\infty}d_{\mathcal{P}_4(H)}(\mathscr{A}(\tau),\mathscr{A}_{\infty})=0.
\end{align*}
\end{theorem}
\begin{proof}
	The proof of this theorem is analogous to that of \cite[Theorem 15]{LYR-QTDS-2025}, and is omitted here for simplicity.
\end{proof}

\section*{Declarations}

\subsection*{Contributions}
The author contributed solely to this work.

\subsection*{Conflict of interest}
The authors have no conflicts to disclose.

\subsection*{Availability of date and materials}
Not applicable.

\newcounter{cankan}

\end{document}